\documentclass[11pt]{article}
\usepackage{amsfonts}
\usepackage{amsfonts,mathrsfs,amssymb,amsthm,mathptm}
\usepackage{hyperref}
\usepackage{amsmath,amscd}
\usepackage{mathptm,pslatex}
\usepackage{color}
\usepackage[all]{xy}
\usepackage{fancyhdr}
\usepackage{tikz-cd}
\usepackage[title]{appendix}
\usepackage[normalem]{ulem}

\begin{document}
\newtheorem{Theo}{Theorem}[section]
\newtheorem{Prop}[Theo]{Proposition}
\newtheorem{Lem}[Theo]{Lemma}
\newtheorem{Koro}[Theo]{Corollary}
\theoremstyle{definition}
\newtheorem{Def}[Theo]{Definition}
\newtheorem{Rem}[Theo]{Remark}
\newtheorem{Bsp}[Theo]{Example}

\newcommand{\add}{{\rm add}}
\newcommand{\Gproj}{{\rm Gproj}}
\newcommand{\con}{{\rm con}}
\newcommand{\gd}{{\rm gl.dim}}
\newcommand{\dd}{{\rm dom.dim}}
\newcommand{\cdm}{{\rm codom.dim}}
\newcommand{\tdim}{{\rm dim}}
\newcommand{\E}{{\rm E}}
\newcommand{\Mor}{{\rm Morph}}
\newcommand{\End}{{\rm End}}
\newcommand{\ind}{{\rm ind}}
\newcommand{\rsd}{{\rm res.dim}}
\newcommand{\rd} {{\rm rep.dim}}
\newcommand{\ol}{\overline}
\newcommand{\overpr}{$\hfill\square$}
\newcommand{\rad}{{\rm rad}}
\newcommand{\soc}{{\rm soc}}
\renewcommand{\top}{{\rm top}}
\newcommand{\pd}{{\rm pdim}}
\newcommand{\id}{{\rm idim}}
\newcommand{\fld}{{\rm fdim}}
\newcommand{\Fac}{{\rm Fac}}
\newcommand{\Gen}{{\rm Gen}}
\newcommand{\fd} {{\rm findim}}
\newcommand{\Fd} {{\rm Findim}}
\newcommand{\Pf}[1]{{\mathscr P}^{<\infty}(#1)}
\newcommand{\DTr}{{\rm DTr}}
\newcommand{\cpx}[1]{#1^{\bullet}}
\newcommand{\D}[1]{{\mathscr D}(#1)}
\newcommand{\Dz}[1]{{\mathscr D}^+(#1)}
\newcommand{\Df}[1]{{\mathscr D}^-(#1)}
\newcommand{\Db}[1]{{\mathscr D}^b(#1)}
\newcommand{\C}[1]{{\mathscr C}(#1)}
\newcommand{\Cz}[1]{{\mathscr C}^+(#1)}
\newcommand{\Cf}[1]{{\mathscr C}^-(#1)}
\newcommand{\Cb}[1]{{\mathscr C}^b(#1)}
\newcommand{\Dc}[1]{{\mathscr D}^c(#1)}
\newcommand{\K}[1]{{\mathscr K}(#1)}
\newcommand{\Kz}[1]{{\mathscr K}^+(#1)}
\newcommand{\Kf}[1]{{\mathscr  K}^-(#1)}
\newcommand{\Kb}[1]{{\mathscr K}^b(#1)}
\newcommand{\DF}[1]{{\mathscr D}_F(#1)}

\newcommand{\Kac}[1]{{\mathscr K}_{\rm ac}(#1)}
\newcommand{\Keac}[1]{{\mathscr K}_{\mbox{\rm e-ac}}(#1)}

\newcommand{\modcat}{\ensuremath{\mbox{{\rm -mod}}}}
\newcommand{\Modcat}{\ensuremath{\mbox{{\rm -Mod}}}}
\newcommand{\Spec}{{\rm Spec}}

\newcommand{\stmc}[1]{#1\mbox{{\rm -{\underline{mod}}}}}
\newcommand{\Stmc}[1]{#1\mbox{{\rm -{\underline{Mod}}}}}
\newcommand{\prj}[1]{#1\mbox{{\rm -proj}}}
\newcommand{\inj}[1]{#1\mbox{{\rm -inj}}}
\newcommand{\Prj}[1]{#1\mbox{{\rm -Proj}}}
\newcommand{\Inj}[1]{#1\mbox{{\rm -Inj}}}
\newcommand{\PI}[1]{#1\mbox{{\rm -Prinj}}}
\newcommand{\GP}[1]{#1\mbox{{\rm -GProj}}}
\newcommand{\GI}[1]{#1\mbox{{\rm -GInj}}}
\newcommand{\gp}[1]{#1\mbox{{\rm -Gproj}}}
\newcommand{\gi}[1]{#1\mbox{{\rm -Ginj}}}

\newcommand{\opp}{^{\rm op}}
\newcommand{\otimesL}{\otimes^{\rm\mathbb L}}
\newcommand{\rHom}{{\rm\mathbb R}{\rm Hom}\,}
\newcommand{\pdim}{\pd}
\newcommand{\Hom}{{\rm Hom}}
\newcommand{\Coker}{{\rm Coker}}
\newcommand{ \Ker  }{{\rm Ker}}
\newcommand{ \Cone }{{\rm Con}}
\newcommand{ \Img  }{{\rm Im}}
\newcommand{\Ext}{{\rm Ext}}
\newcommand{\StHom}{{\rm \underline{Hom}}}
\newcommand{\StEnd}{{\rm \underline{End}}}

\newcommand{\KK}{I\!\!K}

\newcommand{\gm}{{\rm _{\Gamma_M}}}
\newcommand{\gmr}{{\rm _{\Gamma_M^R}}}

\def\vez{\varepsilon}\def\bz{\bigoplus}  \def\sz {\oplus}
\def\epa{\xrightarrow} \def\inja{\hookrightarrow}

\newcommand{\lra}{\longrightarrow}
\newcommand{\llra}{\longleftarrow}
\newcommand{\lraf}[1]{\stackrel{#1}{\lra}}
\newcommand{\llaf}[1]{\stackrel{#1}{\llra}}
\newcommand{\ra}{\rightarrow}
\newcommand{\dk}{{\rm dim_{_{k}}}}

\newcommand{\holim}{{\rm Holim}}
\newcommand{\hocolim}{{\rm Hocolim}}
\newcommand{\colim}{{\rm colim\, }}
\newcommand{\limt}{{\rm lim\, }}
\newcommand{\Add}{{\rm Add }}
\newcommand{\Prod}{{\rm Prod }}
\newcommand{\Tor}{{\rm Tor}}
\newcommand{\Cogen}{{\rm Cogen}}
\newcommand{\Tria}{{\rm Tria}}
\newcommand{\Loc}{{\rm Loc}}
\newcommand{\Coloc}{{\rm Coloc}}
\newcommand{\tria}{{\rm tria}}
\newcommand{\Con}{{\rm Con}}
\newcommand{\Thick}{{\rm Thick}}
\newcommand{\thick}{{\rm thick}}
\newcommand{\Sum}{{\rm Sum}}

{\Large \bf
\begin{center}
Bounded $t$-structures, finitistic dimensions, and singularity categories of triangulated categories
\end{center}}

\medskip\centerline{\textbf{Rudradip Biswas}, \textbf{Hongxing Chen}$^*$, \textbf{Kabeer Manali Rahul}, \textbf{Chris J. Parker}, and \textbf{Junhua Zheng}}

\renewcommand{\thefootnote}{\alph{footnote}}
\setcounter{footnote}{-1} \footnote{ $^*$ Corresponding author.
Email: chenhx@cnu.edu.cn.}
\renewcommand{\thefootnote}{\alph{footnote}}
\setcounter{footnote}{-1} 
\renewcommand{\thefootnote}{\alph{footnote}}
\setcounter{footnote}{-1} 

\begin{abstract}
Recently, Amnon Neeman settled a bold conjecture by Antieau, Gepner, and Heller regarding the relationship between the regularity of finite-dimensional noetherian schemes and the existence of bounded $t$-structures on their derived categories of perfect complexes.

In this paper, using different methods, we prove some very general results about the existence of bounded $t$-structures on (not necessarily algebraic or topological) triangulated categories and their invariance under completion. We show that if the opposite category of an essentially small triangulated category has finite finitistic dimension in our sense, then the existence of a bounded t-structure on it forces it to be equal to its completion. We also prove a parallel result regarding the equivalence of all bounded t-structures on any intermediate triangulated category between the starting category and its completion.

Our general treatment, when specialized to the case of schemes, immediately gives us Neeman's theorem as an application and significantly generalizes another remarkable theorem by Neeman about the equivalence of bounded $t$-structures on the bounded derived categories of coherent sheaves. When specialized to other cases like associative rings, nonpositive DG-rings, connective $\mathbb{E}_1$-rings, triangulated categories without models, etc., we get many other applications. Under mild finiteness assumptions, these results not only give a categorical obstruction (the singularity category in our sense) to the existence of bounded $t$-structures on a triangulated category, but also provide plenty of triangulated categories on which all bounded $t$-structures are equivalent. The strategy used in our treatment is introducing a new concept of finitistic dimension for triangulated categories and lifting $t$-structures along completions of triangulated categories.
\end{abstract}

{\footnotesize\tableofcontents\label{contents}}

\section{Introduction}\label{Introduction}
Bounded $t$-structures of triangulated categories (see \cite{BBD}) have attracted considerable attention in many branches of mathematics including algebraic geometry, algebraic topology, representation theory, and category theory. In particular, they have been widely applied to the representation theory of finite groups of Lie type on character sheaves, motivic homotopy theory, the theory of stability conditions on triangulated categories (for example, see \cite{DL,VV,TB}), etc. A fundamental problem regarding bounded $t$-structures is \emph{how to judge whether a general triangulated category admits a bounded t-structure}. Obviously, any easily computable obstruction to the existence of bounded $t$-structures would be helpful for understanding this problem.

An outstanding development, made by Antieau, Gepner and Heller, was finding $K$-theoretic obstructions to the existence of bounded $t$-structures (see \cite{AGH}). They proved that if a small, stable $\infty$-category has a bounded $t$-structure, then its negative $K$-group in degree $-1$ vanishes, and all negative $K$-groups vanish when additionally the heart of the $t$-structure is noetherian (see \cite[Theorems 1.1 and 1.2]{AGH}). Although there exist a lot of singular schemes with vanishing negative $K$-groups, they boldly conjectured that if $X$ is a finite-dimensional noetherian scheme, then the derived category  $\mathscr{D}^{\rm perf}(X)$ of perfect complexes on $X$ has a bounded $t$-structure if and only if $X$ is regular (see \cite[Conjecture 1.5]{AGH}). Smith proved the conjecture for affine $X$ (see \cite[Theorem 1.2]{HS}). More recently, Neeman has proved a generalization of this conjecture:

\begin{Theo}{\rm\cite[Theorem 0.1]{Neeman4}}\label{GC}
Let $X$ be a finite-dimensional, noetherian scheme and let $Z\subseteq X$ be a closed subset.
Let $\mathscr{D}^{{\rm perf}}_Z(X)$ be the derived category of perfect complexes on $X$ whose cohomology is supported on $Z$. Then $\mathscr{D}^{{\rm perf}}_Z(X)$  has a bounded $t$-structure if and only if $Z$ is contained in the regular locus of $X$.
\end{Theo}

In the conclusion of Theorem \ref{GC}, one direction is clear because the condition that
$Z$ is contained in the regular locus of $X$ is equivalent to the equality $\mathscr{D}^{{\rm perf}}_Z(X)=\mathscr{D}^{b}_{{\rm coh}, Z}(X)$. Here $\mathscr{D}^{b}_{{\rm coh}, Z}(X)$ denotes the derived category of bounded complexes of sheaves with \emph{coherent} cohomology supported on $Z$, and this category always has an obvious bounded $t$-structure.

However, the other direction in the conclusion of Theorem \ref{GC}, which says that the existence of a bounded $t$-structure implies regularity, is highly nontrivial. One of the key points in Neeman's proof is the use of metric techniques in the theory of approximable triangulated categories. This language of approximability has proved to be very useful in settling several open problems and conjectures on strong generation of triangulated categories (see \cite{Neeman3}).

Another concept introduced by Neeman is the notion of a completion of a triangulated category. Roughly speaking, given a triangulated category equipped with a \textit{good metric}, there is a procedure to produce a new triangulated category, called its \emph{completion} with respect to this metric (see Section \ref{Completion-TC} for the relevant definitions and construction). Note that this completion only depends on the \textit{equivalence class} of the good metric. For example, with the hypothesis of Theorem \ref{GC}, the category $\mathscr{D}^{{\rm perf}}_Z(X)$ has an intrinsically defined equivalence class of good metrics. The completion of $\mathscr{D}^{{\rm perf}}_Z(X)$ with respect to any metric in this equivalence class, called a \textit{preferred good metric}  (see Example \ref{Induced}), is just the triangulated category $\mathscr{D}^{b}_{{\rm coh}, Z}(X)$. So, the assertion of Theorem \ref{GC} can be formulated equivalently: \emph{$\mathscr{D}^{{\rm perf}}_Z(X)$ has a bounded $t$-structure if and only if it is equal to its completion with respect to a preferred good metric.}

With similar techniques, Neeman has also proved the following theorem on the equivalence of all bounded $t$-structures on $\mathscr{D}^{b}_{{\rm coh}, Z}(X)$ in some particular cases.

\begin{Theo}{\rm\cite[9.3.1]{Neeman4}}\label{ETSCH}
Let $X$ be a finite-dimensional, noetherian scheme and let $Z\subseteq X$ be a closed subset. Then all the bounded $t$-structures on the category $\mathscr{D}^{b}_{{\rm coh}, Z}(X)$ are equivalent if any of the following conditions hold:
$(a)$ $Z$ is contained in the regular locus of $X$; $(b)$ $X$ admits a dualizing complex; $(c)$ $X$ is separated and quasiexcellent, and $Z=X$.
\end{Theo}

In the last paragraph of \cite[Section 9]{Neeman4}),  he further said that  ``\emph{presumably the equivalence of the bounded $t$-structures on $\mathscr{D}^{b}_{{\rm coh}, Z}(X)$ holds in a generality greater than we can prove now}''.
This can be related to a more general problem of classifying all triangulated categories with only one equivalence class of
bounded $t$-structures.

In this paper, we give a new obstruction to the existence of bounded $t$-structures on general triangulated categories satisfying a finiteness condition and establish related results about the equivalence of bounded $t$-structures. In particular, we show that,  for \emph{any} finite-dimensional, noetherian scheme $X$ and \emph{any} closed subset $Z$ of $X$, all bounded  $t$-structures on $\mathscr{D}^{b}_{{\rm coh}, Z}(X)$ are equivalent, achieving Neeman's expectation on a generalization of Theorem \ref{ETSCH}.

\subsection{Our main results on bounded $t$-structures and completions}\label{Section 1.1}

In this section, we state our main results, which establish a relationship between the existence of bounded $t$-structures and the invariance of triangulated categories under completion. These results can be viewed as categorical generalizations of Neeman's theorems to arbitrary triangulated categories, formulated in the language of completions. Beyond the case of noetherian schemes, our results can also be applied to many other contexts yielding a series of new corollaries.

Instead of working with preferred good metrics on a triangulated category, as Neeman does, we work with good metrics generated by objects, and complete the category with respect to such metrics. Assume that the opposite category of our triangulated category has finite finitistic dimension at a fixed object. Then we show that the existence of a bounded $t$-structure on the triangulated category implies the invariance of this category under taking the completion with respect to the good metric generated by the object. We also show that all bounded $t$-structures on any intermediate category between our triangulated category and its completion are equivalent.

To state our results more precisely, we first introduce some definitions and notation.
Throughout this section, let $\mathscr{S}$ be an essentially small triangulated category with the shift functor denoted by $[1]$.

Let $G$ be an object of $\mathscr{S}$. For each integer $n$, we define
$G(-\infty, n]:=\{G[-i]\mid i\leqslant n\}$ and denote by $\langle G\rangle ^{(-\infty, n]}$ (resp., $\langle G\rangle ^{[n,\infty)}$) the smallest full subcategory of $\mathscr{S}$ containing $G[-n]$ and closed under extensions, direct summands and positive (resp., negative) shifts. We refer to Definition \ref{coprod} for the construction of objects of these categories by iteration. Let $\langle G\rangle:=\bigcup_{n\in\mathbb{N}}\langle G\rangle ^{(-\infty, n]}$
be the \emph{thick subcategory of $\mathscr{S}$ generated by $G$}. If $\mathscr{S}=\langle G\rangle$, then $G$ is called a \emph{classical generator} of $\mathscr{S}$. For $\mathscr{X}\subseteq\mathscr{S}$ a full subcategory, we denote by $\mathscr{X}^{\bot}$ the full subcategory of $\mathscr{S}$ consisting of all objects $Y$ with $\Hom_\mathscr{S}(X,Y)=0$ for $X\in\mathscr{X}$.

In the representation theory of algebras, finitistic dimension and related concepts are ubiquitous, and it is from there that we derive our inspiration. Preceding our work, Krause has introduced an important notion of finitistic dimension for triangulated categories in terms of the generativity of objects (see \cite{Kr2} or Definition \ref{K-FD}), which supplies the first way to characterize the finiteness of the (small) finitistic dimension of a ring as a property of the derived category of perfect complexes. However, this notion of finitistic dimension is not connected/related to the good metrics, with respect to which we will complete a triangulated category to obtain our results, and is thus not suitable for our general treatment. For our aim, we introduce the following notion of finitistic dimension for triangulated categories which is new and of independent interest.

\begin{Def}\label{Def-FD}
The \emph{finitistic dimension} of $\mathscr{S}$ at an object $G$ is defined to be
$$\fd(\mathscr{S}, G):=\inf\big\{n\in\mathbb{N}\mid G(-\infty, -1]^{\bot}\subseteq\langle G\rangle^{[0,\infty)}[n]\big\}.$$
We say that $\mathscr{S}$ has \emph{finite finitistic dimension} if there is an object $G$
with $\fd(\mathscr{S}, G)<\infty$.
\end{Def}

The reason why we still use the terminology ``finitistic dimension" for Definition \ref{Def-FD} is that
the finitistic dimension of the derived category of perfect complexes over a ring at the regular module
is exactly the (small) finitistic dimension of the ring (see Lemma \ref{Finite}(5)). For more elementary properties of our definition of finitistic dimension, we refer to Lemma \ref{Finite}. In particular, in the presence of a bounded t-structure on a triangulated category, the only objects at which the category can possibly have finite finitistic dimension are classical generators. Moreover, the finiteness of finitistic dimension (at classical generators) will be shown for several classes of triangulated categories, for example:

$(i)$ a triangulated category with an algebraic $t$-structure or with a strong generator, including
the bounded derived category of certain abelian categories (e.g. the category of finitely generated modules over an Artin algebra or of coherent sheaves over a quasiexcellent scheme or a separated scheme of finite type over a field),

$(ii)$ the singularity category of a Gorenstein Artin algebra or of a self-injective differential graded (DG) algebra over a field,

$(iii)$ the derived category $\mathscr{D}^{{\rm perf}}_Z(X)$ of perfect complexes on a quasicompact, quasiseparated  scheme $X$ with cohomology supported on a closed subset $Z$, where $X$ has finite finitistic dimension (for example, $X$ is finite-dimensional and noetherian, see Definition \ref{SMFFD}) and the complement $X-Z$ is quasicompact,

$(iv)$ the perfect derived category of a DG algebra with some conditions on its cohomology.

\smallskip
Further, some other types of finitistic dimension in Appendix \ref{Section B}, such as the finitistic dimension of nonpositive DG rings introduced in [9], provide upper bounds for the finitistic dimension of particular triangulated categories (see Proposition \ref{Nonpositive} and Corollary B.5). Most importantly, the assumption that the finitistic dimension of certain relevant triangulated categories is finite allows us to implement our methods on bounded $t$-structures.

As we have said above, we are particularly interested in good metrics determined by objects. Recall that a \emph{good metric} on a triangulated category consists of countably many descending subcategories of the category which are closed under extensions and certain shifts (see Definition \ref{definition of metric} for details).

\begin{Def}\label{PEC}
A good metric $\mathscr{M}:=\{\mathscr{M}_n\}_{n\in\mathbb{N}}$ on $\mathscr{S}$ is called
a \emph{$G$-good metric} if it is equivalent to the good metric $\{\langle G\rangle^{(-\infty,-n]}\}_{n\in\mathbb{N}}$ on $\mathscr{S}$ generated by $G$, that is, for each
$n$, there exist nonnegative integers $a_n$ and $b_n$ (depending on $n$)
with $\mathscr{M}_{a_n}\subseteq \langle G\rangle^{(-\infty,-n]}$ and
$\langle G\rangle^{(-\infty,-{b_n}]}\subseteq \mathscr{M}_n$.

The \emph{completion} of $\mathscr{S}$ at $G$, denoted by $\mathfrak{S}_G(\mathscr{S})$, is defined to be the completion of $\mathscr{S}$  with respect to any $G$-good metric on $\mathscr{S}$.
\end{Def}

We denote by $\mathscr{S}\opp$ the \emph{opposite category} of $\mathscr{S}$ and by $\mathscr{S}\Modcat$ the abelian category of additive functors from $\mathscr{S}\opp$ to the category of abelian groups. For simplicity, $\mathscr{S}$ is identified with its essential image in
$\mathscr{S}\Modcat$ under the Yoneda functor. By definition, $\mathfrak{S}_G(\mathscr{S})\subseteq \mathscr{S}\Modcat$ consists of the colimits of Cauchy sequences in $\mathscr{S}$ (with respect to the metric $\mathscr{M}$) which vanish at $\mathscr{M}_n$ for some $n$. In general, $\mathscr{S}$ is \emph{not} a full subcategory of $\mathfrak{S}_G(\mathscr{S})$. Note that equivalent good metrics produce the same completion, and different classical generators of a triangulated category yield equivalent good metrics. Thus $\mathfrak{S}_H(\mathscr{S})=\mathfrak{S}_G(\mathscr{S})$ whenever $\langle H\rangle=\langle G\rangle\subseteq\mathscr{S}$. In particular, if $\mathscr{S}$ has a classical generator, then $\mathfrak{S}_G(\mathscr{S})$ is independent of the choices of classical generators $G$ of $\mathscr{S}$.

The main result of the paper reads as follows.

\begin{Theo}\label{Main result}
Let $\mathscr{S}$ be an essentially small triangulated category with an object $G$.
Suppose that the finitistic dimension of $\mathscr{S}\opp$ at $G$ is finite.
Then the following statements are true.

$(a)$ If $\mathscr{S}$ has a bounded $t$-structure, then $\mathscr{S}=\mathfrak{S}_G(\mathscr{S})$.

$(b)$ If $\mathscr{X}$ is a full triangulated subcategory of $\mathfrak{S}_G(\mathscr{S})$
with $\mathscr{S}\subseteq \mathscr{X}$, then all bounded $t$-structures on
$\mathscr{X}$ are equivalent. In particular, if $\mathscr{S}\subseteq\mathfrak{S}_G(\mathscr{S})$, then
all bounded $t$-structures on $\mathfrak{S}_G(\mathscr{S})$ are equivalent.
\end{Theo}

Theorem \ref{Main result}$(a)$ implies that if there is an object $G\in\mathscr{S}$ with $\fd(\mathscr{S}\opp, G\opp)<\infty$ and $\mathscr{S}\neq \mathfrak{S}_G(\mathscr{S})$, then $\mathscr{S}$ has no bounded $t$-structure. In light of this observation and some known examples of completions, one may reasonably define the \emph{almost singularity category} of $\mathscr{S}$ at the object $G$ as
$$\mathscr{S}_{\rm{sg}}(G):=\mathfrak{S}_G(\mathscr{S})/(\mathscr{S}\cap \mathfrak{S}_G(\mathscr{S})),$$
which is the Verdier quotient of $\mathfrak{S}_G(\mathscr{S})$ by the triangulated subcategory $\mathscr{S}\cap \mathfrak{S}_G(\mathscr{S})$. Further, we call it the \emph{singularity category} of $\mathscr{S}$ at $G$ in the case that $\mathscr{S}\subseteq \mathfrak{S}_G(\mathscr{S})$.  This always happens, for example, if $\mathscr{S}$ has a bounded $t$-structure (see Theorem \ref{BABB} (2)).  We might say that $\mathscr{S}$ is \emph{almost regular} at $G$ if $\mathscr{S}_{\rm{sg}}(G)=0$; \emph{regular} at $G$ if $\mathscr{S}=\mathfrak{S}_G(\mathscr{S})$. Then $\mathscr{S}$ itself is said to be \emph{(almost) regular} if it is (almost) regular at one (and thus also all) of its classical generators. Thus Theorem \ref{Main result}$(a)$ can be roughly phrased as, \emph{the existence of a bounded t-structure on a triangulated category implies the regularity of the category at objects}. In other words, under a finiteness assumption on the finitistic dimension, \textit{the singularity category is an obstruction to the existence of bounded $t$-structures.}

This concept of singularity category for triangulated categories encompasses both the algebraic and geometric notions of singularity categories, and is also of independent interest. We refer to Definition \ref{Singularity} and Example \ref{Ring and Scheme} for more details. In Appendix \ref{ECTC}, we also compare the notion of almost regularity to that of an \textit{almost regular} $\mathbb{E}_1$-ring $R$, and show that these notions agree when $\mathscr{S}$ is the homotopy category of the stable $\infty$-category  of perfect left $R$-module spectra (see Corollary \ref{Spectrum}).

Our strategy for proving Theorem \ref{Main result} is lifting (not necessarily bounded) $t$-structures from a general triangulated category to its completion with respect to any good metric. This is different from Neeman's strategy of the proof of Theorem \ref{GC}, which relies on lifting any bounded $t$-structure on $\mathscr{D}^{{\rm perf}}_Z(X)$ to the standard $t$-structure on the unbounded derived category $\mathscr{D}_{\rm qc, Z}(X)$  up to equivalence (see Theorem \ref{CCST} and \cite[Lemma 6.1]{Neeman4}). We are highlighting a part of our lifting result below.

\begin{Theo} (part of Theorem \ref{ET})\label{lift-introduction}
Let $\mathscr{S}$ be an essentially small triangulated category with a good metric $\mathscr{M}$, and let $(\mathscr{S}^{\leqslant 0},\mathscr{S}^{\geqslant 0})$ be a $t$-structure on $\mathscr{S}$ that is ``extendable" with respect to $\mathscr{M}$ (see Definition \ref{ETSC}). Then $\big(\mathfrak{S}(\mathscr{S}^{\leqslant 0}), \mathfrak{S}(\mathscr{S}^{\geqslant 0})\big)$ is a $t$-structure on $\mathfrak{S}(\mathscr{S})$ with its heart and its coaisle equivalent to the heart and the coaisle of $(\mathscr{S}^{\leqslant 0}, \mathscr{S}^{\geqslant 0})$, respectively. Here, for a full subcategory $\mathscr{A}$ of $\mathscr{S}$, the category $\mathfrak{S}(\mathscr{A})$ can be thought of as the completion of $\mathscr{S}$ with respect to $\mathscr{M}$ and $\mathscr{A}$ (see Definition \ref{Completion}).

\end{Theo}

Theorem \ref{lift-introduction}, and the more elaborate Theorem \ref{ET}, answer a question by Neeman proposed in his \emph{ICM 2022 Proceedings} paper --``\emph{Are there similar theorems about $t$-structures in $\mathscr{S}$ going to $t$-structures in $\mathfrak{S}(\mathscr{S})$?}'' (see \cite[p.1653]{Neeman6}). We also point out that any bounded above $t$-structure is extendable with respect to a metric defined by an object (see Lemma \ref{tool}(3)), which lays the foundation for lifting bounded t-structures.

\subsection{Consequences of our main result}\label{Section 1.2}

The calculation of completions of triangulated categories is very crucial if we are to use Theorem \ref{Main result}$(a)$ to find a potential obstruction to the existence of bounded $t$-structures. Based on Neeman's work on good extensions of triangulated categories (see Theorem \ref{MCP}), we can apply Theorem \ref{Main result} to the category of compact objects of a compactly generated triangulated category.

Let $\mathscr{T}$ be a compactly generated triangulated category which has a single compact generator $G$. Let $\mathscr{T}^{c}\subseteq\mathscr{T}$ consist of all compact objects.  Given a $t$-structure $(\mathscr{T}^{\leqslant 0}, \mathscr{T}^{\geqslant 0})$ on $\mathscr{T}$ in the preferred equivalence class (Definition \ref{prefer}), there are two intrinsic categories $\mathscr{T}_c^{-}$ and $\mathscr{T}_c^{b}$ defined in Definition \ref{TC}, regarded as the \emph{closure} and \emph{bounded closure} of $\mathscr{T}^c$ in $\mathscr{T}$, respectively. They are thick, triangulated subcategories of $\mathscr{T}$ if $\Hom_\mathscr{T}(G, G[i])=0$ for $i\gg 0$, which often holds in practical applications. When the above $t$-structure restricts to a $t$-structure $\big(\mathscr{T}^{\leqslant 0}\cap\mathscr{T}_c^{-}, \mathscr{T}^{\geqslant 0}\cap\mathscr{T}_c^{-}\big)$ on
$\mathscr{T}_c^{-}$ (see \cite[Definition 5.1]{Neeman1}), the category
$\mathscr{T}_c^{b}$ has a bounded $t$-structure $\big(\mathscr{T}^{\leqslant 0}\cap\mathscr{T}_c^{b}, \mathscr{T}^{\geqslant 0}\cap\mathscr{T}_c^{b}\big)$. In this case, an easy conclusion is that if $\mathscr{T}^{c}=\mathscr{T}_c^b$, then $\mathscr{T}^{c}$ has a bounded $t$-structure. Our first corollary provides a converse of this conclusion.

\begin{Koro}\label{CPG}
Suppose that $\Hom_\mathscr{T}(G, G[i])=0$ for $i\gg 0$ and the category $(\mathscr{T}^{c})\opp$ has finite finitistic dimension.

$(1)$ If $\mathscr{T}^{c}$ has a bounded $t$-structure, then $\mathscr{T}^{c}=\mathscr{T}_c^b$.

$(2)$ If $\mathscr{X}$ is a full triangulated subcategory of $\mathscr{T}$ with $\mathscr{T}^{c}\subseteq \mathscr{X}\subseteq \mathscr{T}_c^b$, then all bounded $t$-structures on $\mathscr{X}$ are equivalent.
\end{Koro}

Next, we illustrate two special cases of Corollary \ref{CPG} related to schemes and ordinary rings.

Let $X$ be a quasicompact, quasiseparated scheme (for example, an affine scheme or a noetherian scheme).
Recall that a (cochain) complex $\cpx{M}$ of $\mathscr{O}_X$-modules with quasicoherent cohomology is said to be \emph{pseudocoherent} if, for any open immersion $i: U\to X$ with $U$ an affine open subset of $X$, the restriction complex $i^*(\cpx{M})$  of $\cpx{M}$ to $U$ has bounded above resolutions by finite-rank vector bundles, or equivalently, identifying $U$ with the spectrum of a commutative ring $A$, the complex $i^*(M)$ is isomorphic to a bounded above complex of finitely generated projective $A$-modules. Clearly, pseudocoherence is a local property. Further, let $Z$ be a closed subset of $X$ such that $X-Z$ is quasicompact. We denote by $\mathscr{D}^{p, b}_{{\rm qc}, Z}(X)$ the derived category of pseudocoherent complexes on $X$ with cohomology supported on $Z$ and with bounded cohomology. This category contains $\mathscr{D}^{{\rm perf}}_Z(X)$, and equals $\mathscr{D}^{b}_{{\rm coh}, Z}(X)$ for a noetherian scheme $X$.

Let $R$ be an associative ring with identity. We denote by $R\opp$ the opposite ring of $R$, and by $\prj{R}$ and $R\modcat$ the categories of finitely generated projective  and finitely presented left $R$-modules, respectively. As usual, $\mathscr{K}$, $\mathscr{D}$ and $b$ stand for \emph{homotopy category}, \emph{derived category} and \emph{bounded cohomology}, respectively. For instance, $\mathscr{K}^{-, b}(\prj{R})$ is the homotopy category of bounded above complexes of finitely generated projective left $R$-modules with bounded cohomolgy. We also denote by $\fd(R)$ the \emph{finitistic dimension} of $R$, which is by definition the supremum of projective dimensions of those left $R$-modules having a finite projective resolution by finitely generated projective $R$-modules (for example, see \cite{BASS, XC, Kr2}). When the supremum is taken over the projective dimensions of all left $R$-modules with finite projective dimension, the \emph{big finitistic dimension} $\Fd(R)$ of $R$ is defined. For a commutative noetherian ring $R$, it is known that $\fd(R)\leqslant \Fd(R)=\dim(R)$, the Krull dimension of $R$ (see \cite{RG}).

Corollary \ref{Scheme case}(1) generalizes Theorem \ref{GC} beyond noetherian schemes, and Corollary \ref{Scheme case}(2) extends Theorem \ref{ETSCH} in full generality. \emph{Our assumptions on the scheme $X$ and the finiteness of the finitistic dimension in Corollary \ref{Scheme case} is weaker than the scheme being noetherian and finite-dimensional}. Moreover, our proof of Corollary \ref{Scheme case} is independent of the weak approximability of $\mathscr{D}_{{\rm qc}, Z}(X)$ that was shown in \cite{Neeman4}.

\begin{Koro}\label{Scheme case}
Let $X$ be a quasicompact, quasiseparated scheme and let $Z$ be a closed subset of $X$ such that $X-Z$ is quasicompact. Suppose that $X$ has a finite affine open covering $X=\bigcup_{i=1}^{n}V_i$, where $V_i$ is isomorphic to the spectrum of $R_i$ for some commutative ring $R_i$ with $\fd(R_i)<\infty$ for each $i$.
Then:

$(1)$ If $\mathscr{D}^{{\rm perf}}_Z(X)$  has a bounded $t$-structure, then
$\mathscr{D}^{{\rm perf}}_Z(X)=\mathscr{D}^{p, b}_{{\rm qc}, Z}(X)$. In particular, if all $R_i$ are noetherian rings, then $\mathscr{D}^{{\rm perf}}_Z(X)$ has a bounded $t$-structure
if and only if $Z$ is contained in the regular locus of $X$.

$(2)$ All bounded $t$-structures on any triangulated category between $\mathscr{D}^{{\rm perf}}_Z(X)$ and $\mathscr{D}^{p, b}_{{\rm qc}, Z}(X)$ are equivalent. In particular, if all $R_i$ are noetherian rings, then
all bounded $t$-structures on $\mathscr{D}^{b}_{{\rm coh}, Z}(X)$ are equivalent.
\end{Koro}

Corollary \ref{Ring case}(1) generalizes \cite[Theorem 1.2]{HS} which deals with commutative noetherian rings of finite Krull dimension. Further generalizations of left coherent rings to left coherent $\mathbb{E}_1$-rings and nonpositive DG-rings are given in Corollaries \ref{LCR} and
\ref{b6-dg}.

\begin{Koro}\label{Ring case}
Let $R$ be a ring with identity. Suppose $\fd(R\opp)<\infty$. Then:

$(1)$ If $\Kb{\prj{R}}$ has a bounded $t$-structure, then $\Kb{\prj{R}}=\mathscr{K}^{-, b}(\prj{R})$. In particular, if $R$ is left coherent, then $\Kb{\prj{R}}$ has a bounded $t$-structure if and only if $\Kb{\prj{R}}=\Db{R\modcat}$.

$(2)$ All bounded $t$-structures on any triangulated category between $\Kb{\prj{R}}$ and $\mathscr{K}^{-, b}(\prj{R})$ are equivalent.
\end{Koro}

\begin{Rem}\label{artin-corollary}

For an Artin algebra $R$, Corollary \ref{Ring case}(1) has an easy application: if $\fd(R^{\opp})<\infty$, then $\Kb{\prj{R}}$ has a bounded $t$-structure if and only if $R$ has finite global dimension. It is already known that all bounded $t$-structures on $\Db{R\modcat}$ are equivalent without any finiteness assumptions on the finitistic dimension of $R$ or $R^{\opp}$ (for example, see \cite[Lemma 3.22]{AMY}). However, by Corollary \ref{Ring case}(2), all bounded $t$-structures on any triangulated category between $\Kb{\prj{R}}$ and $\Db{R\modcat}$ are equivalent provided $\fd(R^{\opp})<\infty$. There are plenty of intermediate categories, for example, the bounded derived categories of \emph{resolving} subcategories (in the sense of Auslander and Reiten) of $R\modcat$ as exact categories. An example of a resolving subcategory is the category of finitely generated Gorenstein-projective left $R$-modules.

Corollary \ref{Ring case}(2) (and also Corollary \ref{Scheme case}(2)) does \emph{not} hold without the finiteness assumption on the finitistic dimension. Let $R:=k[x_1,x_2,x_3,\cdots]$ be the polynomial ring in countably many variables over a field $k$. We show that $\Kb{\prj{R}}$ admits two bounded $t$-structures that are \emph{not} equivalent (see Proposition \ref{CE}). In this example, $R$ is commutative, coherent and regular, but has infinite finitistic dimension.
\end{Rem}

Motivated by the main results in the paper, we propose the following open questions.

\smallskip
{\bf Question 1.} For any essentially small triangulated category $\mathscr{S}$ with a classical generator $G$, does the existence of a bounded $t$-structure on $\mathscr{S}$ imply that the singularity category of $\mathscr{S}$ at $G$ is trivial?

{\bf Question 2.} Let $R$ be a left coherent ring. If $\Kb{\prj{R}}$ has a bounded $t$-structure, is $R$ left regular?

{\bf Question 3.} When does the homotopy category of perfect modules over a left coherent $\mathbb{E}_1$-ring have finite finitistic dimension?

\smallskip
Question 2 is a special case of Question 1. In the case of Artin algebras, Question 2 has a positive answer if the well-known \emph{finitistic dimension conjecture} (that is, all Artin algebras have finite finitistic dimension) is true, due to Corollary \ref{Ring case} and Remark \ref{artin-corollary}.
The discussion relevant to Question 3 can be found in Appendix \ref{ECTC}.

\subsection{Outline of the contents}

The contents of this article are sketched as follows. In Section \ref{sect2}, we fix some notation, recall the definitions of $t$-structures and completions of triangulated categories, and recall some central theorems on completions. In Section \ref{example section}, we introduce the notion of an extendable $t$-structure on a triangulated category, and show that such a $t$-structure can be lifted to a $t$-structure on the completion of the triangulated category (see Theorem \ref{ET}). In Section \ref{CMAO}, we provide necessary conditions for the existence of bounded $t$-structures on triangulated categories in terms of their completions. The main result is Theorem \ref{BABB} which implies Theorem \ref{Main result}. We also discuss the lifting of bounded $t$-structures on the completion of a triangulated category to $t$-structures on a bigger triangulated category that has the given category as the full subcategory of compact objects (see Theorem \ref{Equivalent $t$-structure}). In Section \ref{FDTC}, we introduce the notion of finitistic dimension for triangulated categories at a fixed object of the category, and we show how this notion is very well behaved when the category admits a classical generator. We also prove the finiteness of finitistic dimension for several classes of triangulated categories. Finally, we give proofs of all corollaries mentioned in the Introduction.

There are two appendices. In Appendix \ref{ECTC}, we calculate the completion of the homotopy category of perfect complexes over a connective ring spectrum (see Theorem  \ref{Connective case}) and establish a connection between the existence of bounded $t$-structures on the homotopy category and the regularity of the ring spectrum (see Corollary \ref{LCR}). Note that Krause has introduced another notion of the completion of triangulated categories (see \cite{Kr3}). Incredibly, both kinds of completions produce the same triangulated categories in several typical cases (see Remark \ref{Known cases}(2)).

In Appendix \ref{Section B}, we discuss some other ways of defining finitistic dimension for particular triangulated categories that exist in the literature. This enables us to  bound the finitistic dimension of the perfect derived category of a nonpositive DG ring by its finitistic dimension (see Corollary \ref{nonpositive DG ring}).

\section{Preliminaries\label{sect2}}

In this section we briefly recall some notation, definitions, and basic facts used in this paper.

\subsection{General notation and facts}
Throughout the paper, let $\mathscr{T}$ be a triangulated category with the shift functor denoted by $[1]$.
The \emph{extension} of full subcategories $\mathscr{X}$ and $\mathscr{Y}$ of $\mathscr{T}$, denoted by $\mathscr{X} \ast \mathscr{Y}$, is by definition the full subcategory of $\mathscr{T}$ consisting of
objects $Z$ such that there exists a (distinguished) triangle $X\rightarrow Z \rightarrow Y \rightarrow X[1]$ in $\mathscr{T}$ with $X \in \mathscr{X}$ and $Y \in \mathscr{Y}$.
If $\mathscr{X} \ast \mathscr{X}\subseteq \mathscr{X}$, then $\mathscr{X}$ is said to be \emph{closed under extensions} in $\mathscr{T}$.  By $\Hom_{\mathscr{T}}(\mathscr{X},\mathscr{Y})=0$, we mean that $\Hom_{\mathscr{T}}(X,Y)=0$ for any $X\in \mathscr{X}$ and $Y\in \mathscr{Y}$. Further, we define
$\mathscr{X}^{\perp}:=\{Y\in\mathscr{T}\;|\; \Hom_\mathscr{T}(\mathscr{X}, Y)=0\}$ and $ ^{\perp}\mathscr{X}:=\{Y\in\mathscr{T}\;|\; \Hom_\mathscr{T}(Y, \mathscr{X})=0\}.$
For a morphism $f:X\to Y$ in $\mathscr{T}$, the third term $Z$ in a triangle  $X\lraf{f} Y\to Z\to X[1]$ in $\mathscr{T}$ is called the \emph{cone} of $f$ and denoted by ${\rm Cone}(f)$.
Note that ${\rm Cone}(f)$ is unique up to isomorphism. For a triangle functor $F:\mathscr{T}\to\mathscr{T}'$ of triangulated categories, we denote by $F(\mathscr{T})$ the \emph{essential image} of $\mathscr{T}$ under $F$, that is, $F(\mathscr{T}):=\{X'\in\mathscr{T}'\mid X'\simeq F(X),\; X\in\mathscr{T}\}$. The opposite category of $\mathscr{T}$ is denoted by $\mathscr{T}\opp$.

Suppose that $\mathscr{T}$ has (small) coproducts, that is, coproducts indexed over sets exist in $\mathscr{T}$. An object $X\in\mathscr{T}$ is said to be \emph{compact} if the functor $\Hom_{\mathscr{T}}(X,-)$ from $\mathscr{T}$ to the category of abelian groups commutes with coproducts. We denote by $\mathscr{T}^c$ the full subcategory of $\mathscr{T}$ consisting of all compact objects.
This is a full triangulated subcategory of $\mathscr{T}$ closed under direct summands. A \emph{chain} $\{X_{\bullet},f_{\bullet}\}$ in $\mathscr{T}$ consists of a countable collection of objects $\{{X_{n}}\}_{n \in \mathbb{N}}$ together with morphisms $\{f_{n+1}:X_{n} \rightarrow X_{n+1}\}_{n\in \mathbb{N}}$. The \emph{homotopy colimit} of a chain $\{X_{\bullet},f_{\bullet}\}$, denoted by $\mathop{\text{Hocolim}}\limits_{\longrightarrow}X_{n}$, is defined to be the cone of
the map $1-f_*: \mathop{\bigoplus}\limits_{n=0}^{\infty} X_{n} \rightarrow \mathop{\bigoplus}\limits_{n=0}^{\infty} X_{n}$, where $f_*$ stands for the direct sum of $f_{n+1}:X_n\rightarrow X_{n+1}$ for all $n$.

\begin{Def} \label{coprod}
We recall some notation and definitions from {\rm \cite[Reminders 1.1 and 0.1; Definition 1.3]{Neeman3}} (see also {\rm \cite[2.2]{BVB} and \cite{RR}}). If infinite coproducts of objects are involved in the following definitions of categories, $\mathscr{T}$ is assumed to have coproducts.

Let $\mathscr{A}$ be a class of objects in $\mathscr{T}$ and $G$ an object of $\mathscr{T}$.

$(1)$ $\text{smd}(\mathscr{A})$ (resp., $\text{add}(\mathscr{A})$, $\text{Add}(\mathscr{A})$) denotes the full subcategory of $\mathscr{T}$ consisting of all direct summands (resp., finite direct sums, coproducts) of objects in $\mathscr{A}$.

$(2)$ For $n>0$, the subcategories $\text{coprod}_{n}(\mathscr{A})$ and $\text{Coprod}_{n}(\mathscr{A})$ of $\mathscr{T}$ are defined inductively by
\[\text{coprod}_{1}(\mathscr{A}):=\text{add}(\mathscr{A}),\;\;
\text{coprod}_{n+1}(\mathscr{A}):=\text{coprod}_{1}(\mathscr{A}) \ast \text{coprod}_{n}(\mathscr{A}),\]
\[\text{Coprod}_{1}(\mathscr{A}):=\text{Add}(\mathscr{A}),\;\;
\text{Coprod}_{n+1}(\mathscr{A}):=\text{Coprod}_{1}(\mathscr{A}) \ast \text{Coprod}_{n}(\mathscr{A}).\]
Moreover, let $\text{coprod}(\mathscr{A}):=\bigcup_{n>0} \text{coprod}_{n}(\mathscr{A}).$

$(3)$ $\text{Coprod}(\mathscr{A})$ denotes the smallest full subcategory of $\mathscr{T}$ containing $\mathscr{A}$ and closed under coproducts and extensions. Clearly, if  $\mathscr{A}[1]\subseteq\mathscr{A}$ or $\mathscr{A}\subseteq\mathscr{A}[1]$, then $\text{Coprod}(\mathscr{A})$ is closed under \emph{direct summands} in $\mathscr{T}$ by the Eilenberg swindle argument.

$(4)$ For two integers $A\leqslant B$,
let $G[A,B]:=\{G[-i]\mid i\in\mathbb{Z},\; A\leqslant i \leqslant B\}.$
We also allow $A$ and $B$ to be infinite, for example,
$G(-\infty,B]:=\{G[-i]\mid i\in\mathbb{Z}, \; i \leqslant B\}$.
For $n>0$, let
$$
{\langle G \rangle}^{[A,B]}_n:=\text{smd}(\text{coprod}_{n}(G[A,B])),\;\; {\langle G \rangle}^{[A,B]}:=\bigcup_{n>0} {\langle G \rangle}^{[A,B]}_{n},\;\;{\langle G \rangle}_{n}:={\langle G \rangle}^{(-\infty, \infty)}_n,\;\;
\langle G \rangle:= \bigcup_{n>0}{\langle G \rangle}_n.
$$
This means that $\langle G \rangle$ is the smallest full triangulated subcategory of $\mathscr{T}$ containing $G$ and closed under direct summands.

$(5)$ Let $A\leqslant B$ be integers (possibly infinite). We define
$\overline{\langle G \rangle}^{[A,B]}_n:=\text{smd}(\text{Coprod}_n(G[A,B]))$ for $n>0$ and
$\overline{\langle G \rangle}^{[A,B]}:=\text{smd}(\text{Coprod}(G[A,B]))$.
For simplicity, we write $\overline{\langle G \rangle}$ for $\overline{\langle G \rangle}^{(-\infty, \infty)}$ that is the smallest full triangulated subcategory of $\mathscr{T}$ containing $G$ and closed under coproducts.

$(6)$ The object $G$ of $\mathscr{T}$ is called a \emph{classical generator} if $\mathscr{T}=\langle G \rangle$; a \emph{strong generator} if there exists a nonnegative integer $n$ such that $\mathscr{T}={\langle G \rangle}_{n+1}$; and a \emph{compact generator} if $G$ is compact in $\mathscr{T}$ and $\mathscr{T}=\overline{\langle G \rangle}$.

\end{Def}

The following result is immediate from {\rm \cite[Proposition 1.9]{Neeman3}}.

\begin{Lem}\label{Smd}
Let $\mathscr{T}$ be a triangulated category with coproducts and $\mathscr{A}$ a full subcategory of $\mathscr{T}^c$. Then  $\mathscr{T}^c\cap {\rm Coprod}(\mathscr{A})\subseteq {\rm smd(coprod}(\mathscr{A}))$. In particular, if $\mathscr{T}$ is compactly generated and has a compact generator $G$, then $G$ is a classical generator of $\mathscr{T}^c$.
\end{Lem}

\subsection{$t$-structures on triangulated categories}\label{$T$-structure}

In this section, we recall the definition of $t$-structures on triangulated categories, as well as a method for constructing $t$-structures starting from a collection of compact objects.

\begin{Def}{\cite[Definition 1.3.1]{BBD}} \label{DTS}
Let $\mathscr{T}$ be a triangulated category. A pair of full subcategories $(\mathscr{T}^{\leqslant 0},\mathscr{T}^{\geqslant 0})$ in $\mathscr{T}$ is called a \emph{$t$-structure} on $\mathscr{T}$ if the following conditions are satisfied:

(T1) $\mathscr{T}^{\leqslant -1} \subseteq \mathscr{T}^{\leqslant 0}$ and $\mathscr{T}^{\geqslant 0} \subseteq  \mathscr{T}^{\geqslant -1}$, where $\mathscr{T}^{\leqslant n}:=\mathscr{T}^{\leqslant 0}[-n]$ and $\mathscr{T}^{\geqslant n}:=\mathscr{T}^{\geqslant 0}[-n]$ for $n\in\mathbb{Z}$;

(T2) $\Hom_{\mathscr{T}}(\mathscr{T}^{\leqslant -1},\mathscr{T}^{\geqslant 0})=0$;

(T3) for each object $X \in \mathscr{T}$, there exists a triangle $X^{\leqslant -1} \rightarrow X \rightarrow X^{\geqslant 0} \rightarrow X^{\leqslant -1}[1]$ in $\mathscr{T}$ with $X^{\leqslant -1}\in \mathscr{T}^{\leqslant -1}$ and $X^{\geqslant 0}\in \mathscr{T}^{\geqslant 0}$; in other words, $\mathscr{T}=\mathscr{T}^{\leqslant -1}\ast\mathscr{T}^{\geqslant 0}$.

The categories $\mathscr{T}^{\leqslant 0}$ and $\mathscr{T}^{\geqslant 0}$ are called the \emph{aisle} and the \emph{coaisle} of the $t$-structure, respectively.
\end{Def}

Let $(\mathscr{T}^{\leqslant 0},\mathscr{T}^{\geqslant 0})$ be a $t$-structure on $\mathscr{T}$.
Then $(\mathscr{T}^{\leqslant n})^{\perp}=\mathscr{T}^{\geqslant n+1}$ and
$^{\perp}(\mathscr{T}^{\geqslant n+1})=\mathscr{T}^{\leqslant n}$. Up to isomorphism, there exists a unique triangle
$X^{\leqslant n-1}\to X \to X^{\geqslant n} \to X^{\leqslant n-1}[1]$ in $\mathscr{T}$ with $X^{\leqslant n-1}\in \mathscr{T}^{\leqslant n-1}$ and $X^{\geqslant n}\in \mathscr{T}^{\geqslant n}$. Moreover, the category $\mathscr{H}:=\mathscr{T}^{\leqslant 0} \cap \mathscr{T}^{\geqslant 0}$ is an abelian category and called the \emph{heart} of $(\mathscr{T}^{\leqslant 0},\mathscr{T}^{\geqslant 0})$ (see  \cite[Theorem 1.3.6]{BBD}). Further, let
\[\mathscr{T}^{-}:=\mathop{\bigcup}\limits_{n\in \mathbb{N}} \mathscr{T}^{\leqslant n},\;\; \mathscr{T}^{+}:=\mathop{\bigcup}\limits_{n\in \mathbb{N}} \mathscr{T}^{\geqslant -n}\;\; \mbox{and}\;\;\mathscr{T}^{b}:=\mathscr{T}^{-}\cap \mathscr{T}^{+} .\]
We say that $(\mathscr{T}^{\leqslant 0},\mathscr{T}^{\geqslant 0})$ is \emph{bounded above}, \emph{bounded below} and \emph{bounded} if $\mathscr{T}=\mathscr{T}^{-}, \mathscr{T}^{+}$ and $\mathscr{T}^{b}$, respectively.  Two $t$-structures $\mathscr{T}_i:=(\mathscr{T}^{\leqslant 0}_i,\mathscr{T}^{\geqslant 0}_i)$ for $i=1,2$ on $\mathscr{T}$ are said to be \emph{equivalent} if there exists a natural number $n$ with $\mathscr{T}^{\leqslant -n}_1\subseteq \mathscr{T}^{\leqslant 0}_2\subseteq \mathscr{T}^{\leqslant n}_1$. Equivalent $t$-structures give rise to identical $\mathscr{T}^{-}, \mathscr{T}^{+}$ and $\mathscr{T}^{b}$.
Moreover,  $(\mathscr{T}^{\leqslant 0},\mathscr{T}^{\geqslant 0})$ restricts to a
bounded $t$-structure $\big(\mathscr{T}^{\leqslant 0}\cap\mathscr{T}^{b},\mathscr{T}^{\geqslant 0}\cap\mathscr{T}^{b}\big)$ on $\mathscr{T}^{b}$.

\begin{Bsp}\label{Typical example}
Let $\mathscr{T}$ be a triangulated category with coproducts. Let $\mathscr{A}$ be a set of compact objects in $\mathscr{T}$ with $\mathscr{A}[1]\subseteq \mathscr{A}$.  It is known that $(\text{Coprod}(\mathscr{A}), (\text{Coprod}(\mathscr{A})[1])^{\perp})$ is a $t$-structure on $\mathscr{T}$ (see {\rm \cite[Theorem A.1 and Proposition A.2]{AJS}} and \cite[Theorem 2.3.3]{CNS}). This is called the \emph{$t$-structure on $\mathscr{T}$ generated by $\mathscr{A}$}, and such $t$-structures are called \emph{compactly generated $t$-structures}. In particular, for any compact object $G$ of $\mathscr{T}$, there exists a unique $t$-structure $(\mathscr{T}^{\leqslant 0}_G, \mathscr{T}^{\geqslant 0}_G)$ on $\mathscr{T}$ generated by $G$, that is,
$\mathscr{T}^{\leqslant 0}_G= \overline{\langle G \rangle}^{(-\infty, 0]}$ and $\mathscr{T}^{\geqslant 0}_G=G(-\infty,-1]^{\bot}$. Both $\mathscr{T}^{\leqslant 0}_G$ and  $\mathscr{T}^{\geqslant 0}_G$ are closed under coproducts in $\mathscr{T}$. Moreover, if $H$ is another compact object of $\mathscr{T}$ with
$\langle H \rangle=\langle G \rangle$, then the $t$-structures on $\mathscr{T}$
generated by $G$ and by $H$ are equivalent.
\end{Bsp}

\begin{Def}{\rm \cite[Definition 0.14]{Neeman5}}\label{prefer}
Let $\mathscr{T}$ be a compactly generated triangulated category. If $\mathscr{T}$ has a compact generator $G$, then the \emph{preferred equivalence class} of $t$-structures on $\mathscr{T}$ is defined to be the one containing the $t$-structure $(\mathscr{T}^{\leqslant 0}_G, \mathscr{T}^{\geqslant 0}_G)$ generated by $G$.
\end{Def}

The restriction to $\mathscr{T}^c$ of the aisle of a $t$-structure on $\mathscr{T}$ in the preferred equivalence class can be controlled by the compact generator of $\mathscr{T}$.

\begin{Lem}\label{Compact generator}
Let $\mathscr{T}$ be a compactly generated triangulated category with a compact generator $G$. Then the following statements are true.

$(1)$ $\mathscr{T}^c\cap\mathscr{T}_G^{\leqslant 0}=\langle G\rangle^{(-\infty,0]}$.

$(2)$ Let $(\mathscr{T}^{\leqslant 0}, \mathscr{T}^{\geqslant 0})$ be a $t$-structure on $\mathscr{T}$ in the preferred equivalence class. Then there is a nonnegative integer $n$ with
$\langle G\rangle^{(-\infty,-n]} \subseteq \mathscr{T}^c\cap\mathscr{T}^{\leqslant 0}\subseteq\langle G\rangle^{(-\infty,n]}$.
\end{Lem}

\begin{proof}
Recall that $\mathscr{T}_G^{\leqslant 0}=\overline{\langle G\rangle}^{(-\infty,0]}={\rm Coprod}(G(-\infty,0])$. By Lemma \ref{Smd},
$\mathscr{T}^c\cap\mathscr{T}_G^{\leqslant 0}\subseteq\langle G\rangle^{(-\infty,0]}.$
Since $\langle G\rangle^{(-\infty,0]}\subseteq\mathscr{T}^c\cap\mathscr{T}_G^{\leqslant 0}$, we have $\mathscr{T}^c\cap\mathscr{T}_G^{\leqslant 0}=\langle G\rangle^{(-\infty,0]}$. This shows $(1)$.
Since $(\mathscr{T}^{\leqslant 0}, \mathscr{T}^{\geqslant 0})$ is equivalent to $(\mathscr{T}^{\leqslant 0}_G, \mathscr{T}^{\geqslant 0}_G)$,
there is a nonnegative integer $n$ such that $\mathscr{T}_G^{\leqslant -n}\subseteq\mathscr{T}^{\leqslant 0} \subseteq \mathscr{T}_G^{\leqslant n}$.
By $(1)$, we have
$\langle G\rangle^{(-\infty,-n]}=\mathscr{T}^c\cap\mathscr{T}_G^{\leqslant -n}\subseteq\mathscr{T}^c\cap\mathscr{T}^{\leqslant 0} \subseteq \mathscr{T}^c\cap\mathscr{T}_G^{\leqslant n}=\langle G\rangle^{(-\infty,n]}.$
This shows $(2)$.
\end{proof}

Let $\mathscr{S}$ be a triangulated category with a classical generator $G$. We define a thick subcategory of $\mathscr{S}$:
$$\mathscr{S}_{\rm tc}:=\{X\in\mathscr{S}\mid \Hom_\mathscr{S}(G[n], X)=0, \; n \gg 0\}.$$
The objects of $\mathscr{S}_{\rm tc}$ can be regarded as \emph{truncated} objects with respect to $G$, similar to truncated modules over connective $\mathbb{E}_1$-rings (see Definition \ref{truncated}). Clearly, different classical generators define the same $\mathscr{S}_{\rm tc}$. Moreover, $\mathscr{S}_{\rm tc}=\mathscr{S}$ if and only if $G\in\mathscr{S}_{\rm tc}$.

The following result is simple, but useful in practice. It can be used to show that certain triangulated categories cannot have any bounded $t$-structures. For example, for a connective $\mathbb{E}_1$-ring $R$,
the homotopy category of perfect $R$-modules has no bounded $t$-structure provided that there are infinitely many nonzero homotopy groups of $R$.

\begin{Lem}\label{HBTS}
If $\mathscr{S}$ has a bounded $t$-structure, then $\mathscr{S}_{\rm tc}=\mathscr{S}$.
\end{Lem}

\begin{proof}
Let $(\mathscr{S}^{\leqslant 0},\mathscr{S}^{\geqslant 0})$ be a bounded $t$-structure on $\mathscr{S}$. Since $G\in \mathscr{S}$, there is a positive integer $n$ such that $G[n]\in \mathscr{S}^{\leqslant 0}$ and $G[-n]\in\mathscr{S}^{\geqslant 0}$. As $\mathscr{S}^{\geqslant 0}\subseteq\mathscr{S}$ is closed under negative shifts, $G[i]\in \mathscr{S}^{\geqslant 0}$ for  $i\leqslant -n$. It follows from $\Hom_\mathscr{S}(\mathscr{S}^{\leqslant 0}, \mathscr{S}^{\geqslant 0}[-1])=0$ that $\Hom_\mathscr{S}(G,G[j])=0$ for $j\leqslant -2n-1$. This implies $G\in\mathscr{S}_{\rm tc}$ and
thus $\mathscr{S}_{\rm tc}=\mathscr{S}$.
\end{proof}

\subsection{Completions of triangulated categories}\label{Completion-TC}
The main approach in this paper is via the theory of completion of triangulated categories, introduced and developed by Neeman in a series of papers (see \cite{Neeman2,Neeman1,Neeman5}). In this section, we recall the relevant definitions, examples, and results (Theorems \ref{Completion-0} and \ref{MCP}) that we will need throughout the paper. Moreover, we propose the notion of singularity categories for triangulated categories with classical generators, in terms of completions. This simultaneously generalizes the algebraic and geometric notions of singularity categories.

Throughout this section, let $\mathscr{S}$ be an essentially small triangulated category with the shift functor denoted by $[1]$.

\begin{Def}{\rm \cite[Definition 10]{Neeman2}}\label{definition of metric}
A \emph{good metric} on $\mathscr{S}$ is a sequence $\mathscr{M}:=\{\mathscr{M}_{n}\}_{n\in \mathbb{N}}$
of full subcategories of $\mathscr{S}$ containing $0$ and satisfying the following conditions for all $n\in \mathbb{N}$:

(G1) $\mathscr{M}_{n} \ast \mathscr{M}_{n}= \mathscr{M}_{n}$, that is, $\mathscr{M}_n$ is closed under extensions in $\mathscr{S}$;

(G2) $\mathscr{M}_{n+1}[-1] \cup \mathscr{M}_{n+1} \cup \mathscr{M}_{n+1}[1] \subseteq \mathscr{M}_{n}$. This implies $\mathscr{M}_{n+|j|}\subseteq \mathscr{M}_n[j]$ for any $j\in\mathbb{Z}$.

The good metric $\mathscr{M}$ is said to be \emph{finer} than another good metric $\mathscr{N}:=\{\mathscr{N}_{n}\}_{n\in \mathbb{N}}$ if for each $n$, there exists $m\in\mathbb{N}$ such that $\mathscr{M}_m\subseteq \mathscr{N}_n$. We denote this partial order by $\mathscr{M}\preceq \mathscr{N}$. The good metrics $\mathscr{M}$ and $\mathscr{N}$ are said to be \emph{equivalent} if $\mathscr{M}\preceq \mathscr{N}\preceq\mathscr{M}$.
\end{Def}

Good metrics can be obtained from the aisles of $t$-structures by restriction.

\begin{Bsp}\label{Induced}
Let $\mathscr{T}$ be a triangulated category, $(\mathscr{T}^{\leqslant 0}, \mathscr{T}^{\geqslant 0})$ a $t$-structure on $\mathscr{T}$, and $\mathscr{S}$ a full triangulated subcategory of $\mathscr{T}$. Let $\mathscr{M}_{n}:=\mathscr{S}\cap\mathscr{T}^{\leqslant -n}$ for $n\in\mathbb{N}$. Then
$\{\mathscr{M}_{n}\}_{n\in \mathbb{N}}$ is a good metric on $\mathscr{S}$, which is called \emph{the good metric induced from the aisle of $(\mathscr{T}^{\leqslant 0}, \mathscr{T}^{\geqslant 0})$}.

If $\mathscr{T}$ is compactly generated by a compact generator $G$,  and $(\mathscr{T}^{\leqslant 0}, \mathscr{T}^{\geqslant 0})$ is any $t$-structure on $\mathscr{T}$ in the preferred equivalence class, then
$\{\mathscr{T}^c\cap\mathscr{T}^{\leqslant -n}\}_{n\in\mathbb{N}}$ is called a \emph{preferred good metric} on $\mathscr{T}^c$. This is a $G$-good metric (see {\rm Definition \ref{PEC}}) on $\mathscr{T}^c$, due to Lemma \ref{Compact generator}(2).
\end{Bsp}

\begin{Def}{\rm \cite[Definition 1.6]{Neeman1}}\label{definition of C.S.}
Let $\mathscr{M}:=\{\mathscr{M}_{n}\}_{n\in\mathbb{N}}$ be a good metric on $\mathscr{S}$. A chain
of morphisms $\{X_{\bullet},f_{\bullet}\}:\;\; X_0 \stackrel{f_1}{\longrightarrow} X_1 \stackrel{f_2}{\longrightarrow} X_2\stackrel{f_3}{\longrightarrow} X_3\longrightarrow \cdots$
in $\mathscr{S}$ is called a \emph{Cauchy sequence with respect to $\mathscr{M}$} (or a \emph{Cauchy sequence} for short when $\mathscr{M}$ is clear) if for any $i\geqslant 1$, there exists $n_i\geqslant 1$ such that $\text{Cone}(f_j) \in \mathscr{M}_{i}$ for all $j \geqslant n_i$. The Cauchy sequence $\{X_{\bullet},f_{\bullet}\}$ is \emph{stable} if there exists $m\geqslant 1$ such that $f_{j+1}:X_j\to X_{j+1}$ is an isomorphism for $j\geqslant m$.
\end{Def}

\begin{Lem}\label{stable lemma}
Let $\mathscr{A}$ be a full subcategory of $\mathscr{S}$ and let $\{X_{\bullet}, f_{\bullet}\}$ be a Cauchy sequence in $\mathscr{S}$ with respect to $\mathscr{M}$ such that $X_n \in\mathscr{A}$ for all $n\in \mathbb{N}$. If there exists a natural number $m$ such that $\mathscr{A}\subseteq \mathscr{M}_m^{\perp}$ or $\mathscr{A}\subseteq {^{\perp}}\mathscr{M}_m$, then the sequence $\{X_{\bullet},f_{\bullet}\}$ is stable.
\end{Lem}

\begin{proof}
Suppose $\mathscr{A}\subseteq \mathscr{M}_m^{\perp}$. By the definition of Cauchy sequence, there exists a natural number $t$ such that $C_{n}:=\text{Cone}(f_{n+1})\in \mathscr{M}_{m+1}$ for $n\geqslant t$. By (G2) in Definition \ref{definition of metric}, $\mathscr{M}_{m+1}[-1] \subseteq \mathscr{M}_m$. This forces $C_n[-1]\in\mathscr{M}_m$. Since $X_n \in\mathscr{A}\subseteq\mathscr{M}_m^{\perp}$, we have $\Hom(C_{n}[-1], X_n)=0$. Thus $X_{n+1}\simeq X_n\oplus C_n$. Note that $\Hom(C_n, X_{n+1})=0$ since $C_n\in \mathscr{M}_{m+1}\subseteq\mathscr{M}_m$ and $X_{n+1}\in\mathscr{A}\subseteq \mathscr{M}_m^{\perp}$. Consequently, $C_{n}=0$ and $f_n:X_n\to X_{n+1}$ is an isomorphism for $n\geqslant t$. Thus $\{X_{\bullet},f_{\bullet}\}$ is stable. The stability of $\{X_{\bullet},f_{\bullet}\}$ under the condition $\mathscr{A}\subseteq {^{\perp}}\mathscr{M}_m$ can be shown similarly.
\end{proof}

Now, we denote by $\mathscr{S}\Modcat$ the abelian category of additive functors from $\mathscr{S}^{\text{op}}$ to the category of abelian groups. The Yoneda functor $$\mathfrak{y}:\mathscr{S}\lra \mathscr{S}\Modcat,\quad X\mapsto \Hom_\mathscr{S}(-,X)$$
is fully faithful, and the automorphism $[1]: \mathscr{S}\to\mathscr{S}$ induces an automorphism of abelian categories:
$$\Sigma:\mathscr{S}\Modcat\lra  \mathscr{S}\Modcat,\quad F\mapsto [\Sigma(F): X\mapsto F(X[-1])]$$
for $F\in \mathscr{S}\Modcat$ and $X\in\mathscr{S}$.
This implies that $\Sigma^j(\mathfrak{y}(X))\simeq\mathfrak{y}(X[j])$ for $j\in\mathbb{Z}$.

\begin{Def}{\rm \cite[Definition 1.10]{Neeman1}}\label{Completion}
Let $\mathscr{S}$ be a triangulated category with a good metric $\mathscr{M}:=\{\mathscr{M}_{n}\}_{n\in\mathbb{N}}$. Let $\mathscr{A}$ be a full subcategory of $\mathscr{S}$. We define full subcategories $\mathfrak{L}(\mathscr{A})$, $\mathfrak{C}(\mathscr{S})$ and $\mathfrak{S}(\mathscr{A})$ of $\mathscr{S}\Modcat$:

$(1)$ The objects of $\mathfrak{L}(\mathscr{A})$ are the functors $F$ in $\mathscr{S}\Modcat$ such that $F\simeq \mathop{\text{colim}}\limits_{\longrightarrow}\mathfrak{y}(X_{n})$, where $\{X_{\bullet}\}$ is a Cauchy sequence in $\mathscr{S}$ with respect to $\mathscr{M}$ and with $X_n\in\mathscr{A}$ for all $n\in\mathbb{N}$. Clearly, $\mathfrak{y}(\mathscr{A})\subseteq \mathfrak{L}(\mathscr{A})$.

$(2)$ The objects of $\mathfrak{C}(\mathscr{S})$ are the functors $F$ in $\mathscr{S}\Modcat$ with $F(\mathscr{M}_j)=0$ for some $j\geqslant 0$.

$(3)$ $\mathfrak{S}(\mathscr{A}):=\mathfrak{L}(\mathscr{A}) \cap \mathfrak{C}(\mathscr{S})$. In particular, $\mathfrak{S}(\mathscr{S})$ is called \emph{the completion of $\mathscr{S}$} with respect to $\mathscr{M}$.
\end{Def}

\begin{Theo}{\rm \cite[Theorem 15]{Neeman2}(see also \cite[Theorem 2.11]{Neeman1})} \label{Completion-0}
Let $\mathscr{S}$ be a triangulated category with a good metric $\mathscr{M}:=\{\mathscr{M}_{n}\}_{n\in\mathbb{N}}$.  Then the category $\mathfrak{S}(\mathscr{S})$ with the automorphism $\Sigma$ is triangulated, where the triangles are given by the sequences $A\stackrel{\alpha}{\longrightarrow} B\stackrel{\beta}{\longrightarrow}C\stackrel{\gamma}{\longrightarrow}\Sigma(A)$ in $\mathfrak{S}(\mathscr{S})$ which are isomorphic to the colimit of the image under $\mathfrak{y}$ of a Cauchy sequence (with respect to $\mathscr{M}$) of triangles in $\mathscr{S}$:
$$
\{A_{\bullet}\}\stackrel{\alpha_{\bullet}}{\longrightarrow}
\{B_{\bullet}\}\stackrel{\beta_{\bullet}}{\longrightarrow}\{C_{\bullet}\}
\stackrel{\gamma_{\bullet}}{\longrightarrow}\{A_{\bullet}\}[1].
$$
\end{Theo}

In general, $\mathfrak{y}$ does \emph{not} restrict to a functor $\mathscr{S}\to\mathfrak{S}(\mathscr{S})$ because $\mathfrak{y}(\mathscr{S})$ may not be contained in $\mathfrak{C}(\mathscr{S})$. To make up for this deficiency, we introduce a thick subcategory
$\mathscr{S}(\mathscr{M}):=\bigcup_{n\in\mathbb{N}}\mathscr{M}_n^{\perp}$ of $\mathscr{S}$ associated with any good metric $\mathscr{M}$ which embeds into $\mathfrak{S}(\mathscr{S})$. The following result is elementary and we leave its proof to the reader.

\begin{Lem}\label{Properties}
$(1)$ For a full subcategory $\mathscr{A}$ of $\mathscr{S}$ and for any $j\in\mathbb{Z}$, there are equalities of additive categories:
$\mathfrak{L}(\mathscr{A}[j])=\Sigma^{j}(\mathfrak{L}(\mathscr{A}))$, $\mathfrak{C}(\mathscr{S})=\Sigma^{j}(\mathfrak{C}(\mathscr{S}))$ and
$\mathfrak{S}(\mathscr{A}[j])=\Sigma^{j}(\mathfrak{S}(\mathscr{A}))$.

$(2)$  The functor $\Sigma$ restricts to automorphisms of additive categories $\mathfrak{L}(\mathscr{S})$, $\mathfrak{C}(\mathscr{S})$ and $\mathfrak{S}(\mathscr{S})$.

$(3)$ The Yoneda functor $\mathfrak{y}$ restricts to a fully faithful triangle functor $\mathscr{S}(\mathscr{M})\to \mathfrak{S}(\mathscr{S})$. In particular, $\mathfrak{y}(\mathscr{S}(\mathscr{M}))=\mathfrak{y}(\mathscr{S})\cap \mathfrak{S}(\mathscr{S})$.

$(4)$ Equivalent good metrics of $\mathscr{S}$ produce the same $\mathfrak{L}(\mathscr{A})$, $\mathfrak{C}(\mathscr{S})$ and $\mathfrak{S}(\mathscr{S})$.
\end{Lem}

In light of Lemma \ref{Properties}(3), if $\mathscr{S}(\mathscr{M})=\mathscr{S}$, then $\mathscr{M}$ is called an \emph{embeddable} metric. Note that a metric $\mathscr{M}$ is embeddable if and only if $\mathfrak{y}(\mathscr{S})\subseteq\mathfrak{S}(\mathscr{S})$. If $\mathscr{S}$ has a classical generator $G$ and $\mathscr{M}$ is a good metric on $\mathscr{S}$ in the preferred equivalence class, then $\mathscr{S}(\mathscr{M})=\mathscr{S}_{\rm tc}$. In this case, $\mathscr{M}$ is embeddable if and only if $\Hom_\mathscr{S}(G[i],G)=0$ for $i\gg 0$.

Next, we recall a method from  \cite{Neeman1,Neeman2} to calculate $\mathfrak{S}(\mathscr{S})$.

\begin{Def}{\rm \cite[Definition 3.5]{Neeman1}}\label{Good extension}
Let $\mathscr{S}$ and $\mathscr{T}$ be triangulated categories and let $F:\mathscr{S}\to \mathscr{T}$ be a fully faithful triangle functor. We define
$\mathfrak{y}_F:\mathscr{T}\to \mathscr{S}\Modcat$ by sending $X$ to $\Hom_\mathscr{T}(F(-), X)$
for each $X\in \mathscr{T}$. Then the Yoneda embedding of $\mathscr{S}$ is naturally isomorphic to the composition of $F$ with $\mathfrak{y}_F$, that is , $\mathfrak{y}\simeq\mathfrak{y}_F\circ F$.

Suppose that $\mathscr{T}$ has coproducts and $\mathscr{M}$ is a good metric on $\mathscr{S}$. The functor $F$ is called a \emph{good extension} with respect to $\mathscr{M}$ if for any Cauchy sequence $\{A_{\bullet}\}$ in $\mathscr{S}$ with respect to $\mathscr{M}$, the natural map $\mathop{\text{colim}}\limits_{\longrightarrow}\mathfrak{y}(A_{n})\to \mathfrak{y}_F(\mathop{\text{Hocolim}}\limits_{\longrightarrow}F(A_n))$ is an isomorphism in $\mathscr{S}\Modcat$.
\end{Def}

\begin{Bsp}
Let $\mathscr{T}$ be a triangulated category with coproducts and $\mathscr{S}$ a full subcategory of $\mathscr{T}^{{c}}$.
Then the inclusion $\mathscr{S}\subseteq\mathscr{T}$ is a good extension with respect to any good metric $\mathscr{M}$ on $\mathscr{S}$. This follows from the fact that the functor $\Hom_{\mathscr{T}}(X,-):\mathscr{T}\to\mathbb{Z}\Modcat$ for $X\in\mathscr{T}^c$ sends homotopy colimts to colimits.
\end{Bsp}

\begin{Def}{\rm \cite[Definition 19]{Neeman2}}\label{LS}
Let $F:\mathscr{S}\rightarrow \mathscr{T}$ be a good extension with respect to a good metric $\mathscr{M}$, and let $\mathscr{A}\subseteq\mathscr{S}$ be a full subcategory of $\mathscr{S}$. We denote by $\mathfrak{L}'(\mathscr{A})$ the full subcategory of $\mathscr{T}$ consisting of objects $X$ such that $X\simeq \mathop{\text{Hocolim}}\limits_{\longrightarrow}F(X_{n})$ in $\mathscr{T}$, where $\{X_{\bullet}\}$ is a Cauchy sequence in $\mathscr{S}$ with respect to $\mathscr{M}$ and with $X_n\in\mathscr{A}$ for all $n$.
Further, let $\widehat{\mathfrak{S}}(\mathscr{A}):=\mathfrak{L}'(\mathscr{A})\cap \mathfrak{y}_F^{-1}(\mathfrak{C}(\mathscr{S})).$
\end{Def}
Note that the functor $\mathfrak{y}_F$ in Definition \ref{Good extension} restricts to an essentially surjective (or dense in other terminology) functor $\mathfrak{L}'(\mathscr{A})\to \mathfrak{L}(\mathscr{A})$. Moreover, $\mathfrak{y}_F^{-1}(\mathfrak{C}(\mathscr{S}))=\bigcup_{n\in\mathbb{N}} F(\mathscr{M}_{n})^{\perp}$ by \cite[Observation 3.2]{Neeman1}. This implies $F(\mathscr{S})\cap\mathfrak{y}_F^{-1}(\mathfrak{C}(\mathscr{S}))=F(\mathscr{S}(\mathscr{M}))$.

For good metrics induced from $t$-structures, $\mathfrak{L}'(\mathscr{S})$ and $\widehat{\mathfrak{S}}(\mathscr{S})$ can be characterized in terms of intrinsic subcategories of $\mathscr{T}$.

\begin{Def}{\rm (see \cite{Neeman2})}\label{TC}
Let $\mathscr{T}$ be a triangulated category with coproducts and let $(\mathscr{T}^{\leqslant 0},\mathscr{T}^{\geqslant 0})$ be a $t$-structure on $\mathscr{T}$. The full subcategories $\mathscr{T}_c^{-}$ and $\mathscr{T}_c^{b}$ of $\mathscr{T}$ are defined by
$$\mathscr{T}_c^{-}:=\bigcap_{n\in\mathbb{N}}{(\mathscr{T}^{{c}}\ast \mathscr{T}^{\leqslant -n})}\;\;\mbox{and}\;\; \mathscr{T}_c^{b}:=\mathscr{T}_c^{-}\cap\mathscr{T}^{b}.$$
\end{Def}

The category $\mathscr{T}_c^{-}$ can be regarded as the \emph{closure} of $\mathscr{T}^c$ in $\mathscr{T}$, that is, its objects can be approximated by compacts to arbitrarily high order. Equivalent $t$-structures give rise to identical $\mathscr{T}_c^{-}$ and $\mathscr{T}_c^{b}$. Obviously, $\mathscr{T}^{{c}}\subseteq \mathscr{T}_c^{-}$. If $\mathscr{T}^{{c}}\subseteq \mathscr{T}^{-}$, then $\mathscr{T}_c^{-}\subseteq\mathscr{T}^{-}$. The assumption $\mathscr{T}^{{c}}\subseteq \mathscr{T}^{-}$ holds in the case that $\mathscr{T}$  has a compact generator and $(\mathscr{T}^{\leqslant 0}, \mathscr{T}^{\geqslant 0})$ is a $t$-structure on $\mathscr{T}$ in the preferred equivalence class. In this case, all of these subcategories are intrinsically defined, by taking any $t$-structure from the preferred equivalence class.

\begin{Theo}{\rm \cite{Neeman2,Neeman1}}\label{MCP}
Let $\mathscr{S}$ be a triangulated category with a good metric $\mathscr{M}:=\{\mathscr{M}_{n}\}_{n\in\mathbb{N}}$.

$(1)$ Let $\mathscr{T}$ be a triangulated category and let $F:\mathscr{S}\rightarrow \mathscr{T}$ be a good extension with respect to $\mathscr{M}$. Then $\widehat{\mathfrak{S}}(\mathscr{S})$ is a triangulated subcategory of $\mathscr{T}$ and the functor $\mathfrak{y}_F$ restricts to a triangle equivalence
$\widehat{\mathfrak{S}}(\mathscr{S})\rightarrow \mathfrak{S}(\mathscr{S})$.

$(2)$ Suppose that $\mathscr{T}$ is compactly generated and has a compact generator $G$. Let
$(\mathscr{T}^{\leqslant 0}, \mathscr{T}^{\geqslant 0})$ be a $t$-structure on $\mathscr{T}$ in the preferred equivalence class, $\mathscr{S}:=\mathscr{T}^{c}$, $\mathscr{M}_n:=\mathscr{S}\cap \mathscr{T}^{\leqslant -n}$ for each $n\in\mathbb{N}$ and $F:\mathscr{S}\to\mathscr{T}$ the canonical inclusion. Then:

\smallskip
$(a)$ $\mathfrak{y}_F^{-1}(\mathfrak{C}(\mathscr{S}))=\mathscr{T}^+$ and $\mathfrak{L}'(\mathscr{S})\subseteq \mathscr{T}_c^{-}\subseteq \mathscr{T}^{-}$. Thus $\mathscr{S}(\mathscr{M})=\mathscr{S}\cap\mathscr{T}^b=\mathscr{S}_{\rm tc}$ and $\widehat{\mathfrak{S}}(\mathscr{S})\subseteq\mathscr{T}_c^b$.

$(b)$ If $\Hom_\mathscr{T}(G,G[i])=0$ for $i\gg 0$, then $\mathscr{T}_c^b\subseteq\mathscr{T}_c^{-}$ are thick subcategories of $\mathscr{T}$, $\mathfrak{L}'(\mathscr{S})=\mathscr{T}_c^{-}$ and $\widehat{\mathfrak{S}}(\mathscr{S})=\mathscr{T}_c^b$.
\end{Theo}

Theorem \ref{MCP}(1) is exactly \cite[Theorem 20]{Neeman2}; see \cite[Theorem 3.15]{Neeman1} for its proof.
In Theorem \ref{MCP}(2)$(a)$, the first equality is shown in \cite[Example 3.3]{Neeman1}
and other inclusions of categories can be concluded from the proof of \cite[Lemma 7.5(iii)]{Neeman5}, where
the condition $\Hom_\mathscr{T}(G,G[i])=0$ for $i\gg 0$ is not needed. Under this condition, Theorem \ref{MCP}(2)$(b)$ is given in \cite[Example 3.10 and Proposition 0.15(i)]{Neeman1}; see also \cite[Proposition 0.19]{Neeman5} and \cite[Lemma 7.5]{Neeman5} for proofs.

\subsection{Examples of completions of triangulated categories}\label{EXCT}

In this section, we give several examples of completions of triangulated categories that consist of all compact objects of derived categories.

\begin{Bsp}\label{Ordinary ring}
Let $R$ be an associative ring with identity. We consider $\mathscr{T}:=\D{R}$. Throughout the paper,
we always identify $\mathscr{T}^c$ with $\Kb{\prj{R}}$, up to triangle equivalence. By Theorem \ref{MCP}(2) (see also \cite[Example 3.1]{Neeman5}), there are equalities
$\mathscr{T}_c^{-}=\mathscr{K}^{-}(\prj{R})$ and $ \widehat{\mathfrak{S}}(\mathscr{T}^c)=\mathscr{T}_c^{b}=\mathscr{K}^{-, b}(\prj{R})$.
If $R$ is left coherent, then $\widehat{\mathfrak{S}}(\mathscr{T}^c)$ is equivalent to $\Db{R\modcat}$,
the bounded derived category of finitely presented left $R$-modules.
\end{Bsp}

Before giving an example from schemes, we recall several standard notation from algebraic geometry
and a theorem from \cite{Neeman4}.

Throughout this section, let $X$ be a quasicompact, quasiseparated scheme and let $Z$ be a closed subset of $X$ such that $X-Z$ is quasicompact. We denote by $\mathscr{D}_{\rm qc}(X)$ the full subcategory of the unbounded derived category of $\mathscr{O}_X$-modules consisting of (cochain) complexes of $\mathscr{O}_X$-modules with \emph{quasicoherent} cohomology, and by $\mathscr{D}_{{\rm qc}, Z}(X)$ the full subcategory of $\mathscr{D}_{\rm qc}(X)$ consisting of complexes whose cohomology is supported on $Z$ (that is, the restriction of those complexes to $X-Z$ is acyclic).
Note that $\mathscr{D}_{{\rm qc}, Z}(X)$ has a standard $t$-structure $(\mathscr{D}_{{\rm qc}, Z}(X)^{\leqslant 0}, \mathscr{D}_{{\rm qc}, Z}(X)^{\geqslant 0})$, where objects of $\mathscr{D}_{{\rm qc}, Z}(X)^{\leqslant 0}$ and $\mathscr{D}_{{\rm qc}, Z}(X)^{\geqslant 0}$ have nonzero cohomologies concentrated in nonpositive and nonnegative degrees, respectively. Further, we define
$\mathscr{D}^{\rm perf}_Z(X):=\mathscr{D}^{\rm perf}(X)\cap\mathscr{D}_{{\rm qc}, Z}(X)$, the full subcategory of $\mathscr{D}_{\rm qc}(X)$ consisting of all perfect complexes supported on $Z$, where a complex is said to be \emph{perfect} if it is locally isomorphic to a bounded complex of finite-rank vector bundles. We also denote by
$\mathscr{D}^p_{{\rm qc}, Z}(X)$ (respectively, $\mathscr{D}^{p, b}_{{\rm qc}, Z}(X)$) the full subcategory of
$\mathscr{D}_{{\rm qc}, Z}(X)$ consisting of \emph{pseudocoherent} complexes (see Section \ref{Section 1.2}) on $X$, whose cohomology is supported on $Z$ (respectively, and vanishes in all but finitely many degrees). When $X$ is noetherian,
$\mathscr{D}^{p}_{{\rm qc}, Z}(X)=\mathscr{D}^-_{{\rm coh}, Z}(X)$ and
$\mathscr{D}^{p, b}_{{\rm qc}, Z}(X)=\mathscr{D}^{b}_{{\rm coh}, Z}(X),$ where {\rm coh} denotes \emph{coherent cohomology}.  In our discussions, when $Z=X$, the subindex $Z$ in the above categories is always omitted.

\begin{Theo}{\rm \cite[Theorem 3.2(i)-(iii)]{Neeman4}}\label{CCST}
The category $\mathscr{D}_{{\rm qc}, Z}(X)$ is a compactly generated triangulated category with a single compact generator, $\big(\mathscr{D}_{{\rm qc}, Z}(X)\big)^c=\mathscr{D}^{{\rm perf}}_{Z}(X)$, and the standard $t$-structure $(\mathscr{D}_{{\rm qc}, Z}(X)^{\leqslant 0}, \mathscr{D}_{{\rm qc}, Z}(X)^{\geqslant 0})$ on $\mathscr{D}_{{\rm qc}, Z}(X)$ is in the preferred equivalence class.
\end{Theo}

\begin{Bsp}\label{Perfect complex}
Let $\mathscr{T}:=\mathscr{D}_{{\rm qc}, Z}(X)$. Combing Theorems \ref{CCST} and \ref{MCP}(2) with \cite[Theorem 5.1]{Neeman4}, we have
$$
\mathscr{T}^c=\mathscr{D}^{{\rm perf}}_Z(X),\;\; \mathscr{T}_c^{-}=\mathscr{D}^{p}_{{\rm qc}, Z}(X)\;\;\mbox{and}\;\; \widehat{\mathfrak{S}}(\mathscr{T}^c)=\mathscr{T}_c^{b}=\mathscr{D}^{p, b}_{{\rm qc}, Z}(X);$$
see also \cite[Section 10]{Neeman4} for details. In particular,
$\mathscr{T}_c^{-}=\mathscr{D}^{-}_{{\rm coh}, Z}(X)$ and $\widehat{\mathfrak{S}}(\mathscr{T}^c)=\mathscr{D}^{b}_{{\rm coh}, Z}(X)$ for a noetherian scheme $X$.
\end{Bsp}

Next, we give an example from DG (that is, \emph{differential graded}) algebras.

Let $R:=\bigoplus_{i\in\mathbb{Z}}R^i$ be a DG algebra over a commutative ring $k$. We denote by $\D{R}$ the \emph{unbounded derived category} of left DG $R$-modules (for example, see \cite{keller1}). This is a compactly generated triangulated category with ${_R}R$ as a compact generator. Suppose that $R$ is \emph{left noetherian, nonpositive} and \emph{bounded}, that is, $H^0(R)$ is a left noetherian ring with $H^i(R)\in H^0(R)\modcat$ for $i<0$, $R^i=0$ for $i>0$, and $H^i(R)=0$ for $i\ll 0$ (see Appendix \ref{Section B} for more discussions). We denote by $\mathscr{D}^b_f(R)$ (resp., $\mathscr{D}^-_f(R)$) the full subcategory of $\D{R}$ consisting of all objects $X$ with $H^i(X)=0$ whenever $|i|\gg 0$ (resp., $i\gg 0$), and $H^i(X)\in H^0(R)\modcat$ for all $i\in\mathbb{Z}$. Then $\mathscr{D}(R)^c$ is a full triangulated subcategory of $\mathscr{D}^b_f(R)$. Let $\mathscr{D}_{\rm sg}(R):=\mathscr{D}^b_f(R)/\mathscr{D}(R)^c$, the Verdier quotient of $\mathscr{D}^b_f(R)$ by $\mathscr{D}(R)^c$.

\begin{Bsp}\label{nonpositive}
Let $R$ be a left noetherian, nonpositive and bounded DG $k$-algebra.
We consider $\mathscr{T}:=\D{R}$. By Theorem \ref{Connective case} and Corollary \ref{Spectrum}(1)
in the case of DG rings, there are equalities $\mathscr{T}^c={\langle R\rangle}$, $\mathscr{T}_c^-=\mathscr{D}^-_f(R)$ and $\widehat{\mathfrak{S}}(\mathscr{T}^c)=\mathscr{T}_c^{b}=\mathscr{D}^b_f(R)$.
\end{Bsp}

Lemma \ref{Properties}(3) provides an intrinsic way to define the (almost) singularity category for triangulated categories with classical generators.

\begin{Def}\label{Singularity}
Let $\mathscr{S}$ be a triangulated category with a classical generator $G$. The \emph{almost singularity category} of $\mathscr{S}$ is defined as  $$\mathscr{S}_{\rm{sg}}:=\mathfrak{S}_G(\mathscr{S})/\mathfrak{y}(\mathscr{S}_{\rm tc}),$$
the Verdier quotient of  $\mathfrak{S}_G(\mathscr{S})$ by the triangulated subcategory $\mathfrak{y}(\mathscr{S}_{\rm tc})$, where $\mathfrak{S}_G(\mathscr{S})$ denotes the completion of $\mathscr{S}$ with respect to a $G$-good metric on $\mathscr{S}$ (see Definition \ref{PEC}). If $\mathscr{S}_{\rm tc}=\mathscr{S}$ (or equivalently, $\Hom_\mathscr{S}(G[i],G)=0$ for $i\gg 0$), then $\mathscr{S}_{\rm{sg}}$ is further called the
\emph{singularity category} of $\mathscr{S}$.
This is consistent with the singularity category introduced at the end of Section \ref{Section 1.1}.
\end{Def}

\begin{Bsp} \label{Ring and Scheme}
By the calculations in Examples \ref{Ordinary ring}, \ref{Perfect complex} and \ref{nonpositive}, we have the following facts.

$(1)$ If $\mathscr{S}=\Kb{\prj{R}}$ for a left coherent ring $R$, then $\mathscr{S}_{\rm{sg}}=\Db{R\modcat}/\mathscr{K}^{b}(\prj{R})$ which is the singularity category of $R$. In this case, $\mathscr{S}$ is regular if and only if $R$ is \emph{left regular}, that is, each finitely presented left $R$-module has finite projective dimension. This can be generalized to connective ring spectra.
Precisely, if $\mathscr{S}$ is the homotopy category of the stable $\infty$-category of perfect modules over a left coherent $\mathbb{E}_1$-ring $R$, then $\mathscr{S}$ is almost regular if and only if $R$ is almost regular (see Corollary \ref{Spectrum}).

$(2)$ If $\mathscr{S}=\mathscr{D}^{{\rm perf}}_Z(X)$ for a noetherian scheme $X$ with a closed subset $Z$, then $\mathscr{S}_{\rm{sg}}=\mathscr{D}^{b}_{{\rm coh, Z}}(X)/\mathscr{D}^{\rm perf}_Z(X)$, which is the singularity category of $X$ in the case $Z=X$ (in the sense of Buchweitz and Orlov, see \cite{Buchweitz,Orlov}).

$(3)$ If $\mathscr{S}=\D{R}^c$ for a left noetherian, nonpositive and bounded DG $k$-algebra $R$, then
$\mathscr{S}_{\rm{sg}}=\mathscr{D}_{\rm sg}(R)$ which coincides with the usual \emph{singularity category} of $R$ in the literature (for example, see \cite{HB} for some cases).
In particular, if $R$ is a left noetherian (ordinary) ring, then
$\mathscr{D}_{\rm sg}(R)=\Db{R\modcat}/\mathscr{K}^{b}(\prj{R})$.

Finally, we mention a simultaneous generalization of $(1)\mbox{-}(3)$.

Let $\mathscr{T}$ be a compactly generated triangulated category with a compact generator $G$ and let  $\mathscr{S}:=\mathscr{T}^c$. Suppose $\Hom_\mathscr{T}(G,G[i])=0$ for $i\gg 0$. Then  $\mathscr{S}_{\rm{sg}}=\mathscr{T}_c^b/(\mathscr{T}^c\cap \mathscr{T}^b)$ by Theorem \ref{MCP}(2).
Thus $\mathscr{S}$ is almost regular if and only if  $\mathscr{T}^c\cap \mathscr{T}^b=\mathscr{T}_c^b$; $\mathscr{S}$ is regular if and only if $\mathscr{S}=\mathscr{T}_c^b$.
\end{Bsp}

Finally, we mention two results which are related to decompositions of morphisms in triangulated categories.
They will be used in the discussions of the finiteness of finitistic dimension for triangulated categories with strong generators.

\begin{Lem}{\rm \cite{Neeman3}} \label{last lemma}
Suppose that $\mathscr{T}$ is a triangulated category with coproducts and $(\mathscr{T}^{\leq0},\mathscr{T}^{\geq0})$ is a $t$-structure on $\mathscr{T}$ such that $\mathscr{T}^{\geq0}$ is closed under coproducts in $\mathscr{T}$. Let $H$ be an object of $\mathscr{T}^b$ and $m$ a positive integer. The following statements are true.

$(1)$ There exists a positive integer $t$ (only depending on $H$ and $m$) such that, for any integer $n$, $$\mathscr{T}^{\geq n}\cap{\rm Coprod}_m\big(H(-\infty, \infty)\big)\subseteq {\rm smd}\big({\rm Coprod}_m(H[n-t,\infty))\big).$$

$(2)$ If $\mathscr{S}$ is a full triangulated subcategory of $\mathscr{T}$ with $H\in\mathscr{S}\subseteq\mathscr{T}_c^-$, then for any integer $n$,
any morphism $E\to F$ with $E\in\mathscr{S}$ and $F\in{\rm Coprod}_m\big(H[n,\infty)\big)$ factors through an object $F'\in {\rm coprod}_m\big(H[n,\infty)\big)$.

$(3)$ If $\mathscr{S}$ is a full triangulated subcategory of $\mathscr{T}$ with $H\in\mathscr{S}\subseteq\mathscr{T}_c^b$, then for any integer $n$,
$$\mathscr{S}\cap\mathscr{T}^{\geq n}\cap{\rm Coprod}_m(H(-\infty, \infty))\subseteq{\langle H \rangle}_m^{[n-t,\infty)},$$
and thus $\mathscr{S}\cap{\rm Coprod}_m\big(H(-\infty,\infty)\big)\subseteq {\langle H \rangle}_m$.
If, in addition, $\mathscr{S}\subseteq{\rm Coprod}_m\big(H(-\infty,\infty)\big)$ and $\mathscr{S}\subseteq\mathscr{T}$ is closed under direct summands, then $\mathscr{S}={\langle H \rangle}_m$.
\end{Lem}

\begin{proof}
$(1)$ was given in the proof of \cite[Lemma 2.4]{Neeman3}. $(2)$ was shown in \cite[Lemma 2.5]{Neeman3}
for the special case that $\mathscr{T}=\mathscr{D}_{\rm qc}(X)$ and $\mathscr{S}=\mathscr{D}^{b}_{\rm coh}(X)$ for a noetherian scheme $X$, but its proof also applies to the general case of $(2)$ under the assumption $\mathscr{S}\subseteq\mathscr{T}_c^-$. Further, a combination of $(1)$ and $(2)$ yields $(3)$.
\end{proof}

\begin{Bsp} \label{quasiexcellent}
We give an example of a triangulated category satisfying the assumptions of Lemma \ref{last lemma}.

Let $X$ be a noetherian, separated, finite-dimensional, quasiexcellent scheme. We define $\mathscr{T}:=\mathscr{D}_{\rm qc}(X)$. By Theorem \ref{CCST}, $\mathscr{T}$ has a compact generator $G$. So, we can consider the $t$-structure $(\mathscr{T}^{\leqslant 0}_G, \mathscr{T}^{\geqslant 0}_G)$ on $\mathscr{T}$ generated by $G$. Moreover, by Example \ref{Perfect complex}, $\mathscr{T}^c=\mathscr{D}^{{\rm perf}}(X)$ and $\mathscr{T}_c^{b}=\mathscr{D}^{b}_{\rm coh}(X)$. Now, it follows from \cite[Theorem 5.1 and Proof of Main Theorem]{KA} that there exists an object $H\in \mathscr{D}^{b}_{\rm coh}(X)$ and a positive integer $m$ with $\mathscr{T}={\rm Coprod}_m(H(-\infty,\infty))$ and $\mathscr{D}^{b}_{\rm coh}(X)={\langle H \rangle}_m$.
\end{Bsp}

\begin{Lem}\label{Strong generator}
Let $\mathscr{S}$ be a triangulated category with an object $G$ such that $\Hom_{\mathscr{S}}(G[i],G)=0$ for $i\geqslant d+1$ with $d\in\mathbb{N}$. Then, given an integer $n\in\mathbb{N}$, any morphism $D\ra F$ in $\mathscr{S}$ with $D\in\langle G\rangle_{n+1}$ and $F\in G(-\infty,-1]^{\bot}$ factors through an object of $\langle G\rangle_{n+1}^{[-n(d+1),\infty)}.$
\end{Lem}

\begin{proof}
Let $D\in\langle G\rangle_{n+1}$ be an arbitrary object. Then $D$ is a direct summand of an object $D'$ that lies in $\text{coprod}_{n+1}\big(G(-\infty,\infty)\big)$, and further there exists a projection morphism from $D'$ to $D$. Thus, to show Lemma \ref{Strong generator}, it suffices to prove the following:

$(\natural)$ For any $D\in\text{coprod}_{n+1}\big(G(-\infty,\infty)\big)$ and $F\in G(-\infty,-1]^{\bot}$, any morphism $f:D\rightarrow F$ in $\mathscr{S}$ factors through an object of $\text{coprod}_{n+1}\big(G[-n(d+1),\infty)\big)$.

We will prove $(\natural)$  by induction on $n$ and start from $n=0$. By Definition \ref{coprod}(1)(2),
each object of $ \text{coprod}_1\big(G(-\infty,\infty)\big)$ is a direct sum of an object of $\text{coprod}_1\big(G(-\infty,-1)\big)$ with an object of $\text{coprod}_1\big(G[0,\infty)\big)$.
Since $F\in G(-\infty,-1]^{\bot}$, we have $\Hom_{\mathscr{S}}\big(\text{coprod}_1\big(G(-\infty,-1)\big), F\big)=0$. This implies that the morphism $f:D\to F$ with $D\in\text{coprod}_1\big(G(-\infty,\infty)\big)$ factors through an object of $\text{coprod}_1\big(G[0,\infty)\big)$.
Thus $(\natural)$ holds for $n=0$.

Suppose that $(\natural)$ holds for an integer $n-1\geqslant 0$. We now need to show that $(\natural)$ holds for $n$. This will be done by using the dual of {\rm \cite[Lemma 1.6]{Neeman3}}:

Let $\mathscr{T}$ be a triangulated category with full subcategories $\mathscr{A}$, $\mathscr{C}$, $\mathscr{X}$, $\mathscr{Y}$ and $\mathscr{Z}$. Assume $\text{add}(\mathscr{A})=\mathscr{A}$ and $\text{add}(\mathscr{C})=\mathscr{C}$. Suppose that

$(a)$ for any morphisms $X\rightarrow Y$ and $Z\rightarrow Y$ in $\mathscr{T}$, with $X\in\mathscr{X}$, $Y\in\mathscr{Y}$ and $Z\in\mathscr{Z}$, factors as $$X\rightarrow A\rightarrow Y,\quad Z\rightarrow C\rightarrow Y$$ with $A\in\mathscr{A}$ and $C\in\mathscr{C}$;

$(b)$ any morphism $X\rightarrow D$ in $\mathscr{T}$, with $X\in\mathscr{X}$ and $D\in\mathscr{Y}\ast(\mathscr{C}[1])$, factors as $X\rightarrow A\rightarrow D$ with $A\in\mathscr{A}$.

\noindent Then any morphism $E\rightarrow Y$ in $\mathscr{T}$, with $E\in\mathscr{Z}\ast\mathscr{X}$ and $Y\in\mathscr{Y}$, must factor as $E\rightarrow B\rightarrow Y$ with $B\in\mathscr{C}\ast\mathscr{A}$.

\smallskip
Now, we apply the above result to our case by taking $$\mathscr{T}=\mathscr{S},\;\; \mathscr{A}=\text{coprod}_{n}\big(G[-n(d+1),\infty)\big),\;\;
\mathscr{C}=\text{coprod}_{1}\big(G[0,\infty)\big),$$
$$\mathscr{X}=\text{coprod}_{n}\big(G(-\infty,\infty)\big),\;\;
\mathscr{Y}=G(-\infty,-1]^{\bot},\;\;\mathscr{Z}=\text{coprod}_{1}\big(G[-\infty,\infty)\big).$$
Then $$\mathscr{Z}\ast\mathscr{X}=\text{coprod}_{n+1}\big(G(-\infty,\infty)\big),\;\;
\mathscr{C}\ast\mathscr{A}\subseteq\text{coprod}_{n+1}\big(G[-n(d+1),\infty)\big),$$
$$\mathscr{A}[-d-1]=\text{coprod}_{n}\big(G[-(n-1)(d+1),\infty)\big)\subseteq\mathscr{A}.$$
It remains to show that $(a)$ and $(b)$ are true.

By induction, any morphism from $\mathscr{X}$ to $\mathscr{Y}$ factors through an object of $\mathscr{A}[-d-1]$ and thus also of $\mathscr{A}$. By the $n=0$ case, any morphism from $\mathscr{Z}$ to $\mathscr{Y}$ factors through an object of $\mathscr{C}$. So,
$(a)$ is true. Since $\Hom_{\mathscr{S}}(G[i],G)=0$ for $i\geqslant d+1$, we have $\Hom_{\mathscr{S}}(G(-\infty,-1], \mathscr{C}[-d])=0$; in other words,  $\mathscr{C}\subseteq\mathscr{Y}[d]$. As $\mathscr{Y}\subseteq\mathscr{Y}[1]$ and $\mathscr{X}=\mathscr{X}[1]$, we have $\mathscr{Y}\ast(\mathscr{C}[1])\subseteq\mathscr{Y}[d+1]$ and $\mathscr{X}=\mathscr{X}[d+1]$. Now, by the inductive hypothesis, any morphism from $\mathscr{X}$ to $\mathscr{Y}\ast(\mathscr{C}[1])$ factors through an object of $\mathscr{A}$.
This verifies $(b)$. Consequently, any morphism from $\mathscr{Z}\ast\mathscr{X}$ to $\mathscr{Y}$ factor through an object of $\mathscr{C}\ast\mathscr{A}$ and thus also of $\text{coprod}_{n+1}\big(G[-n(d+1),\infty)\big)$. This shows that $(\natural)$ holds for $n$.
\end{proof}

\section{Bounded $t$-structures and completion-invariant triangulated categories}\label{BTSSC}
In this section, we first establish a lifting theorem for (not necessarily bounded) $t$-structures along completions of triangulated categories. Then we discuss the completions of a triangulated category with respect to good metrics associated to objects. Assuming the existence of a bounded $t$-structure on the category and the finiteness of finitistic dimension, we apply the lifting theorem to show that taking completions at those good metrics does not yield a new triangulated category.

\subsection{Lifting $t$-structures along completions of triangulated categories}\label{example section}

In this section we prove that under mild conditions a $t$-structure on a triangulated category can be lifted to a $t$-structure on its completion. The former $t$-structure is said to be \emph{extendable}.
The main result of this section is Theorem \ref{ET}, which provides
a key technique for showing Theorem \ref{Main result}.

Throughout this section, let $\mathscr{S}$ be an essentially small triangulated category with a good metric $\mathscr{M}:=\{\mathscr{M}_{n}\}_{n\in\mathbb{N}}$. We denote by $\mathfrak{S}(\mathscr{S})$ the completion of $\mathscr{S}$ with respect to $\mathscr{M}$. Recall from Definition \ref{Completion} that, for each full subcategory $\mathscr{A}$ of $\mathscr{S}$, we have defined two full subcategories $\mathfrak{L}(\mathscr{A})$  and $\mathfrak{S}(\mathscr{A})$ of $\mathscr{S}\Modcat$ via the Yoneda functor $\mathfrak{y}:\mathscr{S}\to\mathscr{S}\Modcat$.

In general, it may happen that $\mathscr{S}$ has a bounded $t$-structure, but $\mathfrak{S}(\mathscr{S})$ has no bounded $t$-structure.

\begin{Bsp}
Let $R$ be a left noetherian ring and let $\mathscr{S}=\Db{R\modcat}\opp$.
By {\rm \cite[Proposition 0.15(ii)]{Neeman2}}, there exists a good metric on $\mathscr{S}$ such that the completion $\mathfrak{S}(\mathscr{S})$ of $\mathscr{S}$ with respect to the metric is equivalent to $\Kb{\prj{R}}\opp$. Note that $\mathscr{S}$ has an obvious bounded $t$-structure and $\Kb{\prj{R}}\opp\simeq \Kb{\prj{R\opp}}$ as triangulated categories. By {\rm \cite[Theorem 1.2]{HS}}, if $R$ is commutative, singular and has finite Krull dimension, then $\Kb{\prj{R}}$ has no bounded $t$-structure. In this case, $\mathfrak{S}(\mathscr{S})$ has no bounded $t$-structure. For further examples of noetherian schemes and noncommutative rings, we refer to  \cite[page 18]{Neeman6} and Corollary \ref{Ring case}, respectively.
\end{Bsp}

To lift $t$-structures from $\mathscr{S}$ to $\mathfrak{S}(\mathscr{S})$, we introduce a special class of $t$-structures.

\begin{Def} \label{ETSC}
A $t$-structure $(\mathscr{S}^{\leqslant 0}, \mathscr{S}^{\geqslant 0})$ on $\mathscr{S}$ is \emph{extendable} (with respect to the metric $\mathscr{M}$) if there exists a natural number $n$ such that $\mathscr{M}_n\subseteq\mathscr{S}^{\leqslant 0}$. This is equivalent to saying that $\mathscr{M}$ is finer than the good metric $\{\mathscr{S}^{\leqslant -n}\}_{n\in\mathbb{N}}$ on $\mathscr{S}$.
\end{Def}

Obviously, any $t$-structure is extendable with respect to the good metric induced from its aisle. Lemma \ref{tool} below and Theorem \ref{ET} after it contain more examples and properties of extendable $t$-structures.

\begin{Lem}\label{tool}
$(1)$ Let $\mathscr{S}_i:=(\mathscr{S}^{\leqslant 0}_i, \mathscr{S}^{\geqslant 0}_i)$ for $i=1,2$ be equivalent $t$-structures on $\mathscr{S}$. Then $\mathscr{S}_1$ is extendable if and only if so is $\mathscr{S}_2$. In particular, if $\mathscr{S}_1$ is extendable, then so is $\mathscr{S}_1[j]$
for any $j\in \mathbb{Z}$.

$(2)$ Suppose that $(\mathscr{S}^{\leqslant 0}, \mathscr{S}^{\geqslant 0})$ is an extendable $t$-structure on $\mathscr{S}$. Then $\mathscr{S}^{+}\subseteq \mathscr{S}(\mathscr{M})$. In particular, if the $t$-structure is bounded below, then the metric $\mathscr{M}$ on $\mathscr{S}$ is embeddable.

$(3)$ Let $G$ be an object of $\mathscr{S}$. Then any bounded above t-structure on $\mathscr{S}$ is extendable with respect to a $G$-good metric on $\mathscr{S}$ (see Definition \ref{PEC}).
\end{Lem}

\begin{proof}
$(1)$ Suppose that  $\mathscr{S}_{1}$ is extendable. Then there exists a natural number $s$ such that $\mathscr{M}_s\subseteq\mathscr{S}_1^{\leqslant 0}$. Since $\mathscr{S}_1$ and $\mathscr{S}_2$ are equivalent, there exists a positive integer $t$ such that $\mathscr{S}_1^{\leqslant 0}\subseteq \mathscr{S}_2^{\leqslant t}$. This implies $\mathscr{M}_s\subseteq\mathscr{S}_2^{\leqslant t}=\mathscr{S}_2^{\leqslant 0}[-t]$, and therefore $\mathscr{M}_s[t]\subseteq \mathscr{S}_2^{\leqslant 0}$.
It follows from $\mathscr{M}_{s+t}\subseteq \mathscr{M}_s[t]$ that $\mathscr{M}_{s+t}\subseteq \mathscr{S}_2^{\leqslant 0}$. Thus $\mathscr{S}_{2}$ is extendable.

$(2)$ Since $(\mathscr{S}^{\leqslant 0}, \mathscr{S}^{\geqslant 0})$ is extendable, there exists a natural number $n$ such that $\mathscr{M}_n\subseteq\mathscr{S}^{\leqslant 0}$. Then $\mathscr{S}^{\geqslant 1}=(\mathscr{S}^{\leqslant 0})^{\perp}\subseteq{\mathscr{M}_n}^{\perp}$.
Since $\mathscr{S}(\mathscr{M})\subseteq\mathscr{S}$ is a  triangulated subcategory, $\mathscr{S}^{+}\subseteq \mathscr{S}(\mathscr{M})$. If $(\mathscr{S}^{\leqslant 0}, \mathscr{S}^{\geqslant 0})$ is bounded below, then $\mathscr{S}=\mathscr{S}^{+}$ and thus $\mathscr{S}=\mathscr{S}(\mathscr{M})$.
This means that $\mathscr{M}$ is embeddable.

$(3)$ Let $\mathscr{D}:=(\mathscr{S}^{\leqslant 0},\mathscr{S}^{\geqslant 0})$ be a bounded above $t$-structure on $\mathscr{S}$. Then there is a positive integer $r$ with $G\in\mathscr{S}^{\leqslant r}$.
Since $\mathscr{S}^{\leqslant r}\subseteq\mathscr{S}$ is closed under extensions, positive shifts and direct summands, $\langle G\rangle^{(-\infty,0]}\subseteq\mathscr{S}^{\leqslant r}$.
This implies $\mathscr{M}_r:=\langle G\rangle^{(-\infty,-r]}\subseteq\mathscr{S}^{\leqslant 0}$, and therefore $\mathscr{D}$ is extendable.
\end{proof}

\begin{Rem}\label{bounded above extendable} Although not all extendable $t$-structures have to be bounded above, if a $t$-structure $(\mathscr{S}^{\leqslant 0}, \mathscr{S}^{\geqslant 0})$ on $\mathscr{S}$ is extendable with respect to a $G$-good metric $\{\langle G \rangle ^{(-\infty, -n]}\}_{n\in\mathbb{N}}$ where $G$ is a classical generator of $\mathscr{S}$, then $(\mathscr{S}^{\leqslant 0}, \mathscr{S}^{\geqslant 0})$ is bounded above.
\end{Rem}

Our main result on lifting $t$-structures along completions of triangulated categories is the following which also relates to \cite[Problem 7.5]{Neeman6}.

\begin{Theo}\label{ET}
Let $(\mathscr{S}^{\leqslant 0}, \mathscr{S}^{\geqslant 0})$ be an extendable $t$-structure on $\mathscr{S}$ with respect to a good metric $\mathscr{M}$.  Then the following statements are true.

$(1)$ The pair $(\mathfrak{S}(\mathscr{S}^{\leqslant 0}), \mathfrak{S}(\mathscr{S}^{\geqslant 0}))$ is a $t$-structure on $\mathfrak{S}(\mathscr{S})$ with the heart given by $\mathfrak{y}(\mathscr{H})$, where $\mathscr{H}$ denotes the heart of $(\mathscr{S}^{\leqslant 0}, \mathscr{S}^{\geqslant 0})$.

$(2)$ Let $\mathscr{R}$ be any full subcategory of $\mathscr{S}^{\geqslant 0}$. Then the restriction of the Yoneda functor $\mathfrak{y}:\mathscr{S}\to\mathscr{S}\Modcat$ to $\mathscr{R}$ yields an equivalence $\mathfrak{y}|_{\mathscr{R}}:\mathscr{R}\lraf{\simeq}\mathfrak{S}(\mathscr{R})$ of additive categories. Moreover,
$\mathfrak{y}|_{\mathscr{H}}:\mathscr{H}\to\mathfrak{S}(\mathscr{H})$ is an equivalence of abelian categories.

$(3)$ If the $t$-structure $(\mathscr{S}^{\leqslant 0}, \mathscr{S}^{\geqslant 0})$ is bounded above, then so is the $t$-structure $(\mathfrak{S}(\mathscr{S}^{\leqslant 0}), \mathfrak{S}(\mathscr{S}^{\geqslant 0}))$.

$(4)$ If the $t$-structure  $(\mathscr{S}^{\leqslant 0}, \mathscr{S}^{\geqslant 0})$ is bounded below, then the functor $\mathfrak{y}$ restricts to a fully faithful triangle functor $\mathscr{S} \rightarrow \mathfrak{S}(\mathscr{S})$ satisfying that $\mathfrak{y}(\mathscr{S}^{\leqslant 0}) \subseteq \mathfrak{S}(\mathscr{S}^{\leqslant 0})$ and $\mathfrak{y}(\mathscr{S}^{\geqslant 0})=\mathfrak{S}(\mathscr{S}^{\geqslant 0})$.
\end{Theo}

\begin{proof} Let $\mathscr{S}_1:=(\mathscr{S}^{\leqslant 0}, \mathscr{S}^{\geqslant 0})$ and $\mathfrak{S}(\mathscr{S}_1):=(\mathfrak{S}(\mathscr{S}^{\leqslant 0}), \mathfrak{S}(\mathscr{S}^{\geqslant 0}))$.

$(1)$ We first show that $\mathfrak{S}(\mathscr{S}_1)$ satisfies (T1) and (T2) in Definition \ref{DTS}.

By Lemma \ref{Properties}(1) and by (T1) for $\mathscr{S}_1$, we have $\Sigma(\mathfrak{S}(\mathscr{S}^{\leqslant 0}))=\mathfrak{S}(\mathscr{S}^{\leqslant 0}[1])\subseteq \mathfrak{S}(\mathscr{S}^{\leqslant 0})$ and
$\mathfrak{S}(\mathscr{S}^{\geqslant 0})\subseteq\mathfrak{S}(\mathscr{S}^{\geqslant 0}[1])=\Sigma(\mathfrak{S}(\mathscr{S}^{\geqslant 0})).$ This verifies (T1) for $\mathfrak{S}(\mathscr{S}_1)$.

Let $\{X_{\bullet}\}$ and $\{Y_{\bullet}\}$ be Cauchy sequences in $\mathscr{S}$ with $X_n\in\mathscr{S}^{\leqslant -1}$ and $Y_n\in\mathscr{S}^{\geqslant 0}$ for $n\in\mathbb{N}$. Recall that $\Hom_{\mathscr{S}}(\mathscr{S}^{\leqslant -1},\mathscr{S}^{\geqslant 0})=0$ by (T2) for  $\mathscr{S}_1$. Then $$\Hom_{\mathscr{S}\Modcat}(\mathop{\text{colim}}\limits_{i\longrightarrow}\mathfrak{y}(X_i), \mathop{\text{colim}}\limits_{j\longrightarrow}\mathfrak{y}(Y_j))\simeq \mathop{\text{lim}}\limits_{\longleftarrow i}\mathop{\text{colim}}\limits_{j\longrightarrow}\Hom_{\mathscr{S}}(X_i, Y_j)=0.$$
This implies $\Hom_{\mathscr{S}\Modcat}(\mathfrak{L}(\mathscr{S}^{\leqslant -1}), \mathfrak{L}(\mathscr{S}^{\geqslant 0}))=0.$ Clearly, $\mathfrak{S}(\mathscr{S}^{\geqslant 0})\subseteq\mathfrak{L}(\mathscr{S}^{\geqslant 0}))$ and $\Sigma(\mathfrak{S}(\mathscr{S}^{\leqslant 0}))\subseteq\Sigma(\mathfrak{L}(\mathscr{S}^{\leqslant 0}))=\mathfrak{L}(\mathscr{S}^{\leqslant -1})$ by Lemma \ref{Properties}(1). It follows that $\Hom_{\mathscr{S}\Modcat}(\Sigma(\mathfrak{S}(\mathscr{S}^{\leqslant 0})), \mathfrak{S}(\mathscr{S}^{\geqslant 0}))=0$. Thus (T2) holds for $\mathfrak{S}(\mathscr{S}_1)$.

In the following, we verify (T3) in Definition \ref{DTS} for $\mathfrak{S}(\mathscr{S}_1)$.

Let $A\in\mathfrak{L}(\mathscr{S})$. There exists a Cauchy sequence $\{A_{\bullet}, f_{\bullet}\}$ in $\mathscr{S}$ such that $A\simeq\mathop{\text{colim}}\limits_{\longrightarrow}\mathfrak{y}(A_{n})$. By (T3) for $\mathscr{S}_1$, we obtain two chains $\{A_{\bullet}^{\leqslant -1}, f_{\bullet}^{\leqslant -1}\}$ and $\{A_{\bullet}^{\geqslant 0}, f_{\bullet}^{\geqslant 0}\}$ in $\mathscr{S}$ with $A_n^{\leqslant -1}\in \mathscr{S}^{\leqslant -1}$ and $A_n^{\geqslant 0}\in \mathscr{S}^{\geqslant 0}$ for $n\in\mathbb{N}$ and with commutative squares of morphisms:
$$
\xymatrix{
A_n^{\leqslant -1}\ar[r]\ar[d]_-{f_{n+1}^{\leqslant -1}}& A_n\ar[r]\ar[d]_-{f_{n+1}}& A_n^{\geqslant 0}\ar[r]\ar[d]_-{f_{n+1}^{\geqslant 0}}& A_n^{\leqslant -1}[1]\ar[d]_-{f_{n+1}^{\leqslant -1}[1]}\\
A_{n+1}^{\leqslant -1}\ar[r]& A_{n+1}\ar[r]& A_{n+1}^{\geqslant 0}\ar[r]& A_{n+1}^{\leqslant -1}[1].}
$$
In this diagram, if two arrows of triangles are given, then $f_{n+1}^{\leqslant -1}$ and $f_{n+1}^{\geqslant 0}$ exist uniquely by (T1) and (T2) in Definition \ref{$T$-structure}. Taking first the functor $\mathfrak{y}$ and then colimts, we obtain a long exact sequence in $\mathscr{S}\Modcat$:
\[(\ast)\quad
\mathop{\text{colim}}\limits_{\longrightarrow}\mathfrak{y}(A_{n}^{\leqslant -1}) \longrightarrow \mathop{\text{colim}}\limits_{\longrightarrow}\mathfrak{y}(A_{n}) \longrightarrow\mathop{\text{colim}}\limits_{\longrightarrow}\mathfrak{y}(A_{n}^{\geqslant 0}) \longrightarrow \Sigma(\mathop{\text{colim}}\limits_{\longrightarrow}\mathfrak{y}(A_{n}^{\leqslant -1})).
\]
Our next aim is to prove that both $\{A_{\bullet}^{\leqslant -1}, f_{\bullet}^{\leqslant -1}\}$ and $\{A_{\bullet}^{\geqslant 0}, f_{\bullet}^{\geqslant 0}\}$ are Cauchy sequences in $\mathscr{S}$.

Since $\mathscr{S}_1$ is extendable (with respect to $\mathscr{M}$),  there exists a natural number $k$ such that $\mathscr{M}_k\subseteq\mathscr{S}^{\leqslant 0}$. Define $C_n:=\text{Cone}(f_{n+1})$. By the $3\times3$ lemma of triangles in triangulated category, there exist dotted morphisms making the following diagram of triangles in $\mathscr{S}$ commutative:
\[\begin{tikzcd}
{A_{n}^{\geqslant 0}[-1]} \arrow[d] \arrow[r, dotted]                  & {A_{n+1}^{\geqslant 0}[-1]} \arrow[d] \arrow[r, dotted] & {C''_{n}[-1]} \arrow[d, "a_{n}", dotted]    &      \\
A_{n}^{\leqslant -1} \arrow[r, "f_{n+1}^{\leqslant -1}"] \arrow[d] & A_{n+1}^{\leqslant -1} \arrow[r] \arrow[d]       & C'_{n} \arrow[d, "b_{n}", dotted] \arrow[r] & {A_{n}^{\leqslant -1}[1]} \arrow[d] \\
A_{n} \arrow[r, "f_{n+1}"] \arrow[d]                             & A_{n+1} \arrow[r] \arrow[d]                     & C_{n} \arrow[d, "c_{n}", dotted] \arrow[r]  & {A_{n}[1]} \arrow[d]               \\
A_{n}^{\geqslant 0} \arrow[r, "d_{n+1}", dotted]           & A_{n+1}^{\geqslant 0} \arrow[r, dotted]                 & C''_{n} \arrow[r, dotted]                   & {A_{n}^{\geqslant 0}[1]}
\end{tikzcd}\]
Since $\mathscr{S}_1$ is a $t$-structure, $d_{n+1}$ is determined by $f_{n+1}$, that is, $d_{n+1}=f_{n+1}^{\geqslant 0}$. Note that $A_{n+1}^{\leqslant -1}\in\mathscr{S}^{\leqslant -1}$, $A_n^{\leqslant -1}[1]\in\mathscr{S}^{\leqslant -1}[1]\subseteq \mathscr{S}^{\leqslant -1}$ and $\mathscr{S}^{\leqslant -1}$
is closed under extensions in $\mathscr{S}$. This forces $C'_n \in \mathscr{S}^{\leqslant -1}$.
Similarly, $C''_n \in \mathscr{S}^{\geqslant -1}$ due to
$A_{n}^{\geqslant 0},  A_{n+1}^{\geqslant 0}\in\mathscr{S}^{\geqslant 0}$.

By the Cauchy sequence$\{A_{\bullet}, f_{\bullet}\}$, for each $p\in\mathbb{N}$, there exists a nonnegative integer $N_p$ such that $C_n\in\mathscr{M}_p$ for all $n\geqslant N_p$. We consider
$p\geqslant k+2$ and $n\geqslant N_p$. Then $C_n \in\mathscr{M}_p\subseteq \mathscr{M}_{k+2}\subseteq\mathscr{M}_k[2]\subseteq \mathscr{S}^{\leqslant 0}[2]=\mathscr{S}^{\leqslant -2}$. It follows from $\Hom_\mathscr{S}(\mathscr{S}^{\leqslant -2}, \mathscr{S}^{\geqslant -1})=0$ that $\Hom_\mathscr{S}(C_n, C''_n )=0$. This yields $c_n=0$ and therefore
$C_n'\simeq C_n\oplus C''_n[-1]$. Note that $C'_n \in \mathscr{S}^{\leqslant -1}$, $C''_n[-1]\in\mathscr{S}^{\geqslant 0}$ and  $\Hom_\mathscr{S}(\mathscr{S}^{\leqslant -1}, \mathscr{S}^{\geqslant 0})=0$.  Thus $\Hom_\mathscr{S}(C_n', C_n''[-1])=0$. Consequently,
$C_n''=0$, and $b_n$ and $f_{n+1}^{\geqslant 0}$ are isomorphisms.
Now, it is clear that $\{A_{\bullet}^{\leqslant -1}\}$ and $\{A_{\bullet}^{\geqslant 0}\}$ are Cauchy sequences in $\mathscr{S}$, $\mathop{\text{colim}}\limits_{\longrightarrow}\mathfrak{y}(A_n^{\leqslant -1})\in\Sigma(\mathfrak{L}(\mathscr{S}^{\leqslant 0}))$ and
$\mathop{\text{colim}}\limits_{\longrightarrow}\mathfrak{y}(A_n^{\geqslant 0})
\in\mathfrak{L}(\mathscr{S}^{\geqslant 0})$.  Let $m:=N_{k+2}$. Since $C_j''=0$ for each $j\geqslant m$, the sequence
$\{A_{\bullet}^{\geqslant 0}\}$ is stable (see Definition \ref{definition of C.S.}) and there is an isomorphism
$$
(\ast\ast)\quad \mathop{\text{colim}}\limits_{\longrightarrow}\mathfrak{y}(A_n^{\geqslant 0})\simeq \mathfrak{y}(A_j^{\geqslant 0})\in\mathfrak{y}(\mathscr{S}^{\geqslant 0}).
$$
\noindent Further, we claim that if $A\in\mathfrak{S}(\mathscr{S})$, then $\mathop{\text{colim}}\limits_{\longrightarrow}\mathfrak{y}(A_{n}^{\leqslant -1})\in \Sigma(\mathfrak{S}(\mathscr{S}^{\leqslant 0}))$ and
$ \mathop{\text{colim}}\limits_{\longrightarrow}\mathfrak{y}(A_n^{\geqslant 0})\in\mathfrak{S}(\mathscr{S}^{\geqslant 0})$.

Since $\mathscr{M}_{k+1}\subseteq \mathscr{M}_k[1]\subseteq \mathscr{S}^{\leqslant -1}$ and $(\mathscr{S}^{\leqslant -1})^{\perp}=\mathscr{S}^{\geqslant 0}$, we have $\mathscr{S}^{\geqslant 0}\subseteq{\mathscr{M}_{k+1}}^{\perp}$. This implies that $\mathfrak{y}(\mathscr{S}^{\geqslant 0}))\subseteq \mathfrak{C}(\mathscr{S})$ and further $\mathfrak{y}(\mathscr{S}^{\geqslant 0})\subseteq\mathfrak{S}(\mathscr{S}^{\geqslant 0})$. Together with $(\ast\ast)$,  $\mathop{\text{colim}}\limits_{\longrightarrow}\mathfrak{y}(A_n^{\geqslant 0})\simeq \mathfrak{y}(A_m^{\geqslant 0})\in\mathfrak{S}(\mathscr{S}^{\geqslant 0})$.
Moreover, the evaluation of the sequence $(\ast)$ at an object $S\in\mathscr{S}$ yields a long exact sequence:
$$
\Hom_\mathscr{S}(S,A_m^{\geqslant 0}[-1])\longrightarrow
\mathop{\text{colim}}\limits_{\longrightarrow}\Hom_\mathscr{S}(S, A_{n}^{\leqslant -1})\longrightarrow \mathop{\text{colim}}\limits_{\longrightarrow}\Hom_\mathscr{S}(S, A_{n}) \longrightarrow\Hom_\mathscr{S}(S, A_m^{\geqslant 0}).
$$
Since $\mathscr{S}^{\geqslant 0}\subseteq{\mathscr{M}_{k+1}}^{\perp}$ and $A_m^{\geqslant 0}\in \mathscr{S}^{\geqslant 0}\supseteq \mathscr{S}^{\geqslant 1}$, we obtain $\Hom_\mathscr{S}(S,A_m^{\geqslant 0}[-1])=0=\Hom_\mathscr{S}(S, A_m^{\geqslant 0})$ for any $S\in \mathscr{M}_{k+1}$. In this case, $\mathop{\text{colim}}\limits_{\longrightarrow}\Hom_\mathscr{S}(S, A_{n}^{\leqslant -1})\simeq \mathop{\text{colim}}\limits_{\longrightarrow}\Hom_\mathscr{S}(S, A_{n}).$
It follows that $A\in \mathfrak{C}(\mathscr{S})$ if and only if $\mathop{\text{colim}}\limits_{\longrightarrow}\mathfrak{y}(A_n^{\leqslant -1})\in \mathfrak{C}(\mathscr{S})$; equivalently, $A\in\mathfrak{S}(\mathscr{S})$ if and only if $\mathop{\text{colim}}\limits_{\longrightarrow}\mathfrak{y}(A_n^{\leqslant -1})\in\Sigma(\mathfrak{S}(\mathscr{S}^{\leqslant 0}))$.
This shows the claim. Thus, the sequence $(\ast)$ for each $A\in\mathfrak{S}(\mathscr{S})$ is required in (T3) of Definition \ref{DTS} for $\mathfrak{S}(\mathscr{S}_1)$. This shows that $\mathfrak{S}(\mathscr{S}_1)$ is a $t$-structure on $\mathfrak{S}(\mathscr{S})$.

By $(\ast)$ and $(\ast\ast)$, each object of $\mathfrak{S}(\mathscr{S}^{\geqslant 0})$ is isomorphic to $\mathfrak{y}(Y)$ for some object $Y\in \mathscr{S}^{\geqslant 0}$. This implies that $\mathfrak{y}(\mathscr{S}^{\geqslant 0})=\mathfrak{S}(\mathscr{S}^{\geqslant 0})$ and the functor $\mathfrak{y}$ restricts to an equivalence
$\mathfrak{y}_1:\mathscr{S}^{\geqslant 0}\to\mathfrak{S}(\mathscr{S}^{\geqslant 0}).$
Let $\mathscr{A}:=\mathfrak{S}(\mathscr{S}^{\leqslant 0})\cap\mathfrak{S}(\mathscr{S}^{\geqslant 0})$, the heart of the $t$-structure $\mathfrak{S}(\mathscr{S}_1)$. Then $
\mathscr{A}=\mathfrak{L}(\mathscr{S}^{\leqslant 0})\cap\mathfrak{C}(\mathscr{S})\cap\mathfrak{L}(\mathscr{S}^{\geqslant 0})$ and $ \mathfrak{S}(\mathscr{H})=\mathfrak{L}(\mathscr{H})\cap\mathfrak{C}(\mathscr{S}).$
Since $\mathscr{H}=\mathscr{S}^{\leqslant 0}\cap\mathscr{S}^{\geqslant 0}$ and $\mathfrak{y}(\mathscr{S}^{\geqslant 0})=\mathfrak{S}(\mathscr{S}^{\geqslant 0})\subseteq\mathfrak{C}(\mathscr{S})$, we have $\mathfrak{S}(\mathscr{H})\subseteq\mathscr{A}$ and $\mathfrak{y}(\mathscr{H})\subseteq\mathfrak{S}(\mathscr{H})$. To show $\mathscr{A}\subseteq\mathfrak{S}(\mathscr{H})$, we now take an object $F\in\mathscr{A}$. Since $\mathfrak{y}_1$ is an equivalence, there exists an object $M\in \mathscr{S}^{\geqslant 0}$ such that $F\simeq \mathfrak{y}_1(M)=\mathfrak{y}(M)$. Further, by $F\in\mathfrak{S}(\mathscr{S}^{\leqslant 0})$, we can find a Cauchy sequence $\{M_{\bullet}\}$ in $\mathscr{S}$ with $M_{n}\in \mathscr{S}^{\leqslant 0}$ for all $n\in\mathbb{N}$, such that $F\simeq \mathop{\text{colim}}\limits_{\rightarrow}\mathfrak{y}(M_n)$.
Consequently, there exists an isomorphism $\theta:\mathfrak{y}(M)\to\mathop{\text{colim}}\limits_{\rightarrow}\mathfrak{y}(M_n)$. Since filtered colimits of objects in $\mathscr{S}\Modcat$ are calculated pointwisely, the Hom-functor $\Hom_{\mathscr{S}\Modcat}(\mathfrak{y}(X),-)$ for each $X\in\mathscr{S}$ commutes with filtered colimits by the Yoneda lemma. Thus $\theta$ is the composition of a morphism $\theta':\mathfrak{y}(M)\to\mathfrak{y}(M_n)$ with the canonical morphism $\mathfrak{y}(M_n)\to \mathop{\text{colim}}\limits_{\rightarrow}\mathfrak{y}(M_n)$ for some positive integer $n$. Clearly, $\theta'$ is a split monomorphism. Still by the Yoneda lemma, $M$ is a direct summand of $M_n$. Since $\mathscr{S}^{\leqslant 0}\subseteq\mathscr{S}$ is closed under direct summands, $M\in \mathscr{S}^{\leqslant 0}$, forcing $M\in\mathscr{H}$. This shows $\mathscr{A}\subseteq \mathfrak{y}(\mathscr{H})$. Up to now, we have
$\mathfrak{S}(\mathscr{H})\subseteq \mathscr{A}\subseteq\mathfrak{y}(\mathscr{H})\subseteq\mathfrak{S}(\mathscr{H}).$ Thus $\mathscr{A}=\mathfrak{S}(\mathscr{H})=\mathfrak{y}(\mathscr{H})$ and the functor $\mathfrak{y}$ (and also $\mathfrak{y}_1$) restricts to an equivalence
$\mathfrak{y}_0:\mathscr{H}\to\mathfrak{S}(\mathscr{H})$. The exactness of $\mathfrak{y}_0$ can be checked directly following the definitions of the exact structures of $\mathscr{H}$ and $\mathfrak{S}(\mathscr{H})$ as well as the triangles of $\mathfrak{S}(\mathscr{S})$ defined in Definition \ref{Completion}.

$(2)$ Clearly, $\mathfrak{y}(\mathscr{R})\subseteq \mathfrak{S}(\mathscr{R})$ if and only if $\mathfrak{y}(\mathscr{R})\subseteq \mathfrak{C}(\mathscr{S})$. Since $\mathscr{R}\subseteq\mathscr{S}^{\geqslant 0}$ and $\mathfrak{y}(\mathscr{S}^{\geqslant 0}))\subseteq \mathfrak{C}(\mathscr{S})$, we have $\mathfrak{y}(\mathscr{R})\subseteq \mathfrak{S}(\mathscr{R})$.
By $\mathscr{S}^{\geqslant 0}\subseteq{\mathscr{M}_{k+1}}^{\perp}$, it follows from  Lemma \ref{stable lemma} that each Cauchy sequence in $\mathscr{S}$ with all terms in $\mathscr{S}^{\geqslant 0}$
is stable. This implies $\mathfrak{L}(\mathscr{R})\subseteq\mathfrak{y}(\mathscr{R})$, and therefore $\mathfrak{S}(\mathscr{R})\subseteq\mathfrak{y}(\mathscr{R})$. Thus $\mathfrak{y}(\mathscr{R})= \mathfrak{S}(\mathscr{R})$. Since $\mathfrak{y}$ is fully faithful, its restriction $\mathfrak{y}|_{\mathscr{R}}:\mathscr{R}\rightarrow \mathfrak{S}(\mathscr{R})$ to $\mathscr{R}$
is an equivalence.

$(3)$ Assume that $\mathscr{S}_1$ is bounded above, that is, $\mathscr{S}_1^{-}=\mathscr{S}$. To show that
$\mathfrak{S}(\mathscr{S}_1)$ is bounded above, it suffices to show that   $\mathfrak{S}(\mathscr{S}^{\geqslant 0}) \subseteq\mathfrak{S}(\mathscr{S}_1)^{-}$. Let $U\in \mathfrak{S}(\mathscr{S}^{\geqslant 0})$. By $(2)$, there is an object $V\in\mathscr{S}^{\geqslant 0}$ such that $U\simeq\mathfrak{y}(V)$. Since $\mathscr{S}_1$ is bounded above, there exists a positive integer $n$ such that $V\in \mathscr{S}^{\leqslant n}$. Then
\[
U\simeq \mathfrak{y}(V) \in \mathfrak{S}(\mathscr{S})\cap\mathfrak{L}(\mathscr{S}^{\leqslant n})=  \mathfrak{S}(\mathscr{S}^{\leqslant n})=\mathfrak{S}(\mathscr{S}^{\leqslant 0}[-n]) =\Sigma^{-n}(\mathfrak{S}(\mathscr{S}^{\leqslant 0})) \subseteq \mathfrak{S}(\mathscr{S})^{-}.
\]

$(4)$ Assume that $\mathscr{S}_1$ is bounded below, that is, $\mathscr{S}_1^{+}=\mathscr{S}$. By Lemma \ref{tool}(2), the metric $\mathscr{M}$ on $\mathscr{S}$ is embeddable (see Definition \ref{Completion}). Now, $(4)$ follows from $(2)$ and Lemma \ref{Properties}(3).
\end{proof}

A useful consequence of Theorem \ref{ET} is the following result. This and Lemma \ref{Restriction} will play an important role in the proof of Theorem \ref{Main result}(1) (see also Theorem \ref{BABB}).

\begin{Koro}\label{Unexpected}
Let $(\mathscr{S}^{\leqslant 0}, \mathscr{S}^{\geqslant 0})$ be a $t$-structure on $\mathscr{S}$.
Suppose that the good metric $\mathscr{M}$ is equivalent to the good metric $\{\mathscr{S}^{\leqslant -n}\}_{n\in \mathbb{N}}$ on $\mathscr{S}$. Then the restriction of the Yoneda functor $\mathfrak{y}:\mathscr{S}\to\mathscr{S}\Modcat$ to $\mathscr{S}^{+}:=\bigcup_{n\in \mathbb{N}} \mathscr{S}^{\geqslant -n}$ yields an equivalence of triangulated categories: $$\mathscr{S}^{+}\lraf{\simeq}\mathfrak{S}(\mathscr{S}).$$
In particular, if the $t$-structure $(\mathscr{S}^{\leqslant 0}, \mathscr{S}^{\geqslant 0})$ is bounded below, then $\mathscr{S}$ and $\mathfrak{S}(\mathscr{S})$ are equivalent as triangulated categories.
\end{Koro}

\begin{proof}
Since equivalent good metrics on $\mathscr{S}$ produce the same completion, we can assume $\mathscr{M}_n=\mathscr{S}^{\leqslant -n}$ for $n\in\mathbb{N}$.
Clearly, $\mathscr{S}(\mathscr{M}):=\bigcup_{n\in\mathbb{N}}\mathscr{M}_n^{\perp}=\bigcup_{n\in \mathbb{N}} \mathscr{S}^{\geqslant -n+1}=\mathscr{S}^+$.
By Lemma \ref{Properties}(3), the functor $\mathfrak{y}$ restricts to a fully faithful triangle functor $\mathfrak{y}':\mathscr{S}^+\to \mathfrak{S}(\mathscr{S})$. It suffices to show that $\mathfrak{y}'$ is dense.

Let $\mathscr{D}:=(\mathscr{S}^{\leqslant 0}, \mathscr{S}^{\geqslant 0})$. Then $\mathscr{D}$ is extendable with respect to $\mathscr{M}$, due to Definition \ref{ETSC}. By Theorem \ref{ET}(1)(2), the pair
$\mathfrak{S}(\mathscr{D}):=\big(\mathfrak{S}(\mathscr{S}^{\leqslant 0}), \mathfrak{S}(\mathscr{S}^{\geqslant 0})\big)$ is a $t$-structure on $\mathfrak{S}(\mathscr{S})$ and the functor $\mathfrak{y}$ restricts to an additive equivalence
$\mathscr{S}^{\geqslant 0}\to\mathfrak{S}(\mathscr{S}^{\geqslant 0})$. Moreover, for
any $X\in\mathscr{S}$ and $j\in\mathbb{Z}$, we have $\Sigma^j(\mathfrak{y}(X))\simeq\mathfrak{y}(X[j])$, and $\mathfrak{S}(\mathscr{S}^{\geqslant -j})=\Sigma^{j}(\mathfrak{S}(\mathscr{S}^{\geqslant 0}))$ by Lemma \ref{Properties}(1). It follows that the functor $\mathfrak{y}$ restricts to a series of additive equivalences $\mathscr{S}^{\geqslant j}\to\mathfrak{S}(\mathscr{S}^{\geqslant j})$ for all $j$.
Thus, to show the denseness of $\mathfrak{y}'$, we only need to show that $\mathfrak{S}(\mathscr{S})=\bigcup_{n\in \mathbb{N}}\mathfrak{S}(\mathscr{S}^{\geqslant -n})$, that is, $\mathfrak{S}(\mathscr{D})$ is bounded below.

For this aim, we take $F\in \mathfrak{S}(\mathscr{S})$. Recall from Definition \ref{Completion}(3) that $\mathfrak{S}(\mathscr{S}):=\mathfrak{L}(\mathscr{S}) \cap \mathfrak{C}(\mathscr{S})$. Then there exists a nonnegative integer $n$ with $F(\mathscr{S}^{\leqslant -n})=0$. In other words, $\Hom_{\mathscr{S}\Modcat}\big(\mathfrak{y}(\mathscr{S}^{\leqslant -n}), F\big)=0$.
Note that, by Definition \ref{Completion}(1), each object of $\mathfrak{S}(\mathscr{S}^{\leqslant -n})$ is a colimit in $\mathscr{S}\Modcat$ of objects that belong to
$\mathfrak{y}\big(\mathscr{S}^{\leqslant -n}\big)$. Since $\Hom_{\mathscr{S}\Modcat}(-, F)$ sends colimits to limits, $\Hom_{\mathscr{S}\Modcat}(\mathfrak{S}\big(\mathscr{S}^{\leqslant -n}), F\big)=0$.
It follows that $F\in \mathfrak{S}(\mathscr{S}^{\leqslant -n})^{\bot}\cap \mathfrak{S}(\mathscr{S})$, where
the right orthogonal subcategory associated to a full subcategory of $\mathscr{S}\Modcat$ is calculated in $\mathscr{S}\Modcat$.
Since $\mathfrak{S}(\mathscr{D})$ is a $t$-structure on $\mathfrak{S}(\mathscr{S})$ and $\mathfrak{S}(\mathscr{S}^{\leqslant -n})=\Sigma^n\big(\mathfrak{S}(\mathscr{S}^{\leqslant 0})\big)$, we have $\mathfrak{S}(\mathscr{S}^{\leqslant -n})^{\bot}\cap\mathfrak{S}(\mathscr{S})
=\Sigma^n\big(\mathfrak{S}(\mathscr{S}^{\geqslant 1})\big)=\mathfrak{S}(\mathscr{S}^{\geqslant -n+1})$.
Thus $F\in \mathfrak{S}(\mathscr{S}^{\geqslant -n+1})$. This shows that $\mathfrak{S}(\mathscr{D})$ is bounded below, and therefore the functor $\mathfrak{y}'$ is a triangle equivalence.
\end{proof}

\begin{Rem}
Just Corollary 3.6 alone can also be obtained with other methods - as done in a very recent preprint by Cummings and Gratz (see \cite[Theorem 1.1]{CG}). To compare the two methods, we would like to refer the reader to their interesting work on metric completions of discrete cluster categories.
\end{Rem}

Equivalences between $t$-structures can be characterized by the finiteness of the distance of
$t$-structures. Recall that the \emph{distance} between $t$-structures $\mathscr{S}_i:=(\mathscr{S}^{\leqslant 0}_i,\mathscr{S}^{\geqslant 0}_i)$ for $i=1,2$ on $\mathscr{S}$ is defined as
$$
d(\mathscr{S}_1,\mathscr{S}_2):=\text{inf}\{n_{2}-n_{1}\;|\; n_{1},n_{2} \in \mathbb{Z}\;\; \text{with}\;\; n_{1} \leqslant n_{2} \;\;\text{and}\;\;\mathscr{S}_{1}^{\leqslant n_{1}} \subseteq \mathscr{S}_{2}^{\leqslant 0} \subseteq \mathscr{S}_{1}^{\leqslant n_{2}}\}.
$$
Then $\mathscr{S}_1$ and $\mathscr{S}_2$ are equivalent if and only if $d(\mathscr{S}_1,\mathscr{S}_2)<\infty$. We denote by
$\mathcal{T}_{\mathscr{S}}(\mathscr{S}_1)$ the \emph{equivalence class of $t$-structures} on $\mathscr{S}$ containing $\mathscr{S}_1$, and define
$\mathfrak{S}(\mathscr{S}_1):=(\mathfrak{S}(\mathscr{S}^{\leqslant 0}_1),
\mathfrak{S}(\mathscr{S}^{\geqslant 0}_1)).$
For a $t$-structure $\mathscr{D}:=(\mathscr{D}^{\leqslant 0},\mathscr{D}^{\geqslant 0})$ on $\mathfrak{S}(\mathscr{S})$, we define $\mathfrak{y}^{-1}(\mathscr{D}):=\big(\mathfrak{y}^{-1}(\mathscr{D}^{\leqslant 0}),\mathfrak{y}^{-1}(\mathscr{D}^{\geqslant 0})\big)$, a pair of full subcategories of $\mathscr{S}$ consisting of objects $X$ such that $\mathfrak{y}(X)$ is in $\mathscr{D}^{\leqslant 0}$ and $\mathscr{D}^{\geqslant 0}$, respectively.

The following result conveys that lifting extendable $t$-structures from $\mathscr{S}$ to $\mathfrak{S}(\mathscr{S})$ preserves both the equivalence and the distance of $t$-structures.

\begin{Koro}
Let $\mathscr{S}_1:=(\mathscr{S}^{\leqslant 0}_1, \mathscr{S}^{\geqslant 0}_1)$ be an extendable $t$-structure on $\mathscr{S}$ with respect to a good metric $\mathscr{M}$. The following statements are true.

$(1)$ The map $\mathfrak{S}(-):\mathcal{T}_{\mathscr{S}}(\mathscr{S}_1)\to \mathcal{T}_{\mathfrak{S}(\mathscr{S})}(\mathfrak{S}(\mathscr{S}_1))$ is injective and
$d(\mathscr{S}_1, \mathscr{S}_2)=d(\mathfrak{S}(\mathscr{S}_1),\mathfrak{S}(\mathscr{S}_2))$ for any
$\mathscr{S}_{2}\in\mathcal{T}_{\mathscr{S}}(\mathscr{S}_1)$.

$(2)$ Suppose that the metric $\mathscr{M}$ is embeddable. Then the map $\mathfrak{S}(-)$ in $(1)$
is bijective and its inverse is given by
$\mathfrak{y}^{-1}(-):\mathcal{T}_{\mathfrak{S}(\mathscr{S})}(\mathfrak{S}(\mathscr{S}_1))
\to\mathcal{T}_{\mathscr{S}}(\mathscr{S}_1)$.
\end{Koro}

\begin{proof}
$(1)$ We first show that the map $\mathfrak{S}(-)$ is well defined.
In the proof, the formula will be used freely: $\mathfrak{S}(\mathscr{A}[j])=\Sigma^{j}(\mathfrak{S}(\mathscr{A}))$ for any subcategory $\mathscr{A}$ of $\mathscr{S}$ and for any $j\in\mathbb{Z}$ (see Lemma \ref{Properties}(1)).
We also set
$\mathfrak{S}(\mathscr{S}_1)^{\leqslant 0}:=\mathfrak{S}(\mathscr{S}^{\leqslant 0}_1)$ and
$\mathfrak{S}(\mathscr{S}_1)^{\geqslant 0}:=\mathfrak{S}(\mathscr{S}^{\geqslant 0}_1).
$

Let $\mathscr{S}_{2}\in \mathcal{T}_{\mathscr{S}}(\mathscr{S}_1)$. Then $\mathscr{S}_1$ and $\mathscr{S}_2$ are equivalent. Since $\mathscr{S}_1$ is extendable, it follows from Lemma \ref{tool}(1) that $\mathscr{S}_2$ is extendable. By Theorem \ref{ET}(1), $\mathfrak{S}(\mathscr{S}_1)$ and $\mathfrak{S}(\mathscr{S}_2)$ are $t$-structures on $\mathfrak{S}(\mathscr{S})$. Let $d:=d(\mathscr{S}_1, \mathscr{S}_2)$. Then there is an integer $g$ with $\mathscr{S}^{\geqslant d+g}_1\subseteq\mathscr{S}^{\geqslant 0}_2\subseteq \mathscr{S}^{\geqslant g}_1$. Applying $\mathfrak{S}(-)$ to the inclusions, we obtain $\mathfrak{S}(\mathscr{S}_1)^{\geqslant d+g}\subseteq\mathfrak{S}(\mathscr{S}_2)^{\geqslant 0}\subseteq \mathfrak{S}(\mathscr{S}_1)^{\geqslant g}$. This implies that $\mathfrak{S}(\mathscr{S}_1)$ and  $\mathfrak{S}(\mathscr{S}_2)$ are equivalent. Thus the map $\mathfrak{S}(-)$ is well defined. Moreover, by Theorem \ref{ET}(2), the Yoneda functor $\mathfrak{y}$ restricts to additive equivalences
$\mathscr{S}^{\geqslant n}_1\to \mathfrak{S}(\mathscr{S}_1)^{\geqslant n}$ and $\mathscr{S}^{\geqslant n}_2\to \mathfrak{S}(\mathscr{S}_2)^{\geqslant n}$ for $n\in\mathbb{N}$. It follows that $d=d(\mathfrak{S}(\mathscr{S}_1),\mathfrak{S}(\mathscr{S}_2))$.  Clearly, the map $\mathfrak{S}(-)$ is injective since each $t$-structure is determined by its coaisle.

$(2)$ We show that the map $\mathfrak{y}^{-1}(-)$ is well defined.

Since $\mathscr{M}$ is embeddable, $\mathfrak{y}$ restricts to a fully faithful triangle functor $\mathscr{S}\to\mathfrak{S}(\mathscr{S})$ by Lemma \ref{Properties}(3). In particular, $\mathfrak{y}(\mathscr{S})$ is full triangulated subcategory of $\mathfrak{S}(\mathscr{S})$.
Let $\mathscr{D}:=(\mathscr{D}^{\leqslant 0},\mathscr{D}^{\geqslant 0})\in \mathcal{T}_{\mathfrak{S}(\mathscr{S})}(\mathfrak{S}(\mathscr{S}_1))$. There exists an integer $m$ with $\mathscr{D}^{\geqslant 0}\subseteq \mathfrak{S}(\mathscr{S}_1)^{\geqslant m}.$
Since $\mathfrak{y}$ restricts to an equivalence
$\mathscr{S}^{\geqslant m}_1\to \mathfrak{S}(\mathscr{S}_1)^{\geqslant m}$ by Theorem \ref{ET}(2), we have $\mathfrak{y}(\mathscr{S}^{\geqslant m}_1)=\mathfrak{S}(\mathscr{S}_1)^{\geqslant m}$ and $\mathscr{S}^{\geqslant m}_1=\mathfrak{y}^{-1}(\mathfrak{S}(\mathscr{S}_1)^{\geqslant m})$. This forces $\mathscr{D}^{\geqslant 0}\subseteq \mathfrak{y}(\mathscr{S}^{\geqslant m}_1)\subseteq\mathfrak{y}(\mathscr{S})$. Let  $\mathscr{D}':=\big(\mathfrak{y}(\mathscr{S})\cap\mathscr{D}^{\leqslant 0},\, \mathscr{D}^{\geqslant 0}\big)$. Since $\mathscr{D}$ is a $t$-structure on $\mathfrak{S}(\mathscr{S})$ and $\mathscr{D}^{\geqslant 0}\subseteq\mathfrak{y}(\mathscr{S})$, it is easy to show that $\mathscr{D}'$ is a $t$-structure on $\mathfrak{y}(\mathscr{S})$.
Clearly,
$\mathfrak{y}^{-1}(\mathscr{D}')=\mathfrak{y}^{-1}(\mathscr{D})$. By the triangle equivalence $\mathscr{S}\simeq\mathfrak{y}(\mathscr{S})$, $\mathfrak{y}^{-1}(\mathscr{D})$ is a $t$-structure on $\mathscr{S}$. Since $\mathscr{D}^{\geqslant 0}\subseteq \mathfrak{y}(\mathscr{S}^{\geqslant m}_1)$, we see from Theorem \ref{ET}(2) that
$\mathfrak{y}^{-1}(\mathscr{D}^{\geqslant 0})\subseteq\mathfrak{y}^{-1}(\mathfrak{y}(\mathscr{S}^{\geqslant m}_1))=\mathscr{S}^{\geqslant m}_1$ and $\mathfrak{y}$ restricts to an equivalence $\mathfrak{y}^{-1}(\mathscr{D}^{\geqslant 0})\simeq\mathscr{D}^{\geqslant 0}$. Moreover, $\mathfrak{y}$ also restricts to an equivalence $\mathscr{S}^{\geqslant 0}_1
\simeq\mathfrak{S}(\mathscr{S}_1)^{\geqslant 0}$. Thus
$d(\mathscr{S}_1, \mathfrak{y}^{-1}(\mathscr{D}))=d(\mathfrak{S}(\mathscr{S}_1), \mathscr{D})<\infty$.
This implies $\mathfrak{y}^{-1}(\mathscr{D})\in \mathcal{T}_{\mathscr{S}}(\mathscr{S}_1)$, and therefore the map $\mathfrak{y}^{-1}(-)$ is well defined. It is easy to check that $\mathfrak{S}(-)$ and $\mathfrak{y}^{-1}(-)$ are inverse bijections.
\end{proof}

\subsection{Completing triangulated categories at objects}\label{CMAO}
In this section, we prove our main results - $(a)$ and $(b)$ of Theorem \ref{Main result}. More generally, we provide necessary conditions for the existence of bounded $t$-structures in terms of the completions of a triangulated category. In particular, we show Theorem \ref{BABB}(2) which implies Theorem \ref{Main result}$(a)$. The proofs of these results are based on lifting $t$-structures along completions in Theorem \ref{ET}, and are thus \emph{different} from Neeman's proof of Theorem \ref{GC}. Further, we discuss the equivalence of bounded $t$-structures on the completions of triangulated categories and show Theorem \ref{BABB}(3). This leads to a proof of Theorem \ref{Main result}$(b)$. For a compactly generated triangulated category, we also lift bounded $t$-structures from the completion of the category of its compact objects to itself (see Theorem \ref{Equivalent $t$-structure}).


Throughout this section, \emph{let $\mathscr{S}$ be an essentially small triangulated category with an object $G$.} We take a $G$-good metric $\mathscr{M}:=\{\mathscr{M}_n\}_{n\in\mathbb{N}}$ on $\mathscr{S}$ (see Definition \ref{PEC}). Recall from Lemma \ref{Properties}(4) that equivalent good metrics produce the same completion. So, we assume $\mathscr{M}_n=\langle G\rangle^{(-\infty,-n]}$ for any $n\in\mathbb{N}$. Following Definition \ref{Completion}, we have defined full subcategories of $\mathscr{S}\Modcat$: $\mathfrak{L}_G(\mathscr{A})$, $\mathfrak{C}(\mathscr{S})$ and $\mathfrak{S}_G(\mathscr{A})$, where $\mathscr{A}\subseteq\mathscr{S}$ is a full subcategory and the subscript $G$ reminds us of the good metric determined by $G$. In particular, $\mathfrak{S}_G(\mathscr{S})$ is the completion of $\mathscr{S}$ with respect to the metric $\mathscr{M}$.
As mentioned in the Introduction, we identify $\mathscr{S}$ (or any full subcategory of $\mathscr{S}$) with its essential image in $\mathscr{S}\Modcat$ under the Yoneda functor $\mathfrak{y}:\mathscr{S}\to \mathscr{S}\Modcat$. Moreover, we make a convention: \emph{the left (or right) orthogonal subcategory associated to a full subcategory of $\mathfrak{S}_G(\mathscr{S})$ is calculated in $\mathfrak{S}_G(\mathscr{S})$.}

The following result is very crucial in the proof of our main results from two aspects. One is that the aisle of a bounded $t$-structure on an intermediate triangulated category between $\mathscr{S}$ and $\mathfrak{S}_G(\mathscr{S})$ can be recovered by first restricting to $\mathscr{S}$ and then lifting to $\mathfrak{S}_G(\mathscr{S})$. The other is that assuming the existence of a bounded $t$-structure on the intermediate category and the finiteness of finitistic dimension forces the good metric on $\mathscr{S}$ induced from the $t$-structure to be a $G$-good metric.

\begin{Lem}\label{Restriction}
Suppose that $\mathscr{X}$ is a full triangulated subcategory of $\mathfrak{S}_G(\mathscr{S})$ with $\mathscr{S}\subseteq \mathscr{X}$. Let $(\mathscr{X}^{\leqslant 0},\mathscr{X}^{\geqslant 0})$ be a $t$-structure on $\mathscr{X}$. Then the following statements are true.

$(1)$ If $(\mathscr{X}^{\leqslant 0},\mathscr{X}^{\geqslant 0})$ is bounded above, then $\mathscr{X}\cap\mathfrak{S}_G(\mathscr{S}\cap \mathscr{X}^{\leqslant 0})=\mathscr{X}^{\leqslant 0}$.

$(2)$ If $(\mathscr{X}^{\leqslant 0},\mathscr{X}^{\geqslant 0})$ is bounded and $\mathscr{A}$ is a full subcategory of $\mathscr{S}$ satisfying
$$\mathscr{S}\cap{^{\perp}}\big(G[a, \infty)\big)\subseteq \mathscr{A} \subseteq \langle G\rangle^{(-\infty,\, b]}$$ for some integers $a$ and $b$, then there are nonnegative integers $r$ and $s$ with $\mathscr{A}[r]\subseteq\mathscr{S}\cap \mathscr{X}^{\leqslant 0}\subseteq \mathscr{A}[-s]$.

$(3)$ If $(\mathscr{X}^{\leqslant 0},\mathscr{X}^{\geqslant 0})$ is bounded and  $\fd(\mathscr{S}\opp, G\opp)<\infty$, then there are nonnegative integers $r$ and $s$ with
$\langle G\rangle^{(-\infty,-r]}\subseteq\mathscr{S}\cap \mathscr{X}^{\leqslant 0}\subseteq \langle G\rangle^{(-\infty, s]}$, and thus $\{\mathscr{S}\cap \mathscr{X}^{\leqslant -n}\}_{n\in\mathbb{N}}$ is a $G$-good metric on $\mathscr{S}$.
\end{Lem}

\begin{proof}
$(1)$ By assumption, $\mathscr{S}\subseteq \mathscr{X}\subseteq \mathfrak{S}_G(\mathscr{S})$.
By Lemma \ref{Properties}(3), $\mathscr{S}$ is a full triangulated subcategory of $\mathfrak{S}_G(\mathscr{S})$. Let $\mathscr{D}:=(\mathscr{X}^{\leqslant 0},\mathscr{X}^{\geqslant 0})$. 
As $\mathscr{D}$ is at least bounded above in the hypotheses of $(1)$-$(3)$, there exists an integer $c$ with
$G\in\mathscr{X}^{\leqslant c}$.
Since $\mathscr{X}^{\leqslant c}\subseteq\mathscr{X}$ is closed under extensions, positive shifts and direct summands, $\langle G\rangle^{(-\infty,0]}\subseteq\mathscr{S}\cap\mathscr{X}^{\leqslant c}$.
This implies the inclusion
$$(\ast): \quad \mathscr{M}_c:=\langle G\rangle^{(-\infty,-c]}\subseteq \mathscr{S}\cap\mathscr{X}^{\leqslant 0}.$$

Clearly, $\mathscr{X}^{\leqslant 0}=\mathscr{X}\cap{^\bot}\mathscr{X}^{\geqslant 1}$. Let $P\in\mathscr{X}\cap\mathfrak{S}_G(\mathscr{S}\cap \mathscr{X}^{\leqslant 0})$. Then
$P\simeq \mathop{\text{colim}}\limits_{\longrightarrow}P_{n}$, where $\{P_{\bullet}\}$ is a Cauchy sequence in $\mathscr{S}$ such that $P_n\in\mathscr{S}\cap\mathscr{X}^{\leqslant 0}$ for $n\in\mathbb{N}$.
For any $Q\in \mathscr{X}^{\geqslant 1}$,
$\Hom_\mathscr{X}(P, Q)\simeq \Hom_\mathscr{X}(\mathop{\text{colim}}\limits_{\longrightarrow}P_{n}, Q)\simeq \mathop{\text{lim}}\limits_{\longleftarrow}\Hom_\mathscr{X}(P_n, Q)=0.$ We have the last equality because $P_n\in \mathscr{X}^{\leqslant 0}$ for all $n$ and $Q \in \mathscr{X}^{\geqslant 1}$. Thus, $P\in \mathscr{X}^{\leqslant 0}$, and therefore $\mathscr{X}\cap\mathfrak{S}_G(\mathscr{S}\cap \mathscr{X}^{\leqslant 0})\subseteq \mathscr{X}^{\leqslant 0}$.

To show $\mathscr{X}^{\leqslant 0} \subseteq \mathscr{X}\cap\mathfrak{S}_G(\mathscr{S}\cap \mathscr{X}^{\leqslant 0})$, we first take $F\in\mathscr{X}^{\leqslant 0}$.
Since $\mathscr{X}^{\leqslant 0}\subseteq \mathscr{X}\subseteq \mathfrak{S}_G(\mathscr{S})$,
there is an isomorphism $F\simeq\mathop{\text{colim}}\limits_{\longrightarrow}F_{n}$ in $\mathscr{S}\Modcat$, where $\{F_{\bullet}, f_{\bullet}\}$ is a Cauchy sequence in $\mathscr{S}$ with respect to the $G$-good metric $\mathscr{M}$. Then there is a positive integer $m$ (only depending on $c$) with  $C_i:=\text{Cone}(f_{i+1})\in\mathscr{M}_{c+1}$ for all $i\geqslant m$. In the following, we consider $i\geqslant m$. Note that $\mathscr{M}_c\subseteq \mathscr{X}^{\leqslant 0}\subseteq{^\bot}\mathscr{X}^{\geqslant 1}$ by $(\ast)$, and consequently, $\mathscr{M}_{c+1}=\mathscr{M}_c[1]\subseteq \mathscr{X}^{\leqslant -1}\subseteq
{^\bot}\mathscr{X}^{\geqslant 1}$. This implies $C_i, C_i[-1] \in {^\bot}\mathscr{X}^{\geqslant 1}$. Now, for any $Q\in \mathscr{X}^{\geqslant 1}$, we can apply the functor $\Hom_\mathscr{X}(-, Q)$ to the triangle $C_i[-1]\to F_i\to F_{i+1}\to C_i$ in $\mathscr{S}$ and obtain isomorphisms $\Hom_\mathscr{X}(f_{i+1}, Q): \Hom_\mathscr{X}(F_{i+1}, Q) \to \Hom_\mathscr{X}(F_i, Q)$, as $\Hom_{\mathscr{X}}(C_i,Q)=\Hom_{\mathscr{X}}(C_i[-1],Q)=0$. Thus
$$
0=\Hom_\mathscr{X}(F, Q)\simeq \Hom_\mathscr{X}(\mathop{\text{colim}}\limits_{\longrightarrow}F_{n}, Q)\simeq\mathop{\text{lim}}\limits_{\longleftarrow}\Hom_\mathscr{X}(F_n, Q)\simeq \Hom_\mathscr{X}(F_i, Q).
$$
It follows that $F_i\in \mathscr{S}\cap{^\bot}\mathscr{X}^{\geqslant 1}\subseteq\mathscr{X}\cap{^\bot}\mathscr{X}^{\geqslant 1}=\mathscr{X}^{\leqslant 0}$.  Therefore  $F\in\mathscr{X}\cap\mathfrak{S}_G(\mathscr{S}\cap \mathscr{X}^{\leqslant 0})$. This shows $\mathscr{X}^{\leqslant 0}\subseteq\mathscr{X}\cap\mathfrak{S}_G(\mathscr{S}\cap \mathscr{X}^{\leqslant 0})$.

$(2)$ Now, $\mathscr{D}$ is bounded and assume that the category $\mathscr{A}$ in $(2)$ exists. In $(\ast)$, we can assume $c\geqslant -b$ because as long as $c$ is larger, $\mathscr{M}_c$ is smaller. Since $\langle G\rangle^{(-\infty,b]}=\langle G\rangle^{(-\infty, -c]}[-b-c]$, the inclusions $(\ast)$ and $\mathscr{A}\subseteq\langle G\rangle^{(-\infty,b]}$ imply that $\mathscr{A}\subseteq(\mathscr{S}\cap\mathscr{X}^{\leqslant 0})[-b-c]$. Let $r:=b+c\geqslant 0$.
Then $\mathscr{A}[r]\subseteq \mathscr{S}\cap\mathscr{X}^{\leqslant 0}$.

As $\mathscr{D}$ is bounded below, there exists an integer $d$ with $G\in\mathscr{X}^{\geqslant -d}$; in other words, $G[-d-1]\in\mathscr{X}^{\geqslant 1}$. Similarly, we can choose $d$ bigger enough such that $d+1\geqslant a$. Let $\mathscr{Y}:=\langle G\rangle^{[d+1,\infty)}\subseteq\mathscr{S}$. Clearly, $\mathscr{X}^{\geqslant 1}\subseteq\mathscr{X}$ is closed under extensions, negative shifts and direct summands. This forces $\mathscr{Y}\subseteq\mathscr{X}^{\geqslant 1}$. Since $\mathscr{X}^{\leqslant 0}=\mathscr{X}\cap{^{\bot}\mathscr{X}^{\geqslant 1}}$, we obtain $\mathscr{X}^{\leqslant 0}\subseteq\mathscr{X}\cap{^{\bot}\mathscr{Y}}$. Further, by Definition \ref{coprod}(4), the objects of $\mathscr{Y}$ are constructed from $G[d+1,\infty)$ by taking extensions and direct summands. Thus ${^{\bot}\mathscr{Y}}={^{\bot}\big(G[d+1,\infty)\big)}$. It follows from $\mathscr{S}\subseteq\mathscr{X}$ that $$\mathscr{S}\cap\mathscr{X}^{\leqslant 0}\subseteq\mathscr{S}\cap{^{\bot}\mathscr{Y}}=\mathscr{S}\cap{^{\bot}\big(G[d+1,\infty)\big)}.$$
Let $s:=d+1-a\geqslant 0$. Clearly, $G[a,\infty)=\big(G[d+1,\infty)\big)[s]$. Since  $\mathscr{S}\cap{^{\bot}\big(G[a,\infty)\big)}\subseteq\mathscr{A}$ by the assumptions on $\mathscr{A}$,
we have
$$\mathscr{S}\cap{^{\bot}\big(G[d+1,\infty)\big)}=\big(\mathscr{S}\cap{^\bot}(G[a,\infty))\big)[-s]
\subseteq\mathscr{A}[-s].$$
It follows that $\mathscr{S}\cap\mathscr{X}^{\leqslant 0}\subseteq\mathscr{A}[-s]$.
Thus $\mathscr{A}[r]\subseteq\mathscr{S}\cap \mathscr{X}^{\leqslant 0}\subseteq \mathscr{A}[-s]$.

$(3)$ Now, $\mathscr{D}$ is bounded and assume that $e:=\fd(\mathscr{S}\opp, G\opp)<\infty$.
By Definition \ref{Def-FD}, the inclusion $\mathscr{S}\opp\cap G\opp(-\infty,-1]^{\bot}\subseteq\langle G\opp\rangle^{[0,\infty)}[e]$ holds in $\mathscr{S}\opp$, and therefore taking opposite category yields the inclusion $\mathscr{S}\cap{^{\bot}}(G[1,\infty))\subseteq\langle G\rangle^{(-\infty,0]}[-e]=\langle G\rangle^{(-\infty, e]}$. In $(2)$, we take $a=1$, $b=e$ and $\mathscr{A}=\langle G\rangle^{(-\infty, e]}$.
Then there are nonnegative integers $r_1$ and $s_1$ with $\mathscr{A}[r_1]\subseteq\mathscr{S}\cap \mathscr{X}^{\leqslant 0}\subseteq \mathscr{A}[-s_1]$.
Thus $$\langle G\rangle^{(-\infty,\, e-r_1]}\subseteq\mathscr{S}\cap \mathscr{X}^{\leqslant 0}\subseteq \langle G\rangle^{(-\infty,\, e+s_1]}.$$
Let $r:=\max\{0, r_1-e\}$ and $s:=e+s_1$. Then $\langle G\rangle^{(-\infty,\, -r]}\subseteq\mathscr{S}\cap \mathscr{X}^{\leqslant 0}\subseteq \langle G\rangle^{(-\infty,\, s]}.$
Taking the $n$-th shift to these inclusions leads to
$\langle G\rangle^{(-\infty,-r-n]}\subseteq\mathscr{S}\cap\mathscr{X}^{\leqslant -n}\subseteq \langle G\rangle^{(-\infty, s-n]}$ for $n\in\mathbb{N}$. By Definition \ref{PEC}, $\{\mathscr{S}\cap \mathscr{X}^{\leqslant -n}\}_{n\in\mathbb{N}}$ is a $G$-good metric on $\mathscr{S}$.
\end{proof}

\begin{Rem}
Note that the condition $\mathscr{S}\cap{^{\perp}}\big(G[a, \infty)\big) \subseteq \langle G\rangle^{(-\infty,\, b]}$ for some integers $a$ and $b$ is equivalent to $\fd(\mathscr{S}\opp, G\opp)<\infty$. Moreover, in Lemma \ref{Restriction}(2), if additionally $\mathscr{A}\subseteq\mathscr{S}$ is closed under extensions and positive shifts, then $\{\mathscr{A}[n]\}_{n\in\mathbb{N}}$ is a $G$-good metric on $\mathscr{S}$.

\end{Rem}

We now state the first main result of this section.

\begin{Theo}\label{BABB}
$(1)$ Suppose that $(\mathscr{S}^{\leqslant 0}, \mathscr{S}^{\geqslant 0})$ is a bounded above $t$-structure on $\mathscr{S}$. Then the pair $\big(\mathfrak{S}_G(\mathscr{S}^{\leqslant 0}), \mathfrak{S}_G(\mathscr{S}^{\geqslant 0})\big)$ is a bounded above $t$-structure on $\mathfrak{S}_G(\mathscr{S})$ and has the same heart as
$(\mathscr{S}^{\leqslant 0}, \mathscr{S}^{\geqslant 0})$.

$(2)$ Suppose that $(\mathscr{S}^{\leqslant 0}, \mathscr{S}^{\geqslant 0})$ is a bounded $t$-structure on $\mathscr{S}$. Then $\mathscr{S}\subseteq\mathfrak{S}_G(\mathscr{S})$. Further, if $\fd(\mathscr{S}\opp, G\opp)<\infty$, then $\mathscr{S}=\mathfrak{S}_G(\mathscr{S})$.

$(3)$ Suppose that $\fd(\mathscr{S}\opp, G\opp)<\infty$ and $\mathscr{X}$ is a full triangulated subcategory of $\mathfrak{S}_G(\mathscr{S})$ with $\mathscr{S}\subseteq \mathscr{X}$.
Then all bounded $t$-structures on $\mathscr{X}$ are equivalent.
\end{Theo}

\begin{proof}
$(1)$ Since $(\mathscr{S}^{\leqslant 0},\mathscr{S}^{\geqslant 0})$ is bounded above, it is extendable by Lemma \ref{tool}(3). Now, $(1)$ follows from Theorem \ref{ET}.

$(2)$ Since $(\mathscr{S}^{\leqslant 0},\mathscr{S}^{\geqslant 0})$ is bounded below, we see from Lemma \ref{tool}(2) that
$\mathscr{M}$ is embeddable. By Lemma \ref{Properties}(3), the functor $\mathfrak{y}$ restricts to a fully faithful triangle functor $\mathscr{S}\to\mathfrak{S}_G(\mathscr{S})$. This implies $\mathscr{S}\subseteq\mathfrak{S}_G(\mathscr{S})$.

Suppose $\fd(\mathscr{S}\opp, G\opp)<\infty$. In Lemma \ref{Restriction}(3), we take $\mathscr{X}=\mathscr{S}$ and $(\mathscr{X}^{\leqslant 0},\mathscr{X}^{\geqslant 0})=(\mathscr{S}^{\leqslant 0},\mathscr{S}^{\geqslant 0})$, and thus the good metrics $\mathscr{M}$ and $\{\mathscr{S}^{\leqslant -n}\}_{n\in\mathbb{N}}$ on $\mathscr{S}$ are equivalent. Note that equivalent good metrics produce the same completion. Since the $t$-structure $(\mathscr{S}^{\leqslant 0},\mathscr{S}^{\geqslant 0})$ is bounded below, we see from Corollary \ref{Unexpected} that $\mathscr{S}=\mathfrak{S}_G(\mathscr{S})$.

$(3)$ Let $(\mathscr{X}_i^{\leqslant 0},\mathscr{X}_i^{\geqslant 0})$ for $i=1, 2$ be bounded $t$-structures on $\mathscr{X}$. By Lemma \ref{Restriction}(3), there exist nonnegative integers $r_i$ and $s_i$
with $\langle G\rangle^{(-\infty,-r_i]}\subseteq\mathscr{S}\cap \mathscr{X}_i^{\leqslant 0}\subseteq \langle G\rangle^{(-\infty, s_i]}$. This implies
$$\mathscr{S}\cap \mathscr{X}_1^{\leqslant -(s_1+r_2)}\subseteq \langle G\rangle^{(-\infty,-r_2]}\subseteq \mathscr{S}\cap \mathscr{X}_2^{\leqslant 0}\subseteq \langle G\rangle^{(-\infty, s_2]}\subseteq \mathscr{S}\cap \mathscr{X}_1^{\leqslant s_2+r_1}.$$
Let $t:=\max\{s_1+r_2, s_2+r_1\}$. Then
$$(\mathscr{S}\cap \mathscr{X}_1^{\leqslant 0})[t]=\mathscr{S}\cap \mathscr{X}_1^{\leqslant -t}\subseteq \mathscr{S}\cap \mathscr{X}_2^{\leqslant 0}\subseteq \mathscr{S}\cap \mathscr{X}_1^{\leqslant t}=(\mathscr{S}\cap \mathscr{X}_1^{\leqslant 0})[-t].$$
By Lemma \ref{Properties}(1), $\Sigma^t\big(\mathfrak{S}_G(\mathscr{S}\cap\mathscr{X}_1^{\leqslant 0})\big)\subseteq \mathfrak{S}_G(\mathscr{S}\cap\mathscr{X}_2^{\leqslant 0})\subseteq \Sigma^{-t}\big(\mathfrak{S}_G(\mathscr{S}\cap\mathscr{X}_1^{\leqslant 0})\big),$ and therefore
$$
\mathscr{X}\cap\Sigma^t\big(\mathfrak{S}_G(\mathscr{S}\cap\mathscr{X}_1^{\leqslant 0})\big)\subseteq \mathscr{X}\cap\mathfrak{S}_G(\mathscr{S}\cap\mathscr{X}_2^{\leqslant 0})\subseteq \mathscr{X}\cap\Sigma^{-t}\big(\mathfrak{S}_G(\mathscr{S}\cap\mathscr{X}_1^{\leqslant 0})\big).
$$
Since $\mathscr{X}$ is a full triangulated subcategory of $\mathfrak{S}_G(\mathscr{S})$, we have
{\small
$$
\mathscr{X}\cap\Sigma^t\big(\mathfrak{S}_G(\mathscr{S}\cap\mathscr{X}_1^{\leqslant 0})\big)=\Sigma^t\big(\mathscr{X}\cap\mathfrak{S}_G(\mathscr{S}\cap\mathscr{X}_1^{\leqslant 0})\big),\; \mathscr{X}\cap\Sigma^{-t}\big(\mathfrak{S}_G(\mathscr{S}\cap\mathscr{X}_1^{\leqslant 0})\big)=\Sigma^{-t}\big(\mathscr{X}\cap\mathfrak{S}_G(\mathscr{S}\cap\mathscr{X}_1^{\leqslant 0})\big).
$$}

\noindent It follows from Lemma \ref{Restriction}(1) that $\mathscr{X}_1^{\leqslant -t}\subseteq \mathscr{X}_2^{\leqslant 0}\subseteq \mathscr{X}_1^{\leqslant t}$. Thus $(\mathscr{X}_1^{\leqslant 0},\mathscr{X}_1^{\geqslant 0})$ and $(\mathscr{X}_2^{\leqslant 0},\mathscr{X}_2^{\geqslant 0})$
are equivalent.
\end{proof}

\smallskip
{\bf Proof of Theorem \ref{Main result}.} Note that Theorem \ref{Main result}$(a)$ follows from Theorem \ref{BABB}(2), while Theorem \ref{Main result}$(b)$ is exactly Theorem \ref{BABB}$(3)$. $\square$

\medskip
In the rest of this section, we consider a special class of the completions of triangulated categories that frequently occur in practice.

\emph{Let $\mathscr{S}$ be the full subcategory of a compactly generated triangulated category $\mathscr{T}$ consisting of all compact objects, $G\in\mathscr{T}$ a compact generator, and $(\mathscr{T}^{\leqslant 0}, \mathscr{T}^{\geqslant 0})$ a $t$-structure on $\mathscr{T}$ in the preferred equivalence class.}
We consider the good metric $\mathscr{M}:=\{\mathscr{M}_{n}\}_{n\in\mathbb{N}}$ with $\mathscr{M}_{n}:=\mathscr{S}\cap\mathscr{T}^{\leqslant -n}$ on $\mathscr{S}$,
and denote by $\widehat{\mathscr{S}}$ the completion of $\mathscr{S}$ with respect to $\mathscr{M}$.
By Lemma \ref{Compact generator}(2), $\mathscr{M}$ is a $G$-good metric on $\mathscr{S}$ and thus $\widehat{\mathscr{S}}=\mathfrak{S}_G(\mathscr{S})$. This also implies that $\widehat{\mathscr{S}}$ is
independent of the choices of compact generators of $\mathscr{T}$ (or equivalently, classical generators of $\mathscr{S}$).

We will provide two sufficient conditions for any bounded $t$-structure on $\widehat{\mathscr{S}}$ to be lifted to a $t$-structure on $\mathscr{T}$ in the preferred equivalence class.

\begin{Lem}\label{Inclusion}
The inclusion $\mathscr{S} \subseteq \widehat{\mathscr{S}}$ holds if and only if $\Hom_\mathscr{T}(G[i],G)=0$ for $i\gg 0$.
\end{Lem}

\begin{proof}
By Lemma \ref{Compact generator}(2),
the good metrics $\mathscr{M}$ and $\mathscr{N}:=\{\langle G\rangle^{(-\infty,-n]}\}_{n\in\mathbb{N}}$ on $\mathscr{S}$ are equivalent. This implies $\mathscr{S}(\mathscr{N})=\mathscr{S}(\mathscr{M})$; see the paragraph before Lemma \ref{Properties} for notation. Observe that
$$
\mathscr{S}(\mathscr{N})=\mathscr{S}_{\rm tc}:=\{X\in\mathscr{S}\mid \Hom_\mathscr{S}(G[n], X)=0, \; n \gg 0\}.
$$ By Lemma \ref{Properties}(3), the inclusion $\mathscr{S} \subseteq \widehat{\mathscr{S}}$ holds if and only if $\mathscr{S}=\mathscr{S}(\mathscr{M})$. Clearly, the latter is also equivalent to $\mathscr{S}=\mathscr{S}_{\rm tc}$. Since $G$ is a compact generator of $\mathscr{T}$, it is a classical generator of $\mathscr{S}$ by Lemma \ref{Smd}. It follows that $\mathscr{S}=\mathscr{S}_{\rm tc}$ if and only if $\Hom_\mathscr{T}(G[i],G)=0$ for $i\gg 0$. Thus Lemma \ref{Inclusion} holds.
\end{proof}

We now state the second main result of this section.

\begin{Theo}\label{Equivalent $t$-structure}
Let $(\mathscr{D}^{\leqslant 0}, \mathscr{D}^{\geqslant 0})$ be a bounded $t$-structure on $\widehat{\mathscr{S}}$. Suppose $\Hom_\mathscr{T}(G[i],G)=0$ for $i\gg 0$.
Then the following statements are true.

$(1)$ There exists a full subcategory $\mathscr{B}$ of $\mathscr{S}$ with $\mathscr{B}[1] \subseteq \mathscr{B}$ such that $\text{Coprod}(\mathscr{D}^{\leqslant 0})=\text{Coprod}(\mathscr{B})$. Thus $\text{Coprod}(\mathscr{D}^{\leqslant 0})$ is the aisle of a compactly generated $t$-structure on $\mathscr{T}$.

$(2)$ The $t$-structure $\big(\text{Coprod}(\mathscr{D}^{\leqslant 0}), (\text{Coprod}(\mathscr{D}^{\leqslant 0})[1])^{\perp}\big) $ on $\mathscr{T}$ belongs to the preferred equivalence class if either of the following conditions holds:

\quad $(a)$ $\fd(\mathscr{S}^{\text{op}})<\infty$.

\quad $(b)$ $\mathscr{T}=\text{Coprod}_m(H(-\infty, \infty))$ for some object $H\in \widehat{\mathscr{S}}$ and some positive integer $m$.
\end{Theo}

\begin{proof}
By Lemma \ref{Inclusion}, $\mathscr{S}\subseteq\widehat{\mathscr{S}}$. Let $\mathscr{A}:=\mathscr{D}^{\leqslant 0}$. Since $(\mathscr{A},\mathscr{D}^{\geqslant 0})$ is a bounded above $t$-structure on $\widehat{\mathscr{S}}$, there is a natural number $a$ such that $G[a]\in\mathscr{A}$.
Recall that $G\in\mathscr{S}$ and $\mathscr{T}_{G}^{\leqslant 0}=\overline{\langle G\rangle}^{(-\infty,0]}$. Then $\mathscr{T}_{G}^{\leqslant -a}\subseteq\text{Coprod}(\mathscr{S}\cap\mathscr{A})$. Since $(\mathscr{T}^{\leqslant 0},\mathscr{T}^{\geqslant 0})$ and $(\mathscr{T}_G^{\leqslant 0},\mathscr{T}_G^{\geqslant 0})$ are equivalent, there is an integer $b$ with $b\geqslant a$ such that $\mathscr{T}^{\leqslant -b}\subseteq\mathscr{T}_G^{\leqslant -a}$. Thus $\mathscr{T}^{\leqslant -b}\subseteq\text{Coprod}(\mathscr{S}\cap\mathscr{A})\subseteq\text{Coprod}(\mathscr{A})$.

$(1)$ The idea of the proof of $(1)$ is very similar to the one of \cite[Lemma 8.1]{Neeman4}.

By Theorem \ref{MCP}$(2)(a)$, $\widehat{\mathscr{S}}\subseteq\mathscr{T}_{c}^{-}$. Since $\mathscr{A}\subseteq \widehat{\mathscr{S}}$, each object $A$ of $\mathscr{A}$ is endowed with a triangle $C\rightarrow B\rightarrow A\rightarrow C[1]$ such that $C\in\mathscr{T}^{\leqslant -b}$ and $B\in\mathscr{S}$. Let $\mathscr{B}:=\mathscr{S}\cap(\mathscr{T}^{\leqslant -b}\ast\mathscr{A})$. Then $B\in\mathscr{B}$ and $A\in\mathscr{B}\ast\mathscr{T}^{\leqslant -b-1}\subseteq\mathscr{B}\ast\mathscr{T}^{\leqslant -b}$. Clearly, $\mathscr{S}\cap\mathscr{A}\subseteq\mathscr{B}\subseteq\mathscr{T}^{\leqslant-b}\ast\mathscr{A}$ and $\mathscr{S}\cap\mathscr{A}\subseteq\mathscr{A}\subseteq\mathscr{B}\ast\mathscr{T}^{\leqslant -b}$. Since $\mathscr{T}^{\leqslant -b}\subseteq\text{Coprod}(\mathscr{S}\cap\mathscr{A})$, it follows that  $$\mathscr{B}\subseteq\text{Coprod}(\mathscr{S}\cap\mathscr{A})\ast\mathscr{A}\subseteq\text{Coprod}(\mathscr{A})\ast\mathscr{A}=\text{Coprod}(\mathscr{A}),$$
$$\mathscr{A}\subseteq\mathscr{B}\ast\text{Coprod}(\mathscr{S}\cap\mathscr{A})\subseteq\mathscr{B}\ast\text{Coprod}(\mathscr{B})=\text{Coprod}(\mathscr{B}).$$ Thus $\text{Coprod}(\mathscr{A})=\text{Coprod}(\mathscr{B})$. As $\mathscr{B}$ consists of compact objects, $\text{Coprod}(\mathscr{B})$ (and thus also $\text{Coprod}(\mathscr{A})$) is the aisle of the compactly generated $t$-structure on $\mathscr{T}$ generated by $\mathscr{B}$ (see Example \ref{Typical example}).

$(2)$ Recall that, for a full subcategory of $\mathscr{X}$ of $\mathscr{T}$, the category $\text{Coprod}(\mathscr{X})$ is the smallest full subcategory of $\mathscr{T}$ containing $\mathscr{X}$ and  closed under coproducts and extensions. This implies $\mathscr{X}^{\bot}=\text{Coprod}(\mathscr{X})^{\bot}$.
By the proof of $(1)$,
$\mathscr{T}^{\leqslant -b}\subseteq\text{Coprod}(\mathscr{A})=\text{Coprod}(\mathscr{B})$ with $\mathscr{B}\subseteq\mathscr{S}$. Since $\text{Coprod}(\mathscr{A})=\text{Coprod}(\mathscr{B})$, we have $\mathscr{A}^{\bot}=\text{Coprod}(\mathscr{A})^{\bot}=\text{Coprod}(\mathscr{B})^{\bot}=\mathscr{B}^{\bot}$.

{\bf Case $(a)$}: Since $(\mathscr{A},\mathscr{D}^{\geqslant 0})$ is bounded below and $G\in\mathscr{S}\subseteq\widehat{\mathscr{S}}$, we have $\mathscr{D}^{\geqslant 1}=\mathscr{A}^{\bot}\cap\widehat{\mathscr{S}}$ and there exists $d\in\mathbb{N}$ with $G\in\mathscr{D}^{\geqslant -d+1}$. This implies $G[d,\infty)\subseteq\mathscr{A}^{\bot}=\mathscr{B}^{\bot}$ and therefore $\mathscr{B}\subseteq\mathscr{S}\cap{^\bot}(\mathscr{B}^{\bot})\subseteq\mathscr{S}\cap{^{\bot}}(G[d,\infty))$. Suppose $\fd(\mathscr{S}^{\opp})<\infty$. By Corollary \ref{Some-Any}, $e:=\fd(\mathscr{S}^{\opp}, G\opp)<\infty$. This means that the inclusion $(G\opp(-\infty,-1])^{\bot}\subseteq\langle G\opp\rangle^{[0,\infty)}[e]$ holds in $\mathscr{S}\opp$. Taking opposite category leads to the inclusion $\mathscr{S}\cap{^{\bot}}(G[1,\infty))\subseteq\langle G\rangle^{(-\infty,0]}[-e]$ in $\mathscr{S}$.
Moreover, by Lemma \ref{Compact generator}(2), there exists $a\in\mathbb{N}$ with
$\langle G\rangle^{(-\infty,-a]}\subseteq\mathscr{S}\cap\mathscr{T}^{\leqslant 0}$. Then $\mathscr{S}\cap{^{\bot}}(G[1,\infty))\subseteq(\mathscr{S}\cap\mathscr{T}^{\leqslant 0})[-a-e]$.
Let $f:=1-a-e$. Then
$$\mathscr{S}\cap{^{\bot}}(G[f,\infty))\subseteq\big(\mathscr{S}\cap{^{\bot}}(G[1,\infty))\big)[a+e]
\subseteq\mathscr{S}\cap\mathscr{T}^{\leqslant 0}\subseteq\mathscr{T}^{\leqslant 0}.$$ Since $\mathscr{S}$ is a triangulated subcategory of $\mathscr{T}$, we have
$$\mathscr{B}\subseteq\mathscr{S}\cap{^{\bot}}(G[d,\infty))=\big(\mathscr{S}\cap{^{\bot}}(G[f,\infty))\big)[-d+f]\subseteq\mathscr{T}^{\leqslant 0}[-d+f]=\mathscr{T}^{\leqslant d-f}.$$
Clearly, $\mathscr{T}^{\leqslant d-f}\subseteq\mathscr{T}$ is closed under extensions and coproducts.
This implies $\text{Coprod}(\mathscr{B})\subseteq \mathscr{T}^{\leqslant d-f}$.
Thus $\mathscr{T}^{\leqslant -b}\subseteq\text{Coprod}(\mathscr{A})\subseteq\mathscr{T}^{\leqslant d-f}$,
which shows that $(\text{Coprod}(\mathscr{A}),\text{Coprod}(\mathscr{A})[1]^{\bot})$ is in the preferred equivalence class.

{\bf Case $(b)$}: Since $H\in\widehat{\mathscr{S}}$ and $(\mathscr{A},\mathscr{D}^{\geqslant 0})$ is a bounded below $t$-structure on $\widehat{\mathscr{S}}$, there exists a natural number $d$ with $H\in\mathscr{D}^{\geqslant -d+1}$ and  $\mathscr{D}^{\geqslant 1}=\mathscr{A}^{\bot}\cap\widehat{\mathscr{S}}$. It follows that $H[d,\infty)\subseteq\mathscr{A}^{\bot}=\text{Coprod}(\mathscr{A})^{\bot}=\mathscr{B}^{\bot}.$ Since the objects of $\mathscr{B}$ are compact in $\mathscr{T}$, the category $\mathscr{B}^{\bot}$ is closed under coproducts in $\mathscr{T}$. This forces $\text{Coprod}(H[d,\infty))\subseteq\text{Coprod}(\mathscr{A})^{\bot}$, and therefore
$
\text{Coprod}(\mathscr{A})\subseteq{^{\bot}}(\text{Coprod}(\mathscr{A})^{\bot})
\subseteq{^{\bot}}\text{Coprod}(H[d,\infty)).
$ Note that $\mathscr{T}_c^{-}\subseteq\mathscr{T}^{-}$ by $G\in\mathscr{T}_G^{\leqslant 0}$ and that $\widehat{\mathscr{S}}\subseteq\mathscr{T}_c^b\subseteq \mathscr{T}^b$ by Theorem \ref{MCP}$(2)(a)$. Then
$H\in\mathscr{T}^b$. Recall that $\widehat{\mathscr{S}}$ and $\mathscr{T}_c^-$ are determined by $\mathscr{T}$ and the preferred equivalence class of $t$-structures on $\mathscr{T}$. Thus we can take $\mathscr{T}^{\geqslant 0}=\mathscr{T}_G^{\geqslant 0}$, which is closed under coproducts in $\mathscr{T}$. By Lemma \ref{last lemma}(1), there exists a positive integer $t$ (only depending on $H$ and $m$) such that, for each $n\in\mathbb{Z}$,
$$
\mathscr{T}^{\geqslant n}\cap\text{Coprod}_m(H(-\infty,\infty))
\subseteq\text{smd}(\text{Coprod}_m(H[n-t,\infty)))
$$ Since $H[n-t,\infty)\subseteq(H[n-t,\infty))[1]$, we see that $\text{Coprod}(H[n-t,\infty))\subseteq\mathscr{T}$ is closed under direct summands.
Consequently, $\mathscr{T}^{\geqslant n}\cap\text{Coprod}_m(H(-\infty,\infty))\subseteq\text{Coprod}(H[n-t,\infty))$. Taking $n=t+d$ leads to $\mathscr{T}^{\geqslant t+d}\cap\text{Coprod}_m(H(-\infty,\infty))\subseteq\text{Coprod}(H[d,\infty))$. Since $\mathscr{T}=\text{Coprod}_m(H(-\infty,\infty))$ by assumption, $\mathscr{T}^{\geqslant t+d}\subseteq\text{Coprod}(H[d,\infty))$. It follows that $^{\bot}\text{Coprod}(H[d,\infty))\subseteq{^{\bot}}(\mathscr{T}^{\geqslant t+d})=\mathscr{T}^{\leqslant t+d-1}$, and further, $\text{Coprod}(\mathscr{A})\subseteq{^{\bot}}\text{Coprod}(H[d,\infty))\subseteq\mathscr{T}^{\leqslant t+d-1}$. Thus $\mathscr{T}^{\leqslant -b}\subseteq\text{Coprod}(\mathscr{A})\subseteq\mathscr{T}^{\leqslant t+d-1}$. This means that $(\text{Coprod}(\mathscr{A}),\text{Coprod}(\mathscr{A})[1]^{\bot})$ is in the preferred equivalence class of $t$-structures on $\mathscr{T}$.
\end{proof}

Theorem \ref{Equivalent $t$-structure}(1) generalizes both \cite[Lemma 3.1]{MZ} and \cite[Lemma 8.1]{Neeman4}, which deal with $\Db{R\modcat}$ for a (left) coherent ring $R$ and $\mathscr{D}^{b}_{{\rm coh},\, Z}(X)$ for a noetherian scheme $X$ with a closed subset $Z$, respectively.  Moreover, in Theorem \ref{Equivalent $t$-structure}(2), the case $(a)$ is new, while the case $(b)$ is a categorical version of \cite[Theorem 9.2]{Neeman4} that is focused on bounded $t$-structures on $\mathscr{D}^b_{\rm coh}(X)$ for a noetherian, separated, finite-dimensional, quasiexcellent scheme $X$.

By Theorem \ref{Equivalent $t$-structure}(2) and Theorem \ref{BABB}(2), we also obtain the following result, of which a special case regarding $\mathscr{D}^{{\rm perf}}_Z(X)$ was shown in \cite[Lemma 6.1]{Neeman4}.

\begin{Koro}
Suppose that $\fd(\mathscr{S}^{\text{op}})<\infty$ and $(\mathscr{S}^{\leqslant 0}, \mathscr{S}^{\geqslant 0})$ is a bounded $t$-structure on $\mathscr{S}$. Then the $t$-structure $\big(\text{Coprod}(\mathscr{S}^{\leqslant 0}), (\text{Coprod}(\mathscr{S}^{\leqslant 0})[1])^{\perp}\big) $ on $\mathscr{T}$ belongs to the preferred equivalence class.
\end{Koro}

\section{Finitistic dimensions of triangulated categories}

In this section, we concentrate on the key assumption in Theorem \ref{Main result} - \emph{the finiteness of finitistic dimension of a triangulated category at objects}. Finitistic dimension for general triangulated categories (in our sense) is a new concept and accurately generalizes the finitistic dimension for ordinary rings: the finitistic dimension of the derived category of perfect complexes over a ring at the regular module is equal to the finitistic dimension of the ring (Lemma \ref{Finite}(5)). We also show that several classes of triangulated categories have finite finitistic dimension. These finiteness results are of independent interest, and in particular will be applied to proving all the corollaries in the Introduction.

\subsection{Finiteness of finitistic dimensions of triangulated categories}\label{FDTC}

In this section, we discuss some basic properties of finitistic dimension for triangulated categories and provide several classes of triangulated categories with finite finitistic dimension. These classes include triangulated categories with an algebraic $t$-structure (Lemma \ref{Finite}(3)) or with a strong generator (Proposition \ref{ST}), the singularity category of a Gorenstein Artin algebra or of a self-injective DG algebra (Corollary \ref{Artin algebra}), the derived category of perfect complexes on a scheme with finite fintistic dimension (Proposition \ref{SCMFD}) and the category of compact objects in the derived category of a differential graded ring with some cohomological restrictions (Corollary \ref{nonpositive DG ring} and Example \ref{differential graded}). So, we can apply Theorem \ref{Main result} to these triangulated categories. Further, in Appendix \ref{Section B}, we discuss other ways of defining finitistic dimension for triangulated categories that exist in the literature (Definitions \ref{Def-BFD} and \ref{K-FD}), and explain some differences and commonalities between these dimensions.

Throughout this section, let $\mathscr{S}$ be a triangulated category.
Recall from Definition \ref{Def-FD} that the finitistic dimension of $\mathscr{S}$ at an object $G\in\mathscr{S}$ is defined as:
$$\fd(\mathscr{S},G):=\inf\big\{n\in\mathbb{N}\mid G(-\infty, -1]^{\bot}\subseteq\langle G\rangle ^{[0,\infty)}[n]\big\}.$$
If there is an object $G$ with $\fd(\mathscr{S}, G)<\infty$, then we say that $\mathscr{S}$ has \emph{finite finitistic dimension} and denote this by $\fd(\mathscr{S})<\infty$.
In this paper, we care more about when the finitistic dimension of a triangulated category is finite, rather than the precise value of this dimension. In fact, if $\fd(\mathscr{S}, G)=n<\infty$, then $0\leqslant\fd(\mathscr{S}, G\oplus G[-i])\leqslant n-i$ for  $0\leqslant i\leqslant n$. Moreover, the finiteness of finitistic dimension is invariant under triangle equivalences.

We first collect basic properties of finitistic dimension of triangulated categories\textcolor{blue}{,} and establish its finiteness for some common triangulated categories. Recall that a bounded $t$-structure over a triangulated category $\mathscr{S}$ is said to be \emph{algebraic} (see \cite{AMY}) if its heart is a length category  (that is, objects in the heart admits finite filtrations) with finitely many isomorphism classes of simple objects.

\begin{Lem} \label{Finite}
$(1)$ Let $G, H\in\mathscr{S}$ and $G\in\langle H \rangle$. If $\fd(\mathscr{S},G)<\infty$, then $\fd(\mathscr{S}, H)<\infty$.

$(2)$ Suppose that $\mathscr{S}$ has a bounded $t$-structure. For any $G\in\mathscr{S}$, if either $\fd(\mathscr{S}, G)<\infty$ or $\fd(\mathscr{S}\opp, G\opp)<\infty$, then $G$ is a classical generator of $\mathscr{S}$.

$(3)$ Suppose that $\mathscr{S}$ has an algebraic $t$-structure $(\mathscr{S}^{\leqslant 0}  \mathscr{S}^{\geqslant 0})$. Let $G$ be the direct sum of the isomorphism classes of simple objects in the heart of $(\mathscr{S}^{\leqslant 0}, \mathscr{S}^{\geqslant 0})$. Then $\fd(\mathscr{S}, G)=\fd(\mathscr{S}\opp, G\opp)=0$.

$(4)$ Suppose that $\mathscr{S}={\langle G \rangle}^{(-\infty, m]}$ for some $G\in\mathscr{S}$ and $m\in\mathbb{Z}$. Then $\fd(\mathscr{S}, G)=0$. In particular, if there are integers $n\leqslant m$
with $\mathscr{S}={\langle G \rangle}^{[n, m]}$, then both $\fd(\mathscr{S})$ and $\fd(\mathscr{S}\opp)$ are finite.

$(5)$ Let $R$ be a ring.  Then $\fd(\Kb{\prj{R}}, R)=\fd(R)$.
In particular, $\fd(R)<\infty$ if and only if $\fd(\Kb{\prj{R}})<\infty$.
\end{Lem}

\begin{proof}
$(1)$ Assume $\fd(\mathscr{S},G)=d<\infty$. Since $G\in\langle H \rangle$, there are integers $a\leqslant b$ and a positive integer $n$ such that $G \in {\langle H \rangle}_{n}^{[a,b]}$. This implies that if $X \in {H(-\infty, b]}^{\perp}$,  then $\Hom_\mathscr{S}(G[k],X)=0$ for all $k\geqslant 0$. In other words, ${H(-\infty, b]}^{\perp} \subseteq {G(-\infty,0]}^{\perp}$. By $G \in {\langle H \rangle}_{n}^{[a, b]}$, we have $G[a]\in\langle H\rangle_n^{[0,b-a]}$. This gives rise to
${\langle G \rangle}^{[-a, \infty)} \subseteq {\langle H \rangle}^{[0, \infty)}$. Since
${G(-\infty,-1]}^{\perp} \subseteq {\langle G \rangle}^{[0, \infty)}[d]$, it follows that
\[
{H(-\infty,d+b-a-1]}^{\perp} \subseteq {G(-\infty,d-a-1]}^{\perp} \subseteq {\langle G \rangle}^{[-a,\infty)} \subseteq {\langle H \rangle}^{[0,\infty)}.
\]
Consequently, ${H(-\infty,-1]}^{\perp}\subseteq {\langle H \rangle}^{[0,\infty)}[b-a+d]$. Thus $\fd(\mathscr{S}, H) \leqslant b-a+d<\infty$.

$(2)$ It suffices to show $(2)$ in the case $\fd(\mathscr{S}, G)<\infty$ since $\mathscr{S}\opp$ also has a bounded $t$-structure.

Let $\mathscr{D}:=(\mathscr{S}^{\leqslant 0},\mathscr{S}^{\geqslant 0})$ be a bounded $t$-structure on $\mathscr{S}$. Since $\mathscr{D}$ is bounded above, there exists a positive integer $r$ with $G\in\mathscr{S}^{\leqslant r}$. This forces
$G(-\infty, -1]\subseteq\mathscr{S}^{\leqslant r-1}$, and therefore $\mathscr{S}^{\geqslant r}=(\mathscr{S}^{\leqslant r-1})^{\bot}\subseteq G(-\infty, -1]^{\bot}$.
Let $n:=\fd(\mathscr{S}, G)<\infty$. Then $G(-\infty, -1]^{\bot}\subseteq\langle G\rangle ^{[0,\infty)}[n]$. It follows that $\mathscr{S}^{\geqslant r}\subseteq\langle G\rangle$.
Since $\mathscr{D}$ is bounded below, $\mathscr{S}$ is generated by $\mathscr{S}^{\geqslant 0}$ under taking shifts. This implies $\mathscr{S}\subseteq \langle G\rangle$ and thus $G$ is a classical generator of $\mathscr{S}$.

$(3)$ Let $\mathscr{H}$ be the heart of the $t$-structure $(\mathscr{S}^{\leqslant 0}, \mathscr{S}^{\geqslant 0})$. By \cite[1.3.13.1]{BBD},
$$\mathscr{S}=\bigcup_{n\geqslant 0}\mathscr{H}[n]\ast\mathscr{H}[n-1]\ast\cdots \ast \mathscr{H}[-n],$$
$$\mathscr{S}^{\leqslant 0}=\bigcup_{n\geqslant 0}\mathscr{H}[n]\ast\mathscr{H}[n-1]\ast\cdots \ast \mathscr{H}\quad \mbox{and}\quad\mathscr{S}^{\geqslant 0}=\bigcup_{n\geqslant 0}\mathscr{H}\ast\mathscr{H}[-1]\ast\cdots \ast \mathscr{H}[-n],$$
where the $\ast$ operator is associative.
Since each object in $\mathscr{H}$ admits a finite filtration by simple objects in $\mathscr{H}$, it follows that $G$ is a classical generator of $\mathscr{S}$ and
$${G(-\infty,-1]}^{\perp}=\mathscr{H}(-\infty, -1]^{\perp}=(\mathscr{S}^{\leqslant -1})^{\bot}=\mathscr{S}^{\geqslant 0}=\langle\mathscr{H}\rangle^{[0, \infty)}=\langle G\rangle^{[0, \infty)}.$$
Thus $\fd(\mathscr{S}, G)=0$. Since
$\big((\mathscr{S}^{\geqslant 0})\opp, (\mathscr{S}^{\leqslant 0})\opp \big)$ is an algebraic $t$-structure on $\mathscr{S}\opp$ with the heart $\mathscr{H}\opp$, we have $\fd(\mathscr{S}\opp, G\opp)=0$.

$(4)$ Since $\mathscr{S}={\langle G \rangle}^{(-\infty, m]}$, we have $({\langle G \rangle}^{(-\infty, m]})^\bot=0$. It follows from
${\langle G \rangle}^{(-\infty, -1]}={\langle G \rangle}^{(-\infty, m]}[m+1]$ that $({\langle G \rangle}^{(-\infty, -1]})^\bot=0$. Observe that $({\langle G \rangle}^{(-\infty, -1]})^\bot=G(-\infty, -1]^{\bot}$. Thus $\fd(\mathscr{S}, G)=0$.

$(5)$ Let $\mathscr{S}:=\Kb{\prj{R}}$. Clearly, ${R(-\infty,-1]}^{\perp}=\{\cpx{P}\in\mathscr{S}\mid H^n(\cpx{P})=0, \forall \; n<0\}$. By Definition \ref{coprod}(4), ${\langle R \rangle}^{[0,+\infty)}=\bigcup_{n>0}\text{smd}\big(\text{coprod}_{n}(R[0, +\infty))\big)$ that is the smallest full subcategory of $\mathscr{S}$ containing $\prj{R}$ and closed under negative shifts, extensions and direct summands.
It follows that ${\langle R \rangle}^{[0,+\infty)}\subseteq \mathscr{S}$ consists of all those complexes
$\cpx{P}$ which are isomorphic in $\mathscr{S}$ to a complex $\cpx{Q}\in\mathscr{S}$ with $Q^n=0$ for $n<0$.

Let $d:=\fd(R)\in\mathbb{N}\cup\{\infty\}$.
 For any $n\in\mathbb{N}$ with $n<d+1$, there exists an $R$-module $M_{n}$ which has a deleted projective resolution of length $n$ by finitely generated projective $R$-modules
\[P^{\bullet}_{M_{n}}: 0 \lra P_{n} \lraf{f_n} P_{n-1} \lraf{f_{n-1}} \cdots \lra P_{1} \lraf{f_{1}}P_{0} \lra 0\]
such that $f_n$ does not split.
This means that the complex $P^{\bullet}_{M_{n}} $ is in ${R(-\infty,-1]}^{\perp}$, but not in
${\langle R \rangle}^{[0, \infty)}[n-1]$. Thus
$d\leqslant \fd(\mathscr{S}, R)$.  In particular, if
$d=\infty$, then $\fd(\mathscr{S}, R)=\infty$.

Suppose $d<\infty$. Let $\cpx{Y}:=(Y^i,d_Y^i)_{i\in\mathbb{Z}} \in {R(-\infty,-1]}^{\perp}$. Then $H^i(Y^{\bullet})=0$ for all $i<0$, and therefore  $\Coker(d^{-1}_Y)$ has a finite projective resolution by finitely generated projective $R$-modules. Consequently, the projective dimension of
$\Coker(d^{-1}_Y)$ is at most $d$, and further $Y^{\bullet} \in {\langle R \rangle}^{[0, \infty)}[d]$. This shows $\fd(\mathscr{S}, R) \leqslant d $. Thus $d=\fd(\mathscr{S}, R)$.

If $\fd(R)<\infty$, then $\fd(\mathscr{S})<\infty$. Conversely, assume $\fd(\mathscr{S})<\infty$. Then there exists an object $G\in\mathscr{S}$ with $\fd(\mathscr{S}, G)<\infty$. Since $\mathscr{S}=\langle R \rangle$, we see from $(1)$ that $\fd(\mathscr{S}, R)<\infty$. Thus $\fd(R)<\infty$.
\end{proof}

Lemma \ref{Finite}(1) implies that the finiteness of the finitistic dimension of a triangulated category with a classical generator is independent of the choice of a classical generator.

\begin{Koro}\label{Some-Any}
Let $\mathscr{S}$ be a triangulated category with a classical generator $G$. Then the following
are equivalent:
$(a)$ $\fd(\mathscr{S})<\infty$; $(b)$ $\fd(\mathscr{S}, G)<\infty$; $(c)$ $\fd(\mathscr{S}, H)<\infty$ for any other classical generator $H$ of $\mathscr{S}$.
\end{Koro}

Now, we apply Lemma \ref{Finite} to bounded derived categories and singularity categories.

Let $d$ be a positive integer. Following \cite[Definition 2.2.]{HB}, a DG $k$-algebra $S$ over a field $k$ is said to be \emph{$d$-self-injective} if $S$ is nonpositive, proper (that is,
$\bigoplus_{i\in\mathbb{Z}}H^i(S)$ is a finite-dimensional $k$-module) and
${\langle S \rangle}_1^{\{0\}}={\langle D(S_S)[d-1] \rangle}_1^{\{0\}}$ (see Definition \ref{coprod}(4) for notation) in $\D{S}$, where $D:=\Hom_k(-,k)$ is the $k$-duality.
A class of self-injective DG algebras is given by trivial extension.
For instance, given a finite-dimension $k$-algebra $A$, the trivial extension DG algebra $A\oplus D(A)[d-1]$, with the multiplication of the usual trivial extension of $A$ and with zero differential, is $d$-self-injective (see \cite[Section 6]{HB}).

\begin{Koro}\label{Artin algebra}
$(1)$ Let $R$ be an Artin algebra. Then $\fd\big(\mathscr{D}^b(R\modcat)\big)<\infty$. If $R$ is a Gorenstein algebra, then $\fd\big(\mathscr{D}_{\rm sg}(R)\big)<\infty$.

$(2)$ Let $S$ be a $d$-self-injective DG algebra over a field $k$ with $d\geqslant 1$. Then
$\fd\big(\mathscr{D}_{\rm sg}(S)\big)<\infty$.
\end{Koro}

\begin{proof}
$(1)$ Note that $\mathscr{D}^{b}(R\text{-}{\rm mod})$ has a canonical algebraic $t$-structure with the heart $R\modcat$. By Lemma \ref{Finite}(3), $\fd(\mathscr{D}^{b}(R\text{-}{\rm mod}))<\infty$.
Suppose that $R$ is $n$-Gorenstein, that is, the injective dimensions of ${_R}R$ and $R_R$ are the same and equal to $n$. Let $\mathscr{S}$ be the stable category of the Frobenius category of finitely generated Gorenstein-projective $R$-modules. Then there is a triangle equivalence $\mathscr{S}\simeq \mathscr{D}_{\rm sg}(R)$, due to Buchweitz. So, we can identify these two equivalent categories. Since $R$ is $n$-Gorenstein, $\mathscr{S}$ consists of $n$-th syzygies $\Omega_R^n(X)$ for all $X\in R\modcat$. Now, let $J$ be the radical of $R$ and $m$ the Loewy length of $R$, and let $G:=\Omega_R^n(R/J)$. Observe that each finitely generated $R$-module has a radical series of length less than or equal $m$, and taking $n$-th syzygy of this series produces an iterated sequence of triangles in $\mathscr{S}$. This implies $\mathscr{S}={\langle G \rangle}_m^{\{0\}}$. By Lemma \ref{Finite}(4), $\fd(\mathscr{S}, G)=0$. Thus  $\fd\big(\mathscr{D}_{\rm sg}(R)\big)<\infty$.

$(2)$ Clearly, self-injective DG algebras are Gorenstein in the sense that $\langle S \rangle=\langle D(S)\rangle$ in $\D{S}$.  By \cite[Theorem 0.3(4)]{HB}, the category $\mathscr{D}_{\rm sg}(S)$ is triangle equivalent to the stable category $S\mbox{-}\underline{{\rm CM}}$ of left Cohen-Macaulay DG $S$-modules. Let $G$ be the direct sum of (finitely many) isomorphism classes of simple $H^0(S)$-modules. It follows from \cite[Theorem 0.6 and Definition 0.4]{HB} that $S\mbox{-}\underline{{\rm CM}}={\langle G \rangle}^{[1-d, 0]}$. By Lemma \ref{Finite}(4), $\fd(S\mbox{-}\underline{{\rm CM}}, G)<\infty$. Thus $\fd\big(\mathscr{D}_{\rm sg}(S)\big)<\infty$.
\end{proof}

We point out that the opposite categories of all triangulated categories in Corollary \ref{Artin algebra} have finite finitistic dimension. More examples of triangulated categories satisfying the conditions of Lemma \ref{Finite}(4) can also be found in \cite{XC2} which is related to \emph{Tachikawa’s second conjecture} (that is, all finitely generated self-orthogonal modules over a self-injective Artin algebra are projective).

\begin{Bsp}\label{Self-orthogonal}
Let $A$ be a self-injective Artin algebra and let $M\in A\modcat$ be a \emph{self-orthogonal} module (that is, $\Ext_A^i(M, M)=0$ for $i>0$) containing the module ${_A}A$ as a direct summand. We consider the category $A\Modcat$ of left $A$-modules and its stable category $\mathscr{T}:=\Stmc{A}$ that is a compactly generated triangulated category. Let $\Gamma:=\End_\mathscr{T}(M)$ and
$$\mathscr{E}:=\{X\in A\Modcat\mid \Hom_\mathscr{T}(M, X[n])=0, \;n\neq 0, -1;\; \Hom_\mathscr{T}(M,\, X\oplus X[-1])\in\Gamma\modcat\}.$$ Clearly, $\mathscr{E}$ contains all projective $A$-modules. We denote by  $\underline{\mathscr{E}}$ the stable category of $\mathscr{E}$, which is a full subcategory of $\mathscr{T}$. Following Definition \ref{coprod}(5), the full subcategories
$\overline{\langle M\rangle}^{\{0\}}_1\subseteq \overline{\langle M\rangle}\subseteq\mathscr{T}$ are defined. Now, we focus on a quotient category of additive subcategories of $\mathscr{T}$:
$$\mathscr{S}:=\big(\underline{\mathscr{E}}\cap\overline{\langle M\rangle}\big)/\,\overline{\langle M \rangle}^{\{0\}}_1$$
This category is closely related to \emph{Tachikawa’s second conjecture} in the following sense:

Suppose that ${\langle M \rangle}_1^{\{0\}}={\langle D(A)\otimes_AM\rangle}_1^{\{0\}}$ in $\mathscr{T}$ (this is always true if $A$ is symmetric), where $D$ is the usual duality on $A\modcat$. Then $\mathscr{S}$ vanishes if and only if ${_A}M$ is projective, due to \cite[Corollary 4.9]{XC2}.  Moreover, by \cite[Proposition 1.6(2)(3)]{XC2}, $\mathscr{S}$ is a triangulated category and there exists an $A$-module $S$ with $\mathscr{S}=\langle S\rangle_{2n}^{[-1,0]}$, where $n$ is the Loewy length of $\Gamma$. Thus $\fd(\mathscr{S})<\infty$ and $\fd(\mathscr{S}\opp)<\infty$ by Lemma \ref{Finite}(4).
\end{Bsp}

The following result shows that if a triangulated category has a \emph{strong generator} (see Definition \ref{coprod}(6)) with a negative self-extension vanishing condition, then it has finite finitistic dimension. This condition is not very restrictive in practice. For instance, it holds for the case that the category is a full triangulated subcategory of $\mathscr{T}^{b}$ for a triangulated category $\mathscr{T}$ with a $t$-structure.

\begin{Prop}\label{ST}
If $\mathscr{S}={\langle G \rangle}_{n+1}$ for some $n\in\mathbb{N}$ such that $\Hom_{\mathscr{S}}(\Sigma^iG,G)=0$ for $i\geqslant d+1$ with $d\in\mathbb{N}$. Then $\fd(\mathscr{S}, G)\leqslant n(d+1)<\infty$.
\end{Prop}

\begin{proof}
Let $F\in G(-\infty,-1]^{\bot}$. Assume $\mathscr{S}={\langle G \rangle}_{n+1}$. Then $F\in{\langle G \rangle}_{n+1}$.  By Lemma \ref{Strong generator}, the identity map $F\rightarrow F$ factors through an object of $\langle G\rangle_{n+1}^{[-n(d+1),\infty)}$. This implies  $F\in \langle G\rangle_{n+1}^{[-n(d+1),\infty)}\subseteq \langle G\rangle^{[-n(d+1),\infty)}$. Thus $\fd(\mathscr{S}, G)\leqslant n(d+1)<\infty$.
\end{proof}

\begin{Rem}
In \cite[Definition 3.2]{RR}, Rouquier introduced a \emph{dimension} for a triangulated category. It turns out that a triangulated category has a finite dimension if and only if it has a strong generator. In the literature, there appear many classes of (algebraic or geometric) triangulated categories with finite dimensions (for example, see \cite{RR, Neeman3}). We mention two examples:
$(a)$ for an artin ring $R$, the dimension of $\mathscr{D}^b(R\modcat)$ is less than the Loewy length of $R$, and thus $\mathscr{D}^b(R\modcat)$ has a finite dimension (see \cite[Proposition 7.37]{RR});
$(b)$ for a separated scheme $X$ of finite type over a (perfect) field,
$\mathscr{D}^{b}_{{\rm coh}}(X)$ has a finite dimension (see \cite[Theorem 7.38]{RR}).
For these triangulated categories, we have finite finitistic dimension by Proposition \ref{ST}.
In particular, this implies the first statement of Corollary \ref{Artin algebra}(1).
\end{Rem}

When objects of a triangulated category are strongly generated by taking arbitrary coproducts of a special object in a ``big'' triangulated category, we still have finite finitistic dimension.

\begin{Koro}\label{STG}
Let $\mathscr{T}$ be a triangulated category with coproducts, $(\mathscr{T}^{\leq0},\mathscr{T}^{\geq0})$ a $t$-structure on $\mathscr{T}$ such that $\mathscr{T}^{\geq0}\subseteq \mathscr{T}$ is closed under coproducts, $\mathscr{S}$ a full triangulated subcategory of $\mathscr{T}$ closed under direct summands. Suppose $\mathscr{S}\subseteq{\rm Coprod}_m(H(-\infty,\infty))$ for some object $H\in\mathscr{S}\subseteq\mathscr{T}_c^b$ (see Definition \ref{TC}) and some positive integer $m$. Then $\fd(\mathscr{S}, H)<\infty$.
\end{Koro}

\begin{proof}
By Lemma \ref{last lemma}(3), we have $\mathscr{S}={\langle H \rangle}_m$.
Since $H\in\mathscr{S}\subseteq\mathscr{T}_c^b\subseteq \mathscr{T}^b$, there exists a positive integer $n$ such that $H[n]\in \mathscr{T}^{\leqslant 0}$ and $H[-n]\in\mathscr{T}^{\geqslant 0}$. As $\mathscr{T}^{\geqslant 0}\subseteq\mathscr{T}$ is closed under negative shifts, $H[i]\in \mathscr{T}^{\geqslant 0}$ for $i\leqslant -n$. It follows from $\Hom_\mathscr{T}(\mathscr{T}^{\leqslant 0}, \mathscr{T}^{\geqslant 1})=0$ that $\Hom_\mathscr{T}(H[j], H)=0$ for $j\geqslant 2n+1$. Thus $\fd(\mathscr{S}, H)<\infty$ by Proposition \ref{ST}.
\end{proof}

A combination of Example \ref{quasiexcellent} and Corollary \ref{STG} yields the following result.

\begin{Koro}\label{NSQS}
Let $X$ be a noetherian, separated, finite-dimensional, quasiexcellent scheme. Then we have
$\fd(\mathscr{D}^{b}_{\rm coh}(X))<\infty$.
\end{Koro}

\begin{Rem}
In general, for a finite-dimensional, noetherian scheme $X$, the category $\mathscr{D}^{b}_{\rm coh}(X)$ may have infinite finitistic dimension. In fact, by Lemma \ref{Finite}(2), a necessary condition for
$\fd(\mathscr{D}^{b}_{\rm coh}(X))<\infty$ is that $\mathscr{D}^{b}_{\rm coh}(X)$ has a classical generator. However, this condition does not hold even for some affine schemes.  For a commutative noetherian $R$, the existence of a classical generator for $\Db{R\modcat}$ implies the openness of the regular locus of $R$. Thus, for any commutative, noetherian, local ring $R$ whose regular locus
is not open in the spectrum of $R$, we have $\fd(\Db{R\modcat})=\infty$. In this case,
$\fd(\Kb{\prj{R}})<\infty$ by Lemma \ref{Finite}(5) since $\fd(R)\leqslant \dim(R)<\infty$.
\end{Rem}

The following result shows that, under a finiteness condition on a classical generator (not necessarily a strong generator) of a triangulated category and a ring theoretic condition on its endomorphism ring, we also have finite finitistic dimension.

\begin{Prop}\label{Generating properties}
Let $G$ be a  nonzero, classical generator of $\mathscr{S}$. Suppose that the ring
$R:=\End_{\mathscr{S}}(G)$ is left coherent and the direct sum $\bigoplus_{i\in\mathbb{Z}}\Hom_{\mathscr{S}}(G, G[i])$ of left $R$-modules is finitely presented.
If $\mathscr{S}$ is idempotent complete, $R$ is semisimple and $\Hom_{\mathscr{S}}(G[i], G)=0$ for $i>0$, then $\fd(\mathscr{S}, G)=0$.
\end{Prop}

\begin{proof}
Let $\mathscr{S}^{\geqslant i}:=G(-\infty,i-1]^{\bot}$ for $i\in\mathbb{Z}$. Then
$\mathscr{S}^{\geqslant 0}=G(-\infty,-1]^{\bot}$ and $\mathscr{S}^{\geqslant i+1}=\mathscr{S}^{\geqslant i}[-1]\subseteq\mathscr{S}^{\geqslant i}$. Since $\bigoplus_{i\in\mathbb{Z}}\Hom_{\mathscr{S}}(G, G[i])\in R\modcat$, it is clear that ${\rm Hom}_{\mathscr{S}}(G,G[i])=0$ for $\mid i\mid \gg 0$ and ${\rm Hom}_{\mathscr{S}}(G,G[i])\in R\modcat$ for $i\in\mathbb{Z}$. Now, let $$d:=\sup\big\{i\in\mathbb{N}\mid \Hom_\mathscr{S}(G[i],G)\not=0\big\}.$$ Then $d<\infty$.
For $X\in\mathscr{S}$, we define
$\Lambda_X:=\{i\in\mathbb{Z}\,|\,\Hom_{\mathscr{S}}(G[i],X)\not=0\}.$
Since $R$ is left coherent and $G$ is a classical generator of $\mathscr{S}$, the following property holds:
$(\ast)$ The set $\Lambda_X$ is finite and $\Hom_{\mathscr{S}}(G[i],X)\in R\modcat$ for $i\in\mathbb{Z}$.
Further, we define the following ideal of morphisms:
$$\mathcal{C}:=\{f\in\Hom_{\mathscr{S}}(X,Y)\mid X,Y\in\mathscr{S}\,\text{and}\;\;\Hom_{\mathscr{S}}(G[i],f)=0,\;\forall\; i\in\mathbb{Z}\}.$$
Note that a morphism $f:X\rightarrow Y\in\mathcal{C}$ if and only if  for any $i\in\mathbb{Z}$ and for any morphism $g:G[i]\rightarrow X$, there exists a morphism $h:G[i]\rightarrow{\rm Cone}(f)[-1]$ such that
$g=hf'$, where $f'$ appears in the triangle ${\rm Cone}(f)[-1]\lraf{f'}X \lraf{f} Y\lra {\rm Cone}(f)$. The ideal $\mathcal{C}$ has the following nice property $(\ast\ast)$: If $X\in\langle G\rangle_m$ for some $m\geqslant 1$ and $f:X\to Y$ is the composition of $m$ morphisms in $\mathcal{C}$, then $f=0$.

Let $X_0\in\mathscr{S}^{\geqslant 0}$. Then $\Lambda_{X_0}$ is finite and consists of nonpositive
integers. By the property $(\ast)$, for each $i\in\Lambda_{X_0}$, there exists an object
$G_{i,X_0}\in {\langle G[i] \rangle}_1^{\{0\}}$ and a morphism $f_{i,X_0}: G_{i,X_0}\to X_0$ in $\mathscr{S}$ such that
$\Hom_{\mathscr{S}}(G[i],f_{i,X_0}):{\rm Hom}_{\mathscr{S}}(G[i],G_{i,X_0})\to
\Hom_{\mathscr{S}}(G[i],X_0)$ is a surjective homomorphism of $R$-modules.
Now, we define $X'_0:=\bigoplus_{i\in\Lambda_{X_0}}G_{i,X_0}$ and let $f:X'_0\rightarrow X_0$ be the morphism induced by the family $\{f_{i,X_0}\}_{i\in\Lambda_{X_0}}$. Then $f$ is extended to a triangle in $\mathscr{S}$:
$$
\begin{tikzcd}
(\dag)\quad X'_0 \arrow[r, "f"] & X_0 \arrow[r, "g_0"] & X_1 \arrow[r] & {X'_0[1]}
\end{tikzcd}
$$
The construction of $f$ implies that $g_0\in\mathcal{C}$ and $X'_0\in\langle G\rangle_1^{[0,\infty)}$.
Since $\Hom_{\mathscr{S}}(G[i],G)=0$ for all $i>d$, we have $G[-d]\in\mathscr{S}^{\geqslant 0}$.
As $\mathscr{S}^{\geqslant 0}\subseteq\mathscr{S}$ is closed under extensions, negative shifts and direct summands,  $\langle G\rangle^{[d,\infty)} \in \mathscr{S}^{\geqslant 0}$. Thus $X'_0\in\langle G\rangle_1^{[0,\infty)}\subseteq\mathscr{S}^{\geqslant -d}$ and $X_1\in\mathscr{S}^{\geqslant -d-1}$. The above procedure can be carried out analogously for $X_1$, yielding a distinguished triangle
$X_1' \to X_1 \lraf{g_1}X_2 \to X_1'[1]$,
where $X_1'\in\langle G\rangle_1^{[-d-1,\infty)}\subseteq\mathscr{S}^{\geqslant-2d-1}$, $g_1\in\mathcal{C}$ and $X_2\in\mathscr{S}^{-2d-2}$. Therefore the octahedron axiom of triangulated category produces
a triangle
$G_2 \to X_0 \lraf{g_0g_1}X_2\to G_2[1]$,
where $G_2\in X_0'\ast X_1'\subseteq\langle G\rangle_2^{[-d-1,\infty)}$. More generally, for any $j\in\mathbb{N}$, we obtain a triangle
$$\begin{tikzcd}
G_{j+1} \arrow[r] & X_0 \arrow[r, "g_0g_1\cdots g_j"] & X_{j+1} \arrow[r] & {G_{j+1}[1]}
\end{tikzcd}$$
where $G_{j+1}\subseteq\langle G\rangle_{j+1}^{[-j(d+1),\infty)}$ and $g_i\in\mathcal{C}$ for all $0\leqslant i\leqslant j$.

In the following, we consider the \emph{special case}: $\mathscr{S}$ is idempotent complete, $R$ is semisimple and $\Hom_{\mathscr{S}}(G[i], G)=0$ for $i>0$. Then $d=0$ and $G\in\mathscr{S}^{\geqslant 0}$. This implies
$X_1\in\mathscr{S}^{\geqslant -1}$. Now, we claim that $X_1\in\mathscr{S}^{\geqslant 0}$.
It suffices to show $\Hom_\mathscr{S}(G[1], X_1)=0$. Applying $\Hom_\mathscr{S}(G[1],-)$
to the triangle $(\dag)$ yields an exact sequence
$$
0=\Hom_\mathscr{S}(G[1], X_0)\lra \Hom_\mathscr{S}(G[1], X_1)\lra \Hom_\mathscr{S}(G[1], X_0'[1])\lraf{(f[1])^*}  \Hom_\mathscr{S}(G[1], X_0[1]).
$$
Since $X'_0\in\langle G\rangle_1^{[0,\infty)}$ and $d=0$, we have  $\Hom_\mathscr{S}(G[1], X_0'[1])=\Hom_\mathscr{S}(G[1], G_{0, X_0}[1])$.  It follows that $\Hom_\mathscr{S}(G[1], X_1)=0$ if and only if $\Hom_\mathscr{S}(G, f_{0, X_0}):\Hom_\mathscr{S}(G, G_{0, X_0})\to \Hom_\mathscr{S}(G, X_0)$ is injective. Since $\mathscr{S}$ is idempotent complete, the functor $\Hom_\mathscr{S}(G,-): {\langle G \rangle}_1^{\{0\}}\rightarrow \prj{R}$ is an equivalence. As $R$ is semisimple, $R\modcat=\prj{R}$.
In particular, $\Hom_{\mathscr{S}}(G,X_0)\in\prj{R}$. Consequently, the morphism $f_{0,X_0}: G_{0,X_0}\to X_0$ can be chosen such that $\Hom_{\mathscr{S}}(G,f_{0,X_0})$ is an isomorphism. This implies $X_1\in\mathscr{S}^{\geqslant 0}$.
Since the construction of $X_1$ from $X_0$ does not change the original category
$\mathscr{S}^{\geqslant 0}$, the iterated procedure produces a triangle
$$\begin{tikzcd}
G_{j+1} \arrow[r] & X_0 \arrow[r, "g_0g_1\cdots g_j"] & X_{j+1} \arrow[r] & {G_{j+1}[1]}
\end{tikzcd}$$
satisfying that $G_{j+1}\subseteq\langle G\rangle_{j+1}^{[0,\infty)}$,  $X_{j+1}\in\mathscr{S}^{\geqslant 0}$ and $g_i\in\mathcal{C}$ for all $0\leqslant i\leqslant j\in\mathbb{N}$.
Similarly, if $X_0\in\langle G\rangle_{n+1}$ for some $n\in\mathbb{N}$, then $X_0$ is a direct summand of $G_{n+1}$ and therefore $X_0\in\langle G\rangle_{n+1}^{[0,\infty)}$. This shows $\mathscr{S}^{\geqslant 0}\cap\langle G\rangle_{n+1}\subseteq\langle G\rangle_{n+1}^{[0,\infty)}$. Since $\mathscr{S}^{\geqslant 0}$ contains $G$ and is closed under negative shifts, extensions and direct summands in $\mathscr{S}$, we have
$\langle G\rangle_{n+1}^{[0,\infty)}\subseteq\mathscr{S}^{\geqslant 0}\cap\langle G\rangle_{n+1}$. Consequently,
$\langle G\rangle_{n+1}^{[0,\infty)}=\mathscr{S}^{\geqslant 0}\cap\langle G\rangle_{n+1}$. It follows that $$\mathscr{S}^{\geqslant 0}=\mathscr{S}^{\geqslant 0}\cap\langle G\rangle=\mathscr{S}^{\geqslant 0}\cap(\bigcup_{n\in\mathbb{N}}\langle G\rangle_{n+1})=\bigcup_{n\in\mathbb{N}}(\mathscr{S}^{\geqslant 0}\cap\langle G\rangle_{n+1})=\bigcup_{n\in\mathbb{N}}\langle G\rangle_{n+1}^{[0,\infty)}=\langle G\rangle^{[0,\infty)}.$$
Thus $\fd(\mathscr{S}, G)=0$.

The above proof also offers a different way to prove the inequality $\fd(\mathscr{S}, G)\leqslant n(d+1)$ if 
the category $\mathscr{S}$ in Proposition \ref{Generating properties} satisfies that $\mathscr{S}={\langle G \rangle}_{n+1}$ for some $n\in\mathbb{N}$ (compared with the proof of Proposition \ref{ST}).

In fact, since $\mathscr{S}={\langle G \rangle}_{n+1}$, we have $X_0\in {\langle G \rangle}_{n+1}$.
It follows from the property $(\ast\ast)$ that $g_0g_1\cdots g_n=0$. Consequently, $X_0$ is a summand of $G_{n+1}$ and therefore $X_0\in\langle G\rangle_{n+1}^{[-n(d+1),\infty)}\subseteq\langle G\rangle^{[0,\infty)}[n(d+1)]$. Thus $G(-\infty,-1]^{\bot}\subseteq \langle G\rangle^{[0,\infty)}[n(d+1)]$. This implies $\fd(\mathscr{S}, G)\leqslant n(d+1)$.
\end{proof}

\begin{Bsp}\label{differential graded}
Let $S$ be a DG algebra over a commutative ring $k$. Suppose that $H^0(S)$ is a semisimple $k$-algebra, $H^n(S)=0$ for $n<0$, and $H^n(S)$ is a finitely generated $H^0(S)$-module for $n\in\mathbb{N}$.
We show that $\fd(\mathscr{D}(S)^c)<\infty$.

In fact, by \cite[Theorem 7.1]{KN}, the pair $\big(\langle S\rangle^{(-\infty, 0]}, \langle S\rangle^{[0, \infty)}\big)$ of full subcategories of $\mathscr{D}(S)^c$ is a bounded $t$-structure on $\mathscr{D}(S)^c$. This implies $(\langle S\rangle^{(-\infty, -1]})^{\bot}=\langle S\rangle^{[0, \infty)}.$ Since $(\langle S\rangle^{(-\infty, -1]})^{\bot}=S(-\infty, -1]^{\bot}$, we obtain $\fd(\mathscr{D}(S)^c, S)=0$. Thus $\fd(\mathscr{D}(S)^c)<\infty$.

We mention that if additionally $H^n(S)=0$ for $n\gg 0$, then the triangulated category $\mathscr{D}(S)^c$ with the classical generator $S$ satisfies the assumptions of Proposition \ref{Generating properties}.
In this case, we can directly show that $\fd(\mathscr{D}(S)^c)<\infty$ by Proposition \ref{Generating properties}.
\end{Bsp}

To establish the finiteness of the finitistic dimension of the derived category of perfect complexes on a general scheme, we propose the following notion which is a generalization of finite-dimensional, noetherian schemes.

\begin{Def}\label{SMFFD}
Let $X$ be a quasicompact, quasiseparated scheme. We say that $X$ has \emph{finite finitistic dimension} if it has a finite affine open covering $X=\bigcup_{i=1}^nV_i$ such that for each $i$, $V_i$ is isomorphic to the spectrum of $R_i$ for some commutative ring $R_i$ of finite finitistic dimension.
\end{Def}

Now, we state a finiteness result on the finitistic dimension for schemes.

\begin{Prop}\label{SCMFD}
Let $X$ be a quasicompact, quasiseparated scheme, and $Z\subseteq X$ a closed subset with quasicompact complement. Suppose that $X$ has finite finitistic dimension. Then $\fd(\mathscr{D}^{\rm perf}_Z(X))<\infty$ and $\fd(\mathscr{D}^{\rm perf}_Z(X)^{\rm op})<\infty$.
\end{Prop}

\begin{proof}
Let $X=\bigcup_{i=1}^sV_i$ be a finite affine open covering of $X$, where $V_i\simeq {\rm Spec}(R_i)$,
the spectrum of $R_i$ for some commutative ring $R_i$ with $\fd(R_i)<\infty$ for each $i$.
Define $m:={\rm max}\{\fd(R_i)\mid 1\leqslant i\leqslant s\}.$ Since $X$ has finite finitistic dimension, we have $m<\infty$. Let $(-)^{\vee}:=\mathscr{R}\mathscr{H}om(-, \mathscr{O}_X)$, the right derived functor defined by the structure sheaf $\mathscr{O}_X$. We first prove the inclusion
$$(\ast):\;\; \big(\mathscr{D}^{{\rm perf}}(X)\cap\mathscr{D}_{\rm qc}(X)^{\geqslant 0}\big)^{\vee}\subseteq\mathscr{D}^{{\rm perf}}(X)\cap\mathscr{D}_{\rm qc}(X)^{\leqslant m}.$$
For this aim, we identify $\mathscr{D}_{\rm qc}(V_i)$ and  $\mathscr{D}^{\rm {perf}}(V_i)$
with $\mathscr{D}(R_i)$ and  $\Kb{\prj{R_i}}$ (up to triangle equivalence), respectively.
Let $\cpx{M}\in\mathscr{D}^{{\rm perf}}(X)\cap\mathscr{D}_{\rm qc}(X)^{\geqslant 0}$ and let $\cpx{M}_i$ be the restriction of $\cpx{M}$ to $V_i$. Then $\cpx{M}{^{\vee}}\in\mathscr{D}^{{\rm perf}}(X)$, and $\cpx{M}_i\in\Kb{\prj{R_i}}\cap \mathscr{D}(R_i)^{\geqslant 0}$.  It follows from $\fd(R_i)\leqslant m$ that
$\cpx{M}_i$ is isomorphic in $\mathscr{D}(R_i)$ to a bounded complex of finitely generated projective $R_i$-modules with nonzero terms concentrated in degrees $\geqslant -m$. Now, we denote by $(\cpx{M}{^{\vee}})_i$ the restriction of the complex $\cpx{M}{^{\vee}}$ to $V_i$. Then
$(\cpx{M}{^{\vee}})_i$ is isomorphic in $\mathscr{D}(R_i)$ to $\mathscr{R}\mathscr{H}om(\cpx{M}_i, R_i)$.
This implies that $(\cpx{M}{^{\vee}})_i\in \mathscr{D}(R_i)^{\leqslant m}$ for each $i$, and therefore
$\cpx{M}{^{\vee}}\in \mathscr{D}_{\rm qc}(X)^{\leqslant m}$. Thus the inclusion $(\ast)$ holds.

Let $\mathscr{T}:=\mathscr{D}_{{\rm qc}, Z}(X)$ and $\mathscr{S}:=\mathscr{D}^{{\rm perf}}_{Z}(X)$. It follows from Theorem \ref{CCST} that $\mathscr{T}^c=\mathscr{S}$,  $\mathscr{T}$ has a compact generator $G$ and the standard $t$-structure $(\mathscr{D}_{{\rm qc}, Z}(X)^{\leqslant 0}, \mathscr{D}_{{\rm qc}, Z}(X)^{\geqslant 0})$ on $\mathscr{T}$ is in the preferred equivalence class. Thus  $G$ is a classical generator of $\mathscr{S}$ by Lemma \ref{Smd} and there exists a natural number $r$ such that
$\mathscr{D}_{{\rm qc},Z}(X)^{\geqslant r}\subseteq\mathscr{T}_G^{\geqslant 0}\subseteq\mathscr{D}_{{\rm qc},Z}(X)^{\geqslant -r}.$
Note that the functor  $(-)^{\vee}: \big(\mathscr{D}^{{\rm perf}}(X)\big)\opp\to\mathscr{D}^{{\rm perf}}(X) $ is an equivalence of triangulated categories which restricts to an equivalence $\mathscr{S}\opp\to \mathscr{S}$. This implies that $G^{\vee}$ is a classical generator of $\mathscr{S}$. Since $\mathscr{T}$ is compactly generated by $\mathscr{S}$, the object $G^{\vee}$ is a compact generator of $\mathscr{T}$. Then the $t$-structures on $\mathscr{T}$ generated by $G$ and $G^{\vee}$ (see Example \ref{Typical example}) are equivalent, and therefore both $t$-structures are in the preferred equivalence class. Thus there exists a natural number $s$ with
$\mathscr{D}_{{\rm qc},Z}(X)^{\leqslant -s}\subseteq
\mathscr{T}_{G^\vee}^{\leqslant 0}\subseteq\mathscr{D}_{{\rm qc},Z}(X)^{\leqslant s}.$
Let $X\in G(-\infty, -1]^{\bot}\cap\mathscr{S}$. Since $\mathscr{T}_G^{\geqslant 0}=G(-\infty, -1]^{\bot}$, we have
$X\in\mathscr{T}_G^{\geqslant 0}\cap\mathscr{S}\subseteq\mathscr{D}_{{\rm qc},Z}^{\geqslant -r}\cap\mathscr{S}$.
It follows from $(\ast)$ that
$$X^{\vee}\in (\mathscr{D}_{{\rm qc},Z}^{\geqslant -r}\cap \mathscr{S})^{\vee}=(\mathscr{D}_{{\rm qc},Z}^{\geqslant -r}\cap \mathscr{D}^{{\rm perf}}(X))^{\vee}\subseteq\mathscr{S}\cap
\mathscr{D}_{{\rm qc},Z}(X)^{\leqslant r+m}\subseteq\mathscr{S}\cap\mathscr{T}_{G^\vee}^{\leqslant r+s+m}.$$
Since $G^{\vee}$ is a compact generator of $\mathscr{T}$, we see from Lemma \ref{Compact generator} that
$\mathscr{S}\cap\mathscr{T}_{G^\vee}^{\leqslant r+s+m}=\langle G^\vee\rangle^{(-\infty,\, r+s+m]}.$ This forces
$X^{\vee}\in\langle G^\vee\rangle^{(-\infty, \,r+s+m]}$, and therefore $X\simeq (X^{\vee})^{\vee}\in
\langle G\rangle^{[-(r+s+m), \infty)}$. Thus
$$G(-\infty, -1]^{\bot}\cap\mathscr{S}\subseteq \langle G\rangle^{[-(r+s+m), \infty)}=\langle G\rangle^{[0, \infty)}[r+s+m].$$
This implies $\fd(\mathscr{S}, G)\leqslant r+s+m<\infty$ and shows $\fd(\mathscr{S})<\infty$.
Since $\mathscr{S}$ and  $\mathscr{S}^{\rm op}$ are equivalent as triangulated categories,  $\fd(\mathscr{S}^{\rm op})<\infty$.
\end{proof}


\subsection{Proofs of consequences of our main result}\label{CMGE}

With the preparations in Section \ref{FDTC}, we now give proofs of a series of consequences of Theorem \ref{Main result}.

\begin{Koro}\label{Compactly generated}
Let $\mathscr{T}$ be a compactly generated triangulated category which has a compact generator $G$.
Suppose that $\Hom_\mathscr{T}(G, G[i])=0$ for $i\gg 0$ and $(\mathscr{T}^{c})\opp$ has finite finitistic dimension.

$(1)$ If $\mathscr{T}^{c}$ has a bounded $t$-structure, then $\mathscr{T}^{c}=\mathscr{T}_c^b$.

$(2)$ If $\mathscr{X}$ is a full triangulated subcategories of $\mathscr{T}$ with $\mathscr{T}^{c}\subseteq \mathscr{X}\subseteq \mathscr{T}_c^b$, then all bounded $t$-structures on $\mathscr{X}$ are equivalent.
\end{Koro}

\begin{proof}
Let $\mathscr{S}:=\mathscr{T}^{c}$. By Lemma \ref{Smd}, $\mathscr{S}=\langle G\rangle$.
This implies $\mathscr{S}\opp=\langle G\opp\rangle$. Since $\fd(\mathscr{S}\opp)<\infty$, it follows from Corollary \ref{Some-Any} that $\fd(\mathscr{S}\opp, G\opp)<\infty$. Now, we consider the inclusion $F:\mathscr{S}\subseteq\mathscr{T}$ which is a good extension. By Lemma \ref{Compact generator}(2),  $\{\mathscr{S}\cap\mathscr{T}^{\leqslant -n}\}_{n\in\mathbb{N}}$ is a $G$-good metric on $\mathscr{S}$.
Moreover, by Theorem \ref{MCP}(1), the functor $\mathfrak{y}_F:\mathscr{T}\to\mathscr{S}\Modcat$ (see Definition \ref{Good extension}) restricts to a triangle equivalence $\widehat{\mathfrak{S}}(\mathscr{S})\rightarrow \mathfrak{S}_G(\mathscr{S})$.
Since the restriction of $\mathfrak{y}_F$ to $\mathscr{S}$ is exactly the Yoneda functor $\mathfrak{y}$. the image of $\mathscr{S}$ under $\mathfrak{y}$ is equal to $\mathfrak{S}_G(\mathscr{S})$ if and only if $\mathscr{S}=\widehat{\mathfrak{S}}(\mathscr{S})$.  Further, $\widehat{\mathfrak{S}}(\mathscr{S})=\mathscr{T}_c^b$ by Theorem \ref{MCP}(2)(b). Thus Corollary \ref{Compactly generated} follows from Theorem \ref{Main result}.
\end{proof}

Next, we apply Corollary \ref{Compactly generated} to derived categories of
schemes or ordinary rings. For the analog of those results for connective $\mathbb{E}_1$-ring spectra,
or in particular, nonpositive DG rings, we refer to Corollaries \ref{LCR} and \ref{b6-dg}.

\begin{Koro}\label{S-1}
Let $X$ be a quasicompact, quasiseparated scheme and let $Z$ be a closed subset of $X$ such that $X-Z$ is quasicompact. Suppose that $X$ has finite finitistic dimension. Then:

$(1)$ If $\mathscr{D}^{{\rm perf}}_Z(X)$  has a bounded $t$-structure, then $\mathscr{D}^{{\rm perf}}_Z(X)=\mathscr{D}^{p, b}_{{\rm qc}, Z}(X)$. In particular, if $X$ is noetherian, then  $\mathscr{D}^{{\rm perf}}_Z(X)$ has a bounded $t$-structure if and only if $Z$ is contained in the regular locus of $X$.

$(2)$ All bounded $t$-structures on any triangulated category between $\mathscr{D}^{{\rm perf}}_Z(X)$ and $\mathscr{D}^{p, b}_{{\rm qc}, Z}(X)$ are equivalent. In particular, if $X$ is noetherian, then all bounded $t$-structures on $\mathscr{D}^{b}_{{\rm coh}, Z}(X)$ are equivalent.
\end{Koro}

\begin{proof}
Let $\mathscr{T}:=\mathscr{D}_{{\rm qc}, Z}(X)$ and $\mathscr{S}:=\mathscr{D}^{\rm perf}_Z(X)$. By Theorem \ref{CCST}, $\mathscr{T}^c=\mathscr{S}$. By Example \ref{Perfect complex}, $\mathscr{T}_c^b=\mathscr{D}^{p, b}_{{\rm qc}, Z}(X)$. Moreover, by Proposition \ref{SCMFD}, $\fd(\mathscr{S}\opp)<\infty$. Thus the first parts of Corollary \ref{S-1}(1) and Corollary \ref{S-1}(2) follow from Corollary \ref{Compactly generated}(1) and Corollary \ref{Compactly generated}(2), respectively. The assertions in Corollary \ref{S-1} for a noetherian scheme $X$ are true by combining the facts:  $\mathscr{D}^{p, b}_{{\rm qc}, Z}(X)=\mathscr{D}^{b}_{{\rm coh}, Z}(X)$; $Z$ is contained in the regular locus of $X$ if and only if $\mathscr{D}^{{\rm perf}}_Z(X)=\mathscr{D}^{b}_{{\rm coh}, Z}(X)$; the category $\mathscr{D}^{b}_{{\rm coh}, Z}(X)$ has an obvious bounded $t$-structure.
\end{proof}

\begin{Koro}\label{R-1}
Let $R$ be a ring. Suppose $\fd(R\opp)<\infty$. Then:

$(1)$ If $\Kb{\prj{R}}$ has a bounded $t$-structure, then $\Kb{\prj{R}}=\mathscr{K}^{-, b}(\prj{R})$. In particular, if $R$ is left coherent, then
$\Kb{\prj{R}}$ has a bounded $t$-structure if and only if $\Kb{\prj{R}}=\Db{R\modcat}$.

$(2)$ All bounded $t$-structures on any triangulated category between $\Kb{\prj{R}}$ and $\mathscr{K}^{-,b}(\prj{R})$ are equivalent.
\end{Koro}

\begin{proof}
Let $\mathscr{T}:=\D{R}$ and $\mathscr{S}:=\Kb{\prj{R}}$. Then $\mathscr{T}^c=\mathscr{S}$.
By Example \ref{Ordinary ring}, $\mathscr{T}_c^b=\mathscr{K}^{-, b}(\prj{R}).$
Since $\fd(R\opp)<\infty$ and $\mathscr{S}\opp$ as a triangulated category is equivalent to $\Kb{\prj{R\opp}}$, we have $\fd(\mathscr{S}\opp)<\infty$ by Lemma \ref{Finite}(5). By Corollary \ref{Compactly generated}(1), the first assertion of Corollary \ref{R-1}(1) holds. If $R$ is left coherent, then $R\modcat$ is an abelian category and $\mathscr{K}^{-, b}(\prj{R})=\Db{R\modcat}$ (up to triangle equivalence) which has an obvious bounded $t$-structure. Thus the second assertion of Corollary \ref{R-1}(1) holds. Corollary \ref{R-1}(2) is a direct consequence of Corollary \ref{Compactly generated}(2).
\end{proof}

\begin{Koro}\label{AT-1}
Let $\mathscr{S}$ be an essentially small triangulated category. If $\mathscr{S}$ has an algebraic $t$-structure, then $\mathscr{S}=\mathfrak{S}_G(\mathscr{S})$ for any classical generator $G$ of $\mathscr{S}$.
\end{Koro}

\begin{proof}
Suppose that $\mathscr{S}$ has an algebraic $t$-structure $(\mathscr{S}^{\leqslant 0},\mathscr{S}^{\geqslant 0})$. Then $\mathscr{S}\opp$ has an algebraic $t$-structure $\big((\mathscr{S}^{\geqslant 0})\opp, (\mathscr{S}^{\leqslant 0})\opp\big)$. Let $H$ be the direct sum of the isomorphism classes of simple objects in the heart of $(\mathscr{S}^{\leqslant 0}, \mathscr{S}^{\geqslant 0})$.
Then $H\opp$ is a classical generator of $\mathscr{S}\opp$, and $\fd(\mathscr{S}\opp, H\opp)<\infty$ by Lemma \ref{Finite}(3). Since $\langle H \rangle=\langle G \rangle$, Theorem \ref{Main result}$(a)$ implies Corollary \ref{AT-1}.
\end{proof}

Corollary \ref{AT-1} generalizes the implication of $(ii)$ to $(i)$ in \cite[Proposition 4.12]{AMY} which deals with $\Kb{\prj{R}}$ for a finite-dimensional algebra over a field.
Note that, in Theorem \ref{Main result}, the only condition on our triangulated category is on the finiteness of the finitistic dimension. One may wonder to what extent Theorem \ref{Main result} $(a)$ and $(b)$ hold without this finiteness assumption. A completion-invariant triangulated category can have bounded t-structures without having finite finitistic dimension. We give an example of such a category below. In this example, the conclusion of Theorem \ref{Main result}$(b)$ is not satisfied.

\begin{Prop}\label{CE}
Let $R:=k[x_1,x_2,x_3,\cdots]$ be the polynomial ring in countably many variables over a field $k$. Then:

$(1)$ $R$ is coherent, $\fd(R)$ is infinite and $\Kb{\prj{R}} = \Db{R\modcat}$. Thus the singularity category of $R$ is trivial.

$(2)$ $\Kb{\prj{R}}$ has two bounded $t$-structures which are not equivalent. In particular, one of the bounded $t$-structures on $\Kb{\prj{R}}$ generates a $t$-structure on $\D{R}$ that is not in the preferred equivalence class.
\end{Prop}

\begin{proof}
$(1)$ It is known that $R$ is coherent. To show $(1)$, it suffices to show that each finitely presented $R$-module has finite projective dimension. Since $R\modcat$ is an abelian subcategory of $R\Modcat$ and finitely generated submodules of finitely presented $R$-modules are finitely presented, we only need to show that each finitely generated ideal $I$ of $R$ has finite projective dimension.

For an $R$-module $X$, we denote by $\pd(_RX)$ the projective dimension of $_RX$.
Let $I:=\Sigma_{i=1}^n Rr_i$ be the ideal of $R$ generated by finitely many elements $r_1,r_2,\cdots,r_n$ of $R$. Then there exists a positive integer $m$ such that $\{r_1,r_2,\cdots,r_n\}\subseteq S:=k[x_1, x_2,\cdots, x_m]$. Let $J:=\Sigma_{i=1}^n Sr_i$ that is a finitely generated ideal of $S$. Note that $S$ has global dimension $m$, and therefore $\pd(_SJ)\leqslant m-1$. Since $R$ is a free $S$-module, $R\otimes_S J\simeq I$ as $R$-modules and further ${_S}I$ is a coproduct of countably many copies of ${_S}J$. This implies that $\pd({_R}I)=\pd({_S}J)\leqslant m-1$. Thus $R$ is regular. For $n\geqslant 1$, let $S_n:=k[x_1, x_2,\cdots x_n]$. Then $\pd(_R\Sigma_{i=1}^n Rx_i)=\pd(_{S_n}\Sigma_{i=1}^n S_nx_i)=\pd(_{S_n}k)-1=n-1$. It follows that $\pd(_R\Sigma_{i=1}^n Rx_i)=n-1$, and therefore $\fd(R)\geqslant n-1$. This forces $\fd(R)=\infty$.

$(2)$ Let $\mathscr{S}:=\Db{R\modcat}$, $\mathscr{A}:=\{\cpx{X}\in\mathscr{S}\mid H^i(\cpx{X})=0, \forall\; i>0\}$ and $\mathscr{B}:=\{\cpx{X}\in\mathscr{S}\mid H^i(\cpx{X})=0, \forall\; i<0\}$. Then $(\mathscr{A}, \mathscr{B})$ is a bounded $t$-structure on $\mathscr{S}$. Let $(-)^\vee=\rHom_R(-, R):\mathscr{S}\to \mathscr{S}$ be the right derived functor defined by the $R$-module $R$.
This functor is an auto-duality of triangulated categories, and thus $(\mathscr{B}^\vee, \mathscr{A}^\vee)$ is also a bounded $t$-structure on $\mathscr{S}$. Now, we show that these two $t$-structures on $\mathscr{S}$ are not equivalent.

For an element $x\in R$, we denote by $K(x)$ the two-term complex $0\to R\lraf{\cdot x} R\to 0$ with nonzero terms in degrees $-1$ and $0$ and with the differential given by the multiplication by $x$. For $n\geqslant 1$, the Koszul complex of the regular sequence $\{x_1, x_2,\cdots, x_n\}$ in $R$ is defined as the tensor complex $K_n:=K(x_1)\otimes_RK(x_2)\otimes_R\cdots\otimes_RK(x_n)$ of all these $K(x_i)$ over $R$. Let $I_n:=\Sigma_{i=1}^n Rx_n$. It is known that ${_R}R/I_n$ is quasi-isomorphic to $K_n$.
Clearly, $R/I_n\in \mathscr{B}$ and thus $K_n^\vee\in\mathscr{B}^\vee$.
Since $K(x_i)^\vee=\Hom_R(K(x_i), R)\simeq K(x_i)[-1]$, there are isomorphisms
$K_n^\vee\simeq K(x_1)^\vee\otimes_RK(x_2)^\vee\otimes_R\cdots\otimes_RK(x_n)^\vee\simeq K_n[-n]\simeq (R/I_n)[-n]$ in $\mathscr{S}$.  This implies $H^n(K_n^\vee)\simeq R/I_n\neq 0$. Consequently,
there is no positive integer $t$ such that $\mathscr{B}^\vee[t]\subseteq \mathscr{A}$.
So, $(\mathscr{A}, \mathscr{B})$ and $(\mathscr{B}^\vee, \mathscr{A}^\vee)$ are not equivalent.

Let $\mathscr{D}':=(\D{R}^{\leqslant 0}, \D{R}^{\geqslant 0})$ be the standard $t$-structure on $\D{R}$. It is in the preferred equivalence class and its restriction to $\mathscr{S}$ is the $t$-structure
$(\mathscr{A}, \mathscr{B})$. Since $\mathscr{B}^\vee$ consists of compact objects, it generates a $t$-structure $\mathscr{D}:=\big(\text{Coprod}(\mathscr{B}^\vee),\; (\text{Coprod}(\mathscr{B}^\vee)[1])^{\perp}\big)$ on $\D{R}$ by Example \ref{Typical example}.
Moreover, by \cite[Lemma 2.5]{Neeman4}, the restriction of $\mathscr{D}$ to $\mathscr{S}$ is equal to $(\mathscr{B}^\vee, \mathscr{A}^\vee)$. Since $(\mathscr{A}, \mathscr{B})$ and $(\mathscr{B}^\vee, \mathscr{A}^\vee)$ are not equivalent, $\mathscr{D}'$ and $\mathscr{D}$ are not equivalent. Thus  $\mathscr{D}$ is not in the preferred equivalence class.
\end{proof}

We also borrow an example of a triangulated category from \cite{Muro} which is neither algebraic nor topological, and show that this category has no bounded t-structure, but its completion is zero.

\begin{Bsp}
Let $R:=\mathbb{Z}/4\mathbb{Z}$ and let $\mathscr{S}$ be the category of finitely generated free $R$-modules. In \cite{Muro}, it was shown that $\mathscr{S}$ is a triangulated category with the identity functor as its shift functor, and more surprisingly, it is neither the stable category of a Frobenius category nor a full triangulated subcategory of the homotopy category of a stable model category.

Clearly, the module $R$ is a classical generator of $\mathscr{S}$ and $R[n]=R$ for any integer $n$. This implies that $\Hom_{\mathscr{S}}(R[n],R)=\End_{\mathscr{S}}(R)\simeq R\neq 0$, and therefore
$\mathscr{S}_{\rm tc}=0$. By Lemma \ref{HBTS}, $\mathscr{S}$ has no bounded $t$-structure. However, the completion $\mathfrak{S}_R(\mathscr{S})$ of $\mathscr{S}$ with respect to the $R$-good metric $\{\mathscr{M}_n\}_{n\in \mathbb{N}}$, where $\mathscr{M}_n=\mathscr{S}$ for all $n$, is zero.
Thus $\mathscr{S}\neq \mathfrak{S}_R(\mathscr{S})$, but the almost singularity category of $\mathscr{S}$ vanishes (see Definition \ref{Singularity}).
In this example, we also have $\fd(\mathscr{S}, R)=0=\fd(\mathscr{S}^{\opp}, R^{\opp})$, since $R$ is commutative and $R[n]^{\bot}=R^{\bot}=0$ in $\mathscr{S}$.
Note that Krause's completion (see \cite{Kr3}) does not apply to non-algebraic triangulated categories, and therefore does not apply here.
\end{Bsp}

\medskip
{\bf Acknowledgements.} H. X. Chen and J. H. Zheng would like to thank Xiaohu Chen and Yaohua Zhang for discussions on the completions of triangulated categories. The work of H. X. Chen was supported by the National Natural Science Foundation of China (Grant 12122112 and 12031014). R. Biswas would like to thank the BIREP group at Universit{\"a}t Bielefeld where he was a postdoc when he started working on this project, and the Hausdorff Research Institute for Mathematics for their hospitality during the ``Spectral Methods in Algebra, Geometry, and Topology" Trimester Program funded by the Deutsche Forschungsgemeinschaft under Germany’s Excellence Strategy – EXC-2047/1 – 390685813. The work of C. J. Parker was supported by the Deutsche Forschungsgemeinschaft (SFB-TRR 358/1 2023 - 491392403). K. Manali Rahul was partially supported by the same SFB grant, the ERC Advanced Grant 101095900-TriCatApp, the Australian Research Council Grant DP200102537, and is a recipient of the AGRTP scholarship. He would also like to thank Università degli Studi di Milano and the BIREP group at Universit\"{a}t Bielefeld for their hospitality.

\begin{appendices}

\section{Completions of perfect modules over connective ring spectra}\label{ECTC}

In this appendix, we calculate the completion of the homotopy category of perfect modules over a connective ring spectrum (Theorem \ref{Connective case} and Corollary \ref{Spectrum}), which generalizes the case of the sphere spectrum. For an introduction to structured ring spectra and their module spectra, we refer to \cite[Chapter 7]{JL}.

For a spectrum $E$, we denote by $\pi_n(E)$  the $n$-th homotopy group of $E$ for each $n\in\mathbb{Z}$.  Recall that an \emph{$\mathbb{E}_1$-ring} $R$ is by definition an $\mathbb{E}_1$-algebra object in the $\infty$-category ${\rm Sp}$ of spectra. In other words, $R$ is a spectrum equipped with a multiplication which is associative up to coherent homotopy. An $\mathbb{E}_1$-ring $R$ is said to be \emph{connective} if $\pi_n(R)=0$ for all $n<0$. An important class of connective $\mathbb{E}_1$-rings is the \emph{Eilenberg\mbox{–}Mac Lane ring spectrum} $H\Lambda$ of an ordinary ring $\Lambda$, where $\pi_n(H\Lambda)=0$ for $n\neq 0$ and $\pi_0(H\Lambda)=\Lambda$.

Let $R$ be an $\mathbb{E}_1$-ring. We denote by ${\rm LMod}_R$ the \emph{stable $\infty$-category of left $R$-module spectra}.  Unless stated otherwise, all module spectra in this section are left module spectra. It is known that ${\rm LMod}_R$ is a compactly generated stable $\infty$-category that has $_RR$ as a compact generator. We denote by ${\rm LMod}_R^{\rm perf}$ the smallest stable subcategory of ${\rm LMod}_R$ containing  ${_R}R$ and closed under direct summands (or retracts in other terminology). An $R$-module $M$ is said to be  \emph{perfect} if it belongs to ${\rm LMod}_R^{\rm perf}$. Roughly speaking, an $R$-module $M$ is perfect if it can be obtained as a successive extension of finitely many (possibly shifted) copies of $R$ or is a direct summand of such an $R$-module. By \cite[Proposition 7.2.4.2]{JL}, an object of ${\rm LMod}_R$ is compact if and only if it is perfect.

For a general stable $\infty$-category $\mathcal{C}$, we denote by ${\rm Ho}(\mathcal{C})$ the \emph{homotopy category} of $\mathcal{C}$. This is a triangulated category (for example, see \cite[Section 1.1.2]{JL} or \cite[Section 2]{BGT}). A morphism $f:X\to Y$ in $\mathcal{C}$ is called an \emph{equivalence} if it is an isomorphism in
${\rm Ho}(\mathcal{C})$. Following \cite[Definition 1.2.1.4]{JL}, a \emph{$t$-structure} on $\mathcal{C}$ is by definition a $t$-structure on ${\rm Ho}(\mathcal{C})$. For $t$-structures, connective $\mathbb{E}_1$-rings stand out since their module spectra have canonical $t$-structures.

From now on, let $R$ be a connective $\mathbb{E}_1$-ring. For each integer $n$, there is a pair of $\infty$-categories:
$$
 ({\rm LMod}_R)_{\geqslant n}:=\{M\in {\rm LMod}_R |\ \pi_i(M)=0,\ \forall\; i<n \},\;\;
({\rm LMod}_R)_{\leqslant n}:=\{M\in {\rm LMod}_R |\ \pi_i(M)=0,\ \forall\; i>n \},
$$
admitting the following nice property (see {\rm \cite[Proposition 7.1.1.13]{JL}}).

\begin{Lem}\label{LMod}
The pair $\big(({\rm LMod}_R)_{\geqslant n},({\rm LMod}_R)_{\leqslant n}  \big)$ is a $t$-structure on ${\rm LMod}_R$ with the heart equivalent to the nerve of $\pi_0(R)\Modcat $.
Moreover, $({\rm LMod}_R)_{\leqslant n}$ and $({\rm LMod}_R)_{\geqslant n}$
are stable under (small) products and filtered colimits in $ {\rm LMod}_R$.
\end{Lem}

By Lemma \ref{LMod}, the inclusion $({\rm LMod}_R)_{\geqslant n} \to {\rm LMod}_R$ has a right adjoint $\tau_{\geqslant n}: {\rm LMod}_R\to ({\rm LMod}_R)_{\geqslant n}$, the inclusion $({\rm LMod}_R)_{\leqslant n} \to {\rm LMod}_R$ has a left adjoint $\tau_{\leqslant n}:{\rm LMod}_R\to ({\rm LMod}_R)_{\leqslant n}$, and both inclusions preserve filtered colimits. Thus $\tau_{\leqslant n}$ preserves compact objects and each $R$-module $M$ is endowed with a fiber sequence $\tau_{\geqslant n}(M)\to M\to \tau_{\leqslant n-1}(M)$ such that $\tau_{\geqslant n}(M)\in({\rm LMod}_R)_{\geqslant n}$ and $\tau_{\leqslant n-1}(M)\in ({\rm LMod}_R)_{\leqslant n-1}$.  For a morphism $f:M\to N$ in ${\rm LMod}_R$, we denote by ${\rm fib}(f)$ and ${\rm cofib}(f)$ the \emph{fiber} and \emph{cofiber} of $f$, respectively.  Then ${\rm fib}(f)[1]$ is equivalent to ${\rm cofib}(f)$.

In general, the condition that an $R$-module is perfect is very strong. The following definition generalizes the concept of perfect modules over connective $\mathbb{E}_1$-rings.
	
\begin{Def}{\rm \cite[Definition 7.2.4.10]{JL}}
An $R$-module  $M$ is \emph{almost perfect} if  there exists an integer $k$ with $M\in ({\rm LMod}_R)_{\geqslant k} $ such that, for each $n\in\mathbb{N}$, the $R$-module $\tau_{\leqslant n}(M)$ is a compact object of the $\infty$-category $({\rm LMod}_R)_{\geqslant k}\cap ({\rm LMod}_R)_{\leqslant n}$.
\end{Def}

The full subcategory of ${\rm LMod}_R$ consisting of almost perfect $R$-modules is denoted by
${\rm LMod}_R^{\rm aperf}$. By \cite[Proposition 7.2.4.11(1)]{JL}, ${\rm LMod}_R^{\rm {aperf}}$
is a stable subcategory of ${\rm LMod}_R$ containing ${\rm LMod}_R^{\rm perf}$ and closed under direct summands. In general, ${\rm LMod}_R^{\rm perf}$ is not easy to describe. But for a left coherent $\mathbb{E}_1$-ring $R$, it can be characterized in terms of objects of $\pi_0(R)\modcat$, \emph{the category of finitely presented left modules over the ring $\pi_0(R)$}.

\begin{Def}{\rm \cite[Definition 7.2.4.16]{JL}}
A connective $\mathbb{E}_1$-ring $R$ is \emph{left coherent} if $\pi_0(R)$ is left coherent as an ordinary ring and $\pi_n(R)$ is a finitely presented left $\pi_0(R)$-module for any $n\in\mathbb{N}$.
\end{Def}

By \cite[Propositions 7.2.4.17]{JL}, an $R$-module $M$ over a left coherent $\mathbb{E}_1$-ring $R$ is \emph{almost perfect} if and only if $\pi_n(M)\in\pi_0(R)\modcat$ for $n\in\mathbb{Z}$ and $\pi_n(M)=0$ for $n\ll 0$.

In the following, we concentrate on some special classes of left coherent $\mathbb{E}_1$-rings.

\begin{Def}{\rm \cite[Definitions 1.1, 1.2, 1.4]{BL}}\label{truncated}
Let $R$ be a left coherent $\mathbb{E}_1$-ring. An $R$-module $M$ is said to be \emph{truncated} if $\pi_n(M)=0$ for $n\gg 0$; and \emph{coherent} if $M$ is both truncated and almost perfect, that is,
$\pi_n(M)\in\pi_0(R)\modcat$ for $n\in\mathbb{Z}$ and $\pi_n(M)=0$ for $|n|\gg 0$.

The $\mathbb{E}_1$-ring $R$ is said to be \emph{truncated} if ${_R}R$ is truncated; \emph{almost regular} if each coherent $R$-module is perfect; and \emph{regular} if $\pi_0(R)$ is left regular  and $H\pi_0(R)$ as a left $R$-module is perfect.
\end{Def}

Coherent $R$-modules are referred to \emph{bounded $R$-module spectra} in \cite[Appendix A]{Kr3} that are pseudo-coherent $R$-modules with bounded homotopy (see \cite[Proposition A1]{Kr3}).
By \cite[Proposition 1.3]{BL}, regular $\mathbb{E}_1$-rings are exactly almost regular $\mathbb{E}_1$-rings $R$ with $\pi_0(R)$ a left regular ring. Examples of regular $\mathbb{E}_1$-rings contain the connective real $K$-theory spectrum, the topological modular forms spectrum and the truncated Brown-Peterson spectrum (see \cite{BL} for details).

Now, we consider the following categories, of which the first three are triangulated categories:
$$\mathscr{T}:={\rm Ho(LMod}_R),\; \;\mathscr{S}:={\rm Ho(LMod}_R^{\rm perf}),\;\;
\mathscr{S}^{\rm ap}:={\rm Ho(LMod}_R^{\rm aperf}),$$
$$\mathscr{T}^{\leqslant 0}:=\{M\in {\rm Ho(LMod}_R) |\ \pi_i(M)=0,\ \forall\; i<0\},\;\;
\mathscr{T}^{\geqslant 0}:=\{M\in {\rm Ho(LMod}_R) |\ \pi_i(M)=0,\ \forall\; i>0\}.
$$
Since $R$ is connective, the module ${_R}R$ is a compact generator of $\mathscr{T}$ and in
$\mathscr{T}^{\leqslant 0}$. This implies $\mathscr{S}=\langle {_R}R\rangle\subseteq\mathscr{T}^{-}$.  By Lemma \ref{LMod}, the pair $(\mathscr{T}^{\leqslant 0}, \mathscr{T}^{\geqslant 0})$ is a $t$-structure on $\mathscr{T}$ in the preferred equivalence class.

With the above preparations, we give the main result of this appendix as follows.

\begin{Theo}\label{Connective case}
Let $R$ be a connective $\mathbb{E}_1$-ring. Then:

$(1)$ $\mathfrak{L}'(\mathscr{S})=\mathscr{T}_c^{-}=\mathscr{S}^{\rm ap}$ and $\widehat{\mathfrak{S}}(\mathscr{S})=\mathscr{T}_c^b=\mathscr{S}^{\rm ap}\cap \mathscr{T}^b$.

$(2)$ $\mathscr{S}\subseteq \widehat{\mathfrak{S}}(\mathscr{S})$
if and only if $R$ is truncated.
\end{Theo}

\begin{proof}
$(1)$ Since $\Hom_{\mathscr{T}}(R,R[i])\simeq \pi_{-i}(R)=0$ for $i>0$, we see from Theorem \ref{MCP}(2) that $\mathfrak{L}'(\mathscr{S})=\mathscr{T}_c^{-}$ and $\widehat{\mathfrak{S}}(\mathscr{S})=\mathscr{T}_c^b$.
It suffices to show $\mathscr{T}_c^{-}=\mathscr{S}^{\rm ap}$.

Let $\mathscr{C}:=({\rm LMod}_R)_{\geqslant 0}\cap {\rm LMod}_R^{\rm {aperf}}$. By \cite[Proposition 7.2.4.11(5)]{JL} and its proof, for each $M\in\mathscr{C}$, there exists a sequence of morphisms
$$
\xymatrix{
D(0)\ar[r]^-{f_1}&D(1)\ar[r]^-{f_2}&D(2)\ar[r]&\cdots\ar[r]&D(n-1)
\ar[r]^-{f_n}&D(n)\ar[r]^-{f_{n+1}}&D(n+1)\ar[r]&\cdots}
$$
together with morphisms $g_n: D(n)\to M$ in ${\rm LMod}_R$, where $g_n=f_{n+1}g_{n+1}$, ${\rm fib}(g_n)\in ({\rm LMod}_R)_{\geqslant n}$ and ${\rm cofib}(f_n)[-n]$ is a free $R$-module of finite rank ($f_0$ denotes the zero map $0\to D(0)$) for $n\in\mathbb{N}$, such that the natural map $\mathop{{\rm colim}} \limits_{\longrightarrow}D(n)\to M$ is an equivalence.
This yields a triangle $\text{fib}(g_n)\to D(n)\to M\to \text{cofib}(g_n)$ in $\mathscr{T}$
with $D(n)\in \mathscr{S} $ and $\text{cofib}(g_n)\simeq \text{fib}(g_n)[1]\in(\text{LMod}_R)_{\geqslant n+1}$. Thus $M\in \mathscr{T}^-_c$. Since almost perfect $R$-modules can be obtained from $\mathscr{C}$ by taking shifts, we have $\mathscr{S}^{\rm ap}\subseteq \mathscr{T}_c^{-}$.

It remains to show $\mathscr{T}_c^{-}\subseteq\mathscr{S}^{\rm ap}$. For this aim, we take $X\in \mathscr{T}_c^{-}$. By Definition \ref{TC}, there exists a fiber sequence
$X_n\lraf{h_n} X\to Y_n$ in ${\rm LMod}_R$ with $X_n\in {\rm LMod}_R^{\rm perf}$ and $Y_n\in ({\rm LMod}_R)_{\geqslant n}$ for each $n\in\mathbb{N}$. We fix an integer $i$ and apply $\pi_i$ to this sequence.
Then for $n\geqslant i+2$, the map $\pi_i(h_n):\pi_i(X_n)\to \pi_i(X)$ is an isomorphism, due to $\pi_i(Y_n)=0=\pi_{i+1}(Y_n)$. Let $m_i:=\max\{0,i+2\}$. We first show that the morphism $\tau_{\leqslant i}(h_{m_i}):\tau_{\leqslant i}(X_{m_i})\to\tau_{\leqslant i}(X)$ is an equivalence in ${\rm LMod}_R$.

Clearly, $\pi_j(h_{m_i}): \pi_j(X_{m_i})\to\pi_j(X)$ are isomorphisms for $j\leqslant i$. Recall that, for an $R$-module $M$, the associated unit adjunction $M\to \tau_{\leqslant i}(M)$ satisfies that $\pi_j(M)\simeq \pi_j(\tau_{\leqslant i}(M))$ for any $j\leqslant i$. This implies that the maps
$\pi_j\big(\tau_{\leqslant i}(h_{m_i})\big): \pi_j\big(\tau_{\leqslant i}(X_{m_i})\big)\to\pi_j\big(\tau_{\leqslant i}(X)\big)$ are isomorphisms for $j\leqslant i$. As the $k$-th homotopy of any object of $({\rm LMod}_R)_{\leqslant i}$ vanishes for $k>i$, the maps   $\pi_j\big(\tau_{\leqslant i}(h_{m_i})\big)$ are isomorphisms for all integers $j$. Consequently, the fiber of $\tau_{\leqslant i}(h_{m_i})$ has no nonzero homotopy groups; in other words, it is equivalent to zero. Thus $\tau_{\leqslant i}(h_{m_i})$ is an equivalence in ${\rm LMod}_R$.

Since $X_{m_i}$ is compact in ${\rm LMod}_R$ and $\tau_{\leqslant i}:{\rm LMod}_R\to ({\rm LMod}_R)_{\leqslant i}$ preserves compact objects, $\tau_{\leqslant i}(X_{m_i})$ is compact in $({\rm LMod}_R)_{\leqslant i}$. It follows from the equivalence $\tau_{\leqslant i}(h_{m_i})$ that $\tau_{\leqslant i}(X)$ is also compact in $({\rm LMod}_R)_{\leqslant i}$. Observe that if $i\leqslant -2$, then $m_i=0$ and $\tau_{\leqslant i}(X_0)\to \tau_{\leqslant i}(X)$ is an equivalence. Since $X_0\in{\rm LMod}_R^{\rm perf}$ and $R\in ({\rm LMod}_R)_{\geqslant 0}$, there exists an integer $k\leqslant -1$ with $X_0\in({\rm LMod}_R)_{\geqslant k}$. Then
$\tau_{\leqslant k-1}(X_0)$ is equivalent to zero. Since $\tau_{\leqslant k-1}(X_0)\simeq\tau_{\leqslant k-1}(X)$, we see that $\tau_{\leqslant k-1}(X)$ is also equivalent to zero.
Now the fiber sequence $\tau_{\geqslant k}(X)\to X\to \tau_{\leqslant k-1}(X)$ implies $X\simeq \tau_{\geqslant k}(X)\in({\rm LMod}_R)_{\geqslant k}$.
Thus $X$ is almost perfect. This shows $\mathscr{T}_c^{-}\subseteq\mathscr{S}^{\rm ap}$.

$(2)$ Since $\mathscr{S}\subseteq \mathscr{S}^{\rm ap}$, we see from $(1)$ that $\mathscr{S}\subseteq \widehat{\mathfrak{S}}(\mathscr{S})$ if and only $\mathscr{S}\subseteq \mathscr{T}^b$.
Recall that $\mathscr{S}$ is generated by ${_R}R$. Thus  $\mathscr{S}\subseteq \mathscr{T}^b$ if and only if $R\in\mathscr{T}^b$. The latter is equivalent to the vanishing of $\pi_m(R)$ for $m\gg 0$. This shows $(2)$.
\end{proof}

A direct consequence of Theorem \ref{Connective case} is the following result. This conveys that our definition of almost regular triangulated categories in Section \ref{Section 1.1}, when specialized to
the case of connective $\mathbb{E}_1$-rings, is consistent with the one of almost regular $\mathbb{E}_1$-rings.

\begin{Koro}\label{Spectrum}
Let $R$ be a left coherent $\mathbb{E}_1$-ring. Then:

$(1)$ $\mathscr{S}_{\rm tc}$ (resp., $\widehat{\mathfrak{S}}(\mathscr{S})$) is the full triangulated subcategory of $\mathscr{T}$ consisting of all perfect and truncated $R$-modules (resp., all coherent $R$-modules).

$(2)$ $\mathscr{S}$ is almost regular (resp., regular) if and only if $R$ is almost regular (resp., $R$ is
almost regular and truncated). In particular, if $\pi_0(R)$ is left regular, then $\mathscr{S}$ is almost regular if and only if $R$ is regular.
\end{Koro}

Finally, left coherent rings in Corollary \ref{Ring case} can be generalized to left coherent $\mathbb{E}_1$-rings. As usual, for an $\mathbb{E}_1$-ring $R$, we denote by $R\opp$ the opposite of $R$ which is also an $\mathbb{E}_1$-ring.

\begin{Koro}\label{LCR}
Let $R$ be a left coherent $\mathbb{E}_1$-ring. Suppose that ${\rm Ho(LMod}_{R\opp}^{\rm perf})$ has finite finitistic dimension. Then the following statements are true.

$(1)$ $\mathscr{S}$ has a bounded $t$-structure if and only if $R$ is almost regular and truncated.

$(2)$ Suppose that $R$ is truncated. Then all bounded $t$-structures on any triangulated category between $\mathscr{S}$ and $\widehat{\mathfrak{S}}(\mathscr{S})$ are equivalent.
\end{Koro}

\begin{proof}
$(1)$ Since $R$ is connective, $\Hom_{\mathscr{S}}(R,R[i])\simeq \pi_{-i}(R)=0$ for all $i>0$. Moreover, the triangulated categories $\mathscr{S}\opp$ and ${\rm Ho(LMod}_{R\opp}^{\rm perf})$ are equivalent, due to \cite[Propositions 7.2.4.4]{JL}. So, the ``only if" part of Corollary \ref{LCR}(1) follows from Corollary \ref{CPG}(1) and Corollary \ref{Spectrum}(2).

Since $R$ is left coherent,  we see from  \cite[Propositions 7.2.4.18]{JL} that the pair of $\infty$-categories
$$\big({\rm LMod}_R^{{\rm aperf}}\cap({\rm LMod}_R)_{\geqslant 0},\; {\rm LMod}_R^{\rm aperf}\cap({\rm LMod}_R)_{\leqslant 0}\big)$$ is a $t$-structure on ${\rm LMod}_R^{{\rm aperf}}$.
Restricting this $t$-structure to the stable $\infty$-category of coherent $R$-modules, we obtain an obvious bounded $t$-structure on $\widehat{\mathfrak{S}}(\mathscr{S})$. If
$R$ is almost regular and truncated, then $\mathscr{S}=\widehat{\mathfrak{S}}(\mathscr{S})$ and thus $\mathscr{S}$ has a bounded $t$-structure. This shows the ``if" part of Corollary \ref{LCR}(1).
Note that, for the “if” part, we don't need the finite finitistic dimension assumption.

$(2)$ Since $R$ is truncated, $\mathscr{S}\subseteq \widehat{\mathfrak{S}}(\mathscr{S})$ by Theorem \ref{Connective case}(2). Now, Corollary \ref{LCR}(2) is a consequence of Corollary \ref{CPG}(2).
\end{proof}

\begin{Rem}\label{Known cases}
$(1)$ Let $R$ be the sphere spectrum. Then it is a left coherent $\mathbb{E}_1$-ring with $\pi_0(R)=\mathbb{Z}$. The description of $\widehat{\mathfrak{S}}(\mathscr{S})$ in Corollary \ref{Spectrum} has been given in \cite[Example 22]{Neeman2}. Since $\mathscr{S}_{\rm tc}=0$, there is no bounded $t$-structure on $\mathscr{S}$ by Lemma \ref{HBTS}. Moreover, by Definition \ref{Singularity}, the \emph{almost singularity category} of $\mathscr{S}$ is equivalent to $\widehat{\mathfrak{S}}(\mathscr{S})$.

$(2)$ It is shown in \cite[Corollary A2]{Kr3} that the completion of $\mathscr{S}$ in the sense of Krause is also equivalent to the triangulated subcategory of $\mathscr{T}$ consisting of all coherent $R$-modules by taking homotopy colimits. Interestingly, different types of completions of $\mathscr{S}$ produce the same triangulated category.
\end{Rem}

\section{Different notions of finitistic dimension of categories}\label{Section B}

In this appendix, we mention some other ways of defining finitistic dimension for triangulated categories. In some cases, these notions were only defined for specific classes of triangulated categories, and we extend these notions to general triangulated categories.

If higher extension groups of objects are considered, then the projective dimensions of objects and the finitistic dimensions of compactly generated triangulated categories can be defined. This imitates the definitions of projective dimensions of modules and the finitistic dimensions of ordinary rings, or more generally, of nonpositive DG rings (see \cite{BSSW}).

Let $\mathscr{T}$ be a compactly generated triangulated category which has a compact generator $G$.
We consider the $t$-structure $(\mathscr{T}^{\leqslant 0}_G, \mathscr{T}^{\geqslant 0}_G)$ on $\mathscr{T}$ generated by $G$. Denote by $\mathscr{T}^{b}_G$ the full subcategory of $\mathscr{T}$ consisting of all \emph{bounded} objects $X$, that is, there is a positive integer $n$ with $X[n]\in\mathscr{T}^{\leqslant 0}_G$ and $X[-n]\in \mathscr{T}^{\geqslant 0}_G$. Let $\mathscr{H}_G:=\mathscr{T}^{\leqslant 0}_G\cap \mathscr{T}^{\geqslant 0}_G$ be the \emph{heart} of the $t$-structure. Then each object of $\mathscr{T}^{b}_G$ can be obtained from objects of $\mathscr{H}_G$ by taking shifts and extensions.

\begin{Def}
The \emph{projective dimension} of an object $X$ in $\mathscr{T}^{b}_G$ with respect to $G$ is defined to
$$\pd_G(X):=\inf\{n\in\mathbb{Z}\mid \Hom_\mathscr{T}(X, Y[i])=0, \;\forall i>n,\; Y\in \mathscr{T}^{b}_G\cap\mathscr{T}^{\leqslant 0}_G \}$$
$$\quad\quad\;\;=\inf\{n\in\mathbb{Z}\mid \Hom_\mathscr{T}(X, Y[i])=0, \;\forall i>n,\; Y\in \mathscr{H}_G\}.$$
\end{Def}
Clearly, $\pd_G(X[n])=\pd_G(X)+n$ for any $n\in\mathbb{Z}$. Moreover, for a triangle $X_1\to X_2\to X_3\to X_1[1]$ in $\mathscr{T}$, $\pd_G(X_2)\leqslant \max\{\pd_G(X_1), \pd_G(X_3)\}$.
Thus the full subcategory of $\mathscr{T}^{b}_G$ consisting of all objects $X$ with
$\pd_G(X)<\infty$ is a triangulated subcategory of $\mathscr{T}^{b}_G$ closed under direct summands.

\begin{Def}\label{Def-BFD}
The \emph{big finitistic dimension} and \emph{finitistic dimension} of $\mathscr{T}$ with respect to $G$ are defined:
$$
{\rm FPD}(\mathscr{T}, G):=\sup\{\pd_G(X)\mid X\in \mathscr{T}^{b}_G\cap\mathscr{T}^{\geqslant 0}_G,\; \pd_G(X)<\infty\},
$$
$$
\quad\quad\quad {\rm fpd}(\mathscr{T}, G):=\sup\{\pd_G(X)\mid X\in \mathscr{T}^c\cap\mathscr{T}^{b}_G\cap\mathscr{T}^{\geqslant 0}_G,\; \pd_G(X)<\infty\},
$$
\end{Def}

Similar to Corollary \ref{Some-Any}, the finiteness of ${\rm FPD}(\mathscr{T}, -)$ and ${\rm fpd}(\mathscr{T}, -)$ is independent of the choice of different compact generators of $\mathscr{T}$.
Moreover, for a ring $R$, ${\rm FPD}(\D{R}, R)=\Fd(R)$ and ${\rm fpd}(\D{R}, R)=\fd(R).$
A close relationship between Definitions \ref{Def-FD} and \ref{Def-BFD} is given in the following proposition:

\begin{Prop}\label{Nonpositive}
Let $\mathscr{T}$ be a compactly generated triangulated category which has a (nonzero) compact generator $G$ with $\Hom_\mathscr{T}(G, G[i])=0$ for $|i|\gg 0$. Suppose that $(\mathscr{T}^c)\opp$ is triangle equivalent to $\mathscr{U}^c$, where $\mathscr{U}$ is a compactly generated triangulated category. Then:

$(1)$ ${\rm fpd}(\mathscr{T}, G)\leqslant \fd(\mathscr{T}^c, G)+a$, where
$a:=\inf\{n\in\mathbb{N}\mid \Hom_\mathscr{T}(G, G[i])=0, i>n\}$.

$(2)$ If $\mathscr{U}$ has a compact generator $H$ with
$H[b, \infty)^{\bot}\subseteq\mathscr{U}^{\leqslant 0}_H$ for some integer $b$ (for example, $\mathscr{U}$ is a weakly approximable triangulated category in the sense of Neeman, see \cite{Neeman5, CNS}), then there is an integer $c$ such that $\fd(\mathscr{T}^c, G)\leqslant{\rm fpd}(\mathscr{T}, G)+c$.

$(3)$ If $\Hom_\mathscr{T}(G, G[i])=0$ for all $i>0$, then
${\rm fpd}(\mathscr{T}, G)\leqslant \fd(\mathscr{T}^c, G)\leqslant\max\{0,\, {\rm fpd}(\mathscr{T}, G)\}.$
In particular, $\fd(\mathscr{T}^c, G)$ is finite if and only if ${\rm fpd}(\mathscr{T}, G)$ is finite.
\end{Prop}

\begin{proof}
We first show ${\rm fpd}(\mathscr{T}, G)=\sup\{\pd_G(X)\mid X\in\mathscr{T}^{c}\cap\mathscr{T}^{\geqslant 0}_G\}$.

By Example \ref{Typical example}, $\mathscr{T}^{\leqslant 0}_G=\overline{\langle G \rangle}^{(-\infty, 0]}$ and $\mathscr{T}^{\geqslant 0}_G=G(-\infty,-1]^{\bot}$. Since $\Hom_\mathscr{T}(G[i],G)=0$ for $i\gg 0$, we have $G\in\mathscr{T}^{b}_G$. Clearly, $G$ is a classical generator of $\mathscr{T}^c$. This implies $\mathscr{T}^c\subseteq\mathscr{T}^{b}_G$. Since $\Hom_\mathscr{T}(G, G[i])=0$ for $i>a$,
$G\in G[a+1, \infty)^{\bot}$. Observe that $G[a+1, \infty)^{\bot}\subseteq \mathscr{T}$ is closed under extensions, positive shifts and coproducts. Thus $\mathscr{T}^{\leqslant 0}_G\subseteq G[a+1, \infty)^{\bot}$. In particular, $\Hom_\mathscr{T}(G, Y[i])=0$ for $i>a$ and $Y\in \mathscr{T}^{b}_G\cap\mathscr{T}^{\leqslant 0}_G$. This forces $\pd_G(G)\leqslant a$. Obviously, $\pd_G(G)\geqslant 0$ by $0\neq G\in\mathscr{T}^{b}_G\cap \mathscr{T}^{\leqslant 0}_G$. Since $G$ generates $\mathscr{T}^c$, $\pd_G(X)<\infty$ for any $X\in\mathscr{T}^c$. So,
${\rm fpd}(\mathscr{T}, G)=\sup\{\pd_G(X)\mid X\in\mathscr{T}^{c}\cap\mathscr{T}^{\geqslant 0}_G\}$.
As $G[t]\in\mathscr{T}^{c}\cap\mathscr{T}^{\geqslant 0}_G$ for some integer $t$, we obtain ${\rm fpd}(\mathscr{T}, G)\geqslant \pd_G(G[t])=\pd_G(G)+t\geqslant t>-\infty$.

$(1)$ Note that $\mathscr{T}^c\cap\mathscr{T}^{\geqslant 0}_G=\mathscr{T}^c\cap G(-\infty, -1]^{\bot}$, and that $\pd_G(X)\leqslant\pd_G(G)$ for any $X\in \langle G\rangle ^{[0,\infty)}$, due to $\pd_G(G[i])=\pd_G(G)+i\leqslant \pd_G(G)$ for any $i\leqslant 0$. If $\mathscr{T}^c\cap G(-\infty, -1]^{\bot}\subseteq\langle G\rangle ^{[0,\infty)}[n]$ for some $n\geqslant 0$, then $\pd_G(X)\leqslant n+a$ for any $X\in \mathscr{T}^c\cap\mathscr{T}^{\geqslant 0}_G$.
Thus ${\rm fpd}(\mathscr{T}, G)\leqslant \fd(\mathscr{T}^c, G)+a$.

$(2)$ If ${\rm fpd}(\mathscr{T}, G)=\infty$, then the inequality in $(2)$ holds trivially. So, we assume
$m:={\rm fpd}(\mathscr{T}, G)<\infty$. Let $X\in \mathscr{T}^c\cap\mathscr{T}^{\geqslant 0}_G$. Then $\pd_G(X)\leqslant m$. Since $G\in\mathscr{T}^{b}_G\cap \mathscr{T}^{\leqslant 0}_G$, $\Hom_{\mathscr{T}^c}(X, G[j])=0$ for $j\geqslant m+1$. Now, let $\Phi:(\mathscr{T}^c)\opp\to \mathscr{U}^c$ be a triangle equivalence and $N:=\Phi(G)$. Then $N$ is a classical generator of $\mathscr{U}^c$, and $\Hom_{\mathscr{U}^c}(N, \Phi(X)[j])=0$ for $j\geqslant m+1$, that is, $\Phi(X)\in N[m+1, \infty)^{\bot}$.
Suppose that $\mathscr{U}$ has a compact generator $H$ with
$H[b, \infty)^{\bot}\subseteq\mathscr{U}^{\leqslant 0}_H$ for some integer $b$. One one hand, since $H$ and $N$ generate equivalent $t$-structures on $\mathscr{U}$, there is a nonnegative integer $e$ with
$\mathscr{U}^{\leqslant -e}_H\subseteq \mathscr{U}^{\leqslant 0}_N$. On the other hand, since $\langle H \rangle=\mathscr{U}^c=\langle N \rangle$, there is an integer $c\geqslant -m$ with
$N[1-c,\infty)^{\bot}\subseteq H[b-e,\infty)^{\bot}$.
Thus $$N[1-c,\infty)^{\bot}\subseteq H[b, \infty)^{\bot}[e]\subseteq\mathscr{U}^{\leqslant 0}_H[e]=\mathscr{U}^{\leqslant -e}_H\subseteq\mathscr{U}^{\leqslant 0}_N.$$
Let $d:=m+c$. Then
$N[m+1, \infty)^{\bot}=\big(N[1-c, \infty)^{\bot}\big)[-d]\subseteq \mathscr{U}^{\leqslant 0}_N[-d].$
From $\Phi(X)\in N[m+1, \infty)^{\bot}$ and $X\in\mathscr{T}^c$, we see that $\Phi(X)\in \mathscr{U}^c\cap\big(\mathscr{U}^{\leqslant 0}_N[-d]\big)$, and therefore $\Phi(X[-d])\in \mathscr{U}^c\cap\mathscr{U}^{\leqslant 0}_N$. Moroever, by Lemma \ref{Compact generator}(1), $\mathscr{U}^c\cap\mathscr{U}^{\leqslant 0}_N=\langle N\rangle ^{(-\infty,0]}$. Hence $\Phi(X[-d])\in\langle N\rangle ^{(-\infty,0]}$. Since $\Phi$ is a triangle equivalence,  we obtain $X[-d]\in \langle G\rangle ^{[0, \infty)}$. Consequently, $\mathscr{T}^c\cap G(-\infty, -1]^{\bot}=\mathscr{T}^c\cap\mathscr{T}^{\geqslant 0}_G\subseteq \langle G\rangle ^{[0, \infty)}[d]$.
Thus $\fd(\mathscr{T}^c, G)\leqslant d$.

$(3)$ By the equivalence $\Phi$, $\Hom_\mathscr{T}(G, G[i])\simeq\Hom_\mathscr{U}(N, N[i])$ for $i\in\mathbb{Z}$. Suppose $\Hom_\mathscr{T}(G, G[i])=0$ for $i>0$. Then $a=0$ in $(1)$ and $\Hom_\mathscr{U}(N, N[i])=0$ for $i>0$. Since $G$ generates $\mathscr{T}^c$, $N$ generates $\mathscr{U}^c$. It follows that $N$ is  a compact generator of $\mathscr{U}$. By \cite[Chapter III, Proposition 2.8]{BI},
$\mathscr{U}^{\leqslant 0}_N=N[1, \infty)^{\bot}$. So, in the proof of $(2)$, we can take $H:=N$, $b:=1$, $e:=0$ and $c:=\max\{0, -m\}$. Now, $(3)$ is a combination of $(1)$ and $(2)$.
\end{proof}

Now, we apply Proposition \ref{Nonpositive}(3) to derived categories of nonpositive DG rings.

Let $R:=\bigoplus_{i\in\mathbb{Z}}R^i$ be a DG ring. For a left DG $R$-module $X$ and for each integer $n$, the $n$-th cohomology group of $X$ is denoted by $H^n(X)$. We say that $R$ is \emph{nonpositive} if $R^i=0$ for $i>0$; \emph{bounded} if $H^i(R)=0$ for almost all $i$; \emph{left noetherian} if $H^0(R)$ is left noetherian and $H^i(R)\in H^0(R)\modcat$ for $i\in\mathbb{Z}$. Each ordinary ring can be regarded as a bounded and nonpositive DG ring concentrated in degree $0$.

Note that the (unbounded) derived category $\D{R}$ of left DG $R$-modules is a compactly generated triangulated category with ${_R}R$ as a compact generator. Thus both ${\rm FPD}(\D{R}, R)$ and ${\rm fpd}(\D{R}, R)$ are defined. They have connections with the following definition. Recall that $\mathscr{D}^b_f(R)$ denotes the full subcategory of $\D{R}$ consisting of objects $X$ with $H^i(X)=0$ whenever $|i|\gg 0$, and $H^i(X)\in H^0(R)\modcat$ for all $i\in\mathbb{Z}$ (see Section \ref{EXCT}).

\begin{Def} {\rm \cite[Definition 7.1]{BSSW}} Let $R$ be a left noetherian and nonpositive DG ring.

The \emph{big finitistic dimension} of $R$, denoted by ${\rm FPD}(R)$, is defined to be the supremum of the projective dimensions of left DG $R$-modules $X$ which satisfies $\pd_R(X)<\infty$ and $H^i(X)=0$ whenever $i<0$ and $i\gg 0$.

The \emph{finitistic dimension} of $R$, denoted by ${\rm fpd}(R)$, is defined to be the supremum of the projective dimensions of left DG $R$-modules $X\in\mathscr{D}^b_f(R)$ which satisfies $\pd_R(X)<\infty$ and $H^i(X)=0$ for $i<0$.
\end{Def}

By definition, ${\rm fpd}(\D{R}, R)\leqslant{\rm fpd}(R)\leqslant{\rm FPD}(R)={\rm FPD}(\D{R}, R).$
A further relationship between ${\rm fpd}(R)$ and $\fd(\mathscr{D}(R)^c, R)$ is given as follows.

\begin{Koro}\label{nonpositive DG ring}
Let $R$ be a left noetherian, nonpositive and bounded DG ring.

$(1)$ If ${\rm fpd}(R)<\infty$, then $\fd(\mathscr{D}(R)^c, R)\leqslant \max\{0, {\rm fpd}(R)\}$. In this case, $\fd(\mathscr{D}(R)^c)<\infty$.

$(2)$ If $R$ is commutative and $\dim(H^0(R))<\infty$, then $\fd(\mathscr{D}(R)^c, R)\leqslant \dim(H^0(R))$.
\end{Koro}

\begin{proof}
Note that $\Hom_{\D{R}}(R[i], R)\simeq H^{-i}(R)$ for $i\in\mathbb{Z}$. Since $R$ is nonpositive and bounded, $H^{-i}(R)=0$ for $i\gg 0$ or $i<0$. Clearly, the derived functor $\rHom_R(-,R): \D{R}\to \D{R\opp}$ restricts to an equivalence $({\D{R}}^c)\opp\to {\D{R\opp}}^c$ of triangulated categories.
Now, we apply Proposition \ref{Nonpositive}(3) to the pair $(\D{R}, {_R}R)$ and obtain
$\fd({\D{R}}^c,\, R)\leqslant\max\{0,\, {\rm fpd}(\D{R},\, R)\}\leqslant \max\{0, {\rm fpd}(R)\}$.
Then $(1)$ holds by Corollary \ref{Some-Any}. Recall that ${\rm fpd}(R)\leqslant {\rm FPD}(R)$. Further, if $R$ is commutative, then ${\rm FPD}(R)\leqslant\dim(H^0(R))$, due to \cite[Theorem C]{BSSW}. Thus $(2)$ holds.
\end{proof}

Note that each pretriangulated DG category can be converted to a stable $\infty$-category, for example, by taking differential graded nerve (see \cite[Section 1.3.1]{JL}), such that both categories have equivalent homotopy categories. In this way the derived category of a DG ring (i.e. a DG category with one object) can be realized as the homotopy category of an $\mathbb{E}_1$-ring spectrum. Now, a combination of Corollaries \ref{LCR} and \ref{nonpositive DG ring}(1) yields the following result.

\begin{Koro}\label{b6-dg}
Let $R$ be a left noetherian, nonpositive and bounded DG ring. Suppose ${\rm fpd}(R\opp)<\infty$. Then:

$(1)$ The category $\mathscr{D}(R)^c$ has a bounded $t$-structure if and only if
$\mathscr{D}(R)^c=\mathscr{D}^b_f(R)$.

$(2)$ All bounded $t$-structures on any triangulated category between $\mathscr{D}(R)^c$ and $\mathscr{D}^b_f(R)$ are equivalent.
\end{Koro}

Finally, we recall a different notion of finitistic dimensions of triangulated categories introduced by Krause
(see \cite{Kr2}), and compare it with Definition \ref{Def-FD}.

Let $\mathscr{S}$ be a triangulated category. For a pair $(X,Y)$ of objects in $\mathscr{S}$, let
$$\Lambda_{X,Y}:=\{n\in\mathbb{Z}\mid \Hom_\mathscr{S}(X, Y[n])\neq 0\}\;\;\mbox{and}\;\;h(X,Y):=\sup\{|i-j|\mid i, j\in \Lambda_{X,Y}\}+1.$$
We understand $h(X,Y)=-\infty$ if $\Lambda_{X,Y}=\emptyset$.
For $X\in\mathscr{S}$ and $m\geqslant 0$, let
${\rm amp}(X):=\sup\{|n|\mid n\in \Lambda_{X, X}\}$ and $
{\rm hom}^m(X):=\{Y\in\mathscr{S}\mid h(X,Y)\leqslant m\}.$
We say that $X$ is \emph{homologically finite} if $h(X,Y)<\infty$ for any $Y\in\mathscr{S}$; equivalently, the set $\Lambda_{X,Y}$ is finite for any $Y\in\mathscr{S}$.

\begin{Def}{\rm \cite[Definition 3]{Kr2}}\label{K-FD}
An object $X$ of the category $\mathscr{S}$ is called a
\emph{finitistic generator} of $\mathscr{S}$ if $X$ is homologically finite and
${\rm hom}^m(X)\subseteq \langle X\rangle_m$ for all $m\geq 0$, where $\langle X\rangle_0:=0$.

The \emph{finitistic dimension} of $\mathscr{S}$ is defined as
${\rm fin.dim}(\mathscr{S}):=\inf\{{\rm amp}(X)\mid X\; \mbox{is a finitistic generator of}\; \mathscr{S}\}.$
If no finitistic generator of $\mathscr{S}$ exists, then ${\rm fin.dim}(\mathscr{S})=\infty$.
\end{Def}
By \cite[Remark 4]{Kr2}, a necessary condition for the existence of a finitistic generator of $\mathscr{S}$ is that \emph{all objects of $\mathscr{S}$ are homologically finite}. This means that our definition of finitistic dimension (see Definition \ref{Def-FD}) is \emph{different} from Krause's definition. For example, the category $\mathscr{D}^{b}(R\text{-}{\rm mod})$ for an Artin algebra $R$ of infinite global dimension does not satisfy this necessary condition, but has finitistic dimension zero in our sense (see Corollary \ref{Artin algebra}). But Krause's definition and our definition still have some commonalities:

Each finitistic generator must be a classical generator. In our discussions, we also prefer the finitistic dimension of a triangulated category at a classical generator (see Corollary \ref{Some-Any}).
Further, the existence of a strong generator implies the one of a finitistic generator provided that all objects of the triangulated category are homologically finite (see \cite[Remark 4]{Kr2}). This can be compared with Proposition \ref{ST}. Finally, by \cite[Theorem 6]{Kr2}, for an ordinary ring $R$, $\fd(R)<\infty$ if and only if ${\rm fin.dim}(\Kb{\prj{R}})<\infty$ in the sense of Definition \ref{K-FD}. The same statement in our sense also holds true (see Lemma \ref{Finite}(5)).

\end{appendices}

{\footnotesize
\medskip

Rudradip Biswas,

Mathematics Institute, Zeeman Building, University of Warwick, Coventry CV4 7AL, United Kingdom.

{\tt Email: Rudradip.Biswas@warwick.ac.uk (R.Biswas)}

\smallskip
Hongxing Chen,

School of Mathematical Sciences  \&  Academy for Multidisciplinary Studies, Capital Normal University, 100048 Beijing,

P. R. China

{\tt Email: chenhx@cnu.edu.cn (H.X.Chen)}

\smallskip

Chris J. Parker,

Fakult\"at f\"ur Mathematik, Universit\"at Bielefeld, 33501, Bielefeld, Germany

{\tt Email: cparker@math.uni-bielefeld.de (C.Parker)}

\smallskip

Kabeer Manali Rahul,

Center for Mathematics and its Application, Mathematical Science Institute, Building 145,
The Australian National

University, Canberra, ACT 2601, Australia

{\tt Email: kabeer.manalirahul@anu.edu.au (K. Manali Rahul)}

\smallskip

Junhua Zheng,

School of Mathematical Sciences, Capital Normal University, 100048
Beijing, P. R. China;

Institute of Algebra and Number Theory,
University of Stuttgart, Pfaffenwaldring 57, 70569 Stuttgart, Germany

{\tt Email: zheng.junhua@mathematik.uni-stuttgart.de (J.H.Zheng)}
}

\end{document}